\documentclass[a4paper,english,10pt]{article}
\usepackage[left=.9in, right=.9in, bottom=1.1in, top=.9in,dvips]{geometry}
\makeatletter

\AtBeginDocument{\setlength\abovedisplayskip{5pt plus 2pt minus 1pt}
\setlength\belowdisplayskip{5pt plus 2pt minus 1pt}}
\usepackage[utf8]{inputenc}
\usepackage{lmodern}
\usepackage{indentfirst}
\usepackage[affil-it]{authblk}
\usepackage{rotating}
\usepackage{tocloft}
\usepackage{titlesec}
 \titleformat{\subparagraph}[hang]{\normalfont}{\thesubparagraph}{0pt}{\underline}
\usepackage{titletoc}
\usepackage{titlefoot}
\usepackage[textsize=scriptsize]{todonotes} 
\makeatletter
\renewcommand{\todo}[2][]{\tikzexternaldisable\@todo[#1]{#2}\tikzexternalenable}
\makeatother

\usepackage{standalone}
\usepackage{subfiles}

\usepackage[hidelinks,bookmarksopen,bookmarksdepth=2]{hyperref}
\usepackage{enumitem}
\setlist[enumerate]{itemsep=0pt,parsep=2.0pt plus 1.0 pt minus 0.5pt,ref=\alph*, label={(\alph*)}, itemsep=0em,topsep=4.0pt plus 2.0 pt minus 1.0pt } 
\setlist[itemize]{itemsep=2.0pt plus 1.0 pt minus 0.5pt, topsep=4.0pt plus 2.0 pt minus 1.0pt}
\usepackage{array}
\newcolumntype{M}[1]{>{\centering\arraybackslash}m{#1}}
\usepackage{multirow}
\usepackage{tikz}
\usetikzlibrary{external}
\usepackage{calc}
\usepackage{graphicx}
\usepackage{subcaption}
\usepackage[labelfont=bf]{caption}
\usepackage{epsfig}
\usepackage{color, soul}

\usepackage{contour}
\usepackage[normalem]{ulem}

\contourlength{1pt}

\usepackage{marginnote}

\usepackage{color,soulutf8}

\usepackage{amsmath, amsthm,empheq}
\usepackage{cases}
\usepackage[capitalize]{cleveref}
\usepackage{amssymb}
\usepackage{bbm}
\usepackage{bm}
\usepackage{marvosym}
\usepackage{mathtools}
\usepackage{mathrsfs}
\usepackage{upgreek}
\usepackage{accents}
\usepackage{textcomp}
\usepackage{slashed} %for gradient on spheres
\usepackage{centernot}
\usepackage{cancel}
\usepackage{harmony} %for musical symbols
\usepackage[nointegrals]{wasysym} %for clock
\wasyfamily
\DeclareFontShape{U}{wasy}{b}{n}{ <-10> ssub * wasy/m/n
 <10> <10.95> <12> <14.4> <17.28> <20.74> <24.88>wasyb10 }{}
\DeclareFontShape{U}{wasy}{b}{it}{ <-10> ssub * wasy/m/n
 <10> <10.95> <12> <14.4> <17.28> <20.74> <24.88>wasyb10 }{}
\DeclareMathAlphabet\mathbfcal{OMS}{cmsy}{b}{n}
\newcommand\numberthis{\addtocounter{equation}{1}\tag{\theequation}}
\renewcommand{\Re}{\operatorname{Re}}
\renewcommand{\Im}{\operatorname{Im}}

\allowdisplaybreaks

\numberwithin{equation}{section}

\theoremstyle{plain}

\newtheorem{arabictheorem}{Theorem}
\newtheorem*{maintheorem}{Main Theorem}

\newtheorem{theorem}{Theorem}
\newtheorem*{theorem*}{Theorem}
\newtheorem{conjecture}[theorem]{Conjecture}
\newtheorem*{conjecture*}{Conjecture}

\newtheorem*{corollary*}{Corollary}
\newtheorem{proposition}[theorem]{Proposition}
\newtheorem{lemma}[theorem]{Lemma}
\numberwithin{theorem}{section} 

\expandafter\let\expandafter\oldproof\csname\string\proof\endcsname
\let\oldendproof\endproof
\renewenvironment{proof}[1][\proofname]{%
  \oldproof[\upshape \bfseries #1]%
}{\oldendproof}
\theoremstyle{definition}
\newtheorem{definition}[theorem]{Definition}
\newtheorem{remark}[theorem]{Remark}
\newtheorem{hypotheses}[theorem]{Hypotheses}

\newtheorem*{remark*}{Remark}

\newcommand{\mb}[1]{\mathbb{#1}}

\newcommand{\mc}[1]{\mathcal{#1}}
\newcommand{\mr}[1]{\mathrm{#1}}

\newcommand{\p}{\partial}
\newcommand{\lp}{\left}
\newcommand{\rp}{\right}
\newcommand\supp{\mathrm{supp}}

\def \a{\alpha}
\def \R {\mathbb{R}}

\def \C{\mathbb{C}}

\def \N{\mathbb{N}}
\def \D{\textup{D}}

\def \e{\varepsilon}
\def \d{\,\textup{d}}

\def \n{\nabla}
\def \exc{\backslash}
\def \p{\partial}

\def \mc{\mathcal}
\def \mb{\mathbb}

\def \mf{\mathfrak}

\def \wstar{\overset{\ast}{\rightharpoonup}}
\def \w{\rightharpoonup}
\def \wH{\xrightharpoonup{\textup{H}}}
\def \supp{\textup{supp}\,}
\def \tp{\textup}

\newlength{\dhatheight}

\usepackage{combelow}

\usepackage{scalerel,stackengine}
\stackMath
\newcommand\reallywidehat[1]{%
\savestack{\tmpbox}{\stretchto{%
  \scaleto{%
    \scalerel*[\widthof{\ensuremath{#1}}]{\kern-.6pt\bigwedge\kern-.6pt}%
    {\rule[-\textheight/2]{1ex}{\textheight}}%WIDTH-LIMITED BIG WEDGE
  }{\textheight}% 
}{0.5ex}}%
\stackon[1pt]{#1}{\tmpbox}%
}
\newlength{\bibitemsep}
\newlength{\bibparskip}
\let\oldthebibliography\thebibliography
\renewcommand\thebibliography[1]{%
  \oldthebibliography{#1}%
  \setlength{\parskip}{\bibitemsep}%
  \setlength{\itemsep}{\bibparskip}%
}

\usepackage{totcount}

\newtotcounter{citnum} %From the package documentation

\begin{document}
 \title{\LARGE \textbf{Oscillations in wave map systems and \\ homogenization of the Einstein equations in symmetry}}

%A short proof of Burnett's conjecture in General Relativity under U(1) symmetry with an elliptic gauge

\author[1]{{\large Andr\'e Guerra}}
\author[2]{{\large  Rita \mbox{Teixeira da Costa}\vspace{0.1cm}}}

\affil[1]{\footnotesize University of Oxford, Andrew Wiles Building, Woodstock Rd, Oxford OX2 6GG, UK \protect \\[-1mm]
{\small\tt{andre.guerra@maths.ox.ac.uk}}\vspace{5pt}\ }

\affil[2]{\footnotesize 
University of Cambridge, Center for Mathematical Sciences, Wilberforce Rd, Cambridge CB3 0WA, UK
\protect \\[-1mm]
{\small\tt{rita.t.costa@dpmms.cam.ac.uk}}\
}

%\date{\today}
\date{}

\maketitle
\vspace*{-.7cm}

\begin{abstract}
In 1989, Burnett conjectured that, under appropriate assumptions, the limit of highly oscillatory solutions to the Einstein vacuum equations is a solution of the Einstein--massless Vlasov system. In a recent breakthrough, Huneau--Luk (arXiv:1907.10743) gave a proof of the conjecture in $U(1)$-symmetry and elliptic gauge. They also require control on up to fourth order derivatives of the metric components. In this paper, we give a streamlined proof of a stronger result and, in the spirit of Burnett's original conjecture, we remove the need for control on higher derivatives. Our methods also apply to general wave map equations.
\end{abstract}

%\unmarkedfntext{
%\hspace{-.8cm}
%%\emph{2020 Mathematics Subject Classification:} \red{codes}\\ 
%%\emph{Keywords.} \red{fill me}\\
%\emph{Acknowledgments.} The authors were supported by the EPSRC, grants [EP/L015811/1] and  [EP/L016516/1] respectively.
%}

\vspace{0.5cm}
\tableofcontents

\section{Introduction}

%In General Relativity, spacetime is represented by a 4-dimensional Lorentzian manifold $(\mc{M},g)$ and some matter fields $f$, which may be scalar or tensorial. The triple $(\mc{M},g,f)$ satisfies an Einstein--matter model system, comprised of the Einstein equations
%\begin{equation}
%\mathbf{Ric}(\bm g)=8\pi\mathbb{T}(f)\,,\label{eq:Einstein-equation-1}
%\end{equation}
%where $\mathbb{T}$ denotes the (trace-reversed) energy-momentum tensor associated to $f$, and the matter field equation
%\begin{equation}
%\mathbf{P}_{\bm{g}}(f)=0\,.\label{eq:matter-equation}
%\end{equation}
%Here, $\mathbf{P}_{\bm{g}}$ is a differential operator which is prescribed from physical considerations but which is moreover heavily constrained by the geometry of $(\mc{M},g)$: from the Bianchi identities, the implication
%\begin{equation}
%\text{\eqref{eq:Einstein-equation-1}--\eqref{eq:matter-equation}}\implies \bm{\nabla}^a \lp(\mathbb{T}_{ab}(f)-\frac12 \bm g_{ab}\,\mr{tr}_{\bm g}\mathbb{T}(f)\rp)=0\,,\label{eq:Einstein-equation-2}
%\end{equation}
%where $\bm \nabla$ is the covariant derivative respect to the metric $\bm g$, must hold. Thus, the Einstein--matter \eqref{eq:Einstein-equation-1}--\eqref{eq:matter-equation} is a
%
% which solves the Einstein equations with respect to some suitable matter fields, see already \eqref{eq:Einstein-equation-1}--\eqref{eq:Einstein-equation-2}.

In General Relativity, spacetime is represented by a 4-dimensional Lorentzian manifold  which solves the Einstein equations with respect to some suitable matter fields, see already \eqref{eq:Einstein-equation-1}. In describing complex gravitational systems, it is often useful to take a coarse-grained view and study \textit{effective models} instead \cite{Ellis1983}. To date, the $\mathrm{\Lambda}$CDM model in cosmology, consisting of an FLRW spacetime with empirically determined parameters, is the most successful effective model for our universe.  However, it is not known how to derive these effective large scale models as \textit{limits} of the Einstein equations at smaller scales, except for very simple toy problems, see e.g.\ \cite{Green2013}. In fact, it is not even clear what the correct notion of \textit{limit} should be \cite{Green2011,Buchert2015}!

%The difficulty in rigorously moving from fine-grained to coarse-grained models is in part due to the intricate coupling between matter and geometry in General Relativity, as can be seen already by the following problem. 
In this paper we consider the simpler problem of determining the weak closure of the vacuum Einstein equations, i.e.\ the Einstein equations in the absence of matter.
Fix a manifold $\bm{\mc{M}}$, and consider a sequence $(\bm g_\e)_\e$ of vacuum Lorentzian metrics on $\bm{\mc M}$:
\begin{equation}
\mathbf{Ric}(\bm g_\e)=0\,. \label{eq:Einstein-equation-e-firstappearance}
\end{equation} 
If ${\bm g_\e}$ converges strongly to a Lorentzian metric ${\bm g}$ in $C^0_\tp{loc}\cap W^{1,2}_\tp{loc}$ as $\e\to 0$, by the structure of the Ricci tensor, it is easy to see we can pass to the limit in \eqref{eq:Einstein-equation-e-firstappearance}; our effective model is then simply vacuum, i.e.\ $\mathbf{Ric}(\bm{g})=0$. On the other hand, if the convergence is only weak, then $(\bm{\mc M},{\bm g})$ is no longer necessarily Ricci-flat, so the effective model is non-trivial. 
From the Einstein equations
\begin{equation}
\mathbf{Ric}(\bm g)=8\pi\mathbb{T}\,,\label{eq:Einstein-equation-1}
\end{equation}
where $\mathbb{T}$ denotes a (trace-reversed) energy-momentum tensor, we are tempted to identify the Ricci tensor obtained in the limit as \textit{matter}. However, in order for $\mb T$ to correspond to a true Einstein matter model, we must supplement \eqref{eq:Einstein-equation-1} with a matter field equation coupled to the geometry of $(\bm{\mc M}, \bm g)$ in order to get a closed system, see already Conjecture \ref{conj:Burnett} below for an example.
%An important Einstein matter model that will play a crucial role in this paper is the so-called Einstein massless-Vlasov system, which models the behavior of a gas of massless particles. In this model, one takes
%$$\mb T=\int \dots$$
%where the particle density $f$ is coupled to the metric $g$ through the Vlasov equation
%$$\dots.$$

%However, in \textit{physical} matter, we have that $\mathbb{T}=\mathbb{T}(f)$ arises from matter fields $f$, which themselves satisfy a field equation
%\begin{equation}
%\mathbf{P}_{\bm{g}}(f)=0\,.\label{eq:matter-equation}
%\end{equation}
%Here, $\mathbf{P}_{\bm{g}}$ is a differential operator prescribed from physical considerations but which is heavily constrained by the geometry of $(\bm{\mc{M}},\bm{g})$: by the Bianchi identities, the implication
%\begin{equation}
%\text{\eqref{eq:Einstein-equation-1}--\eqref{eq:matter-equation}}\implies \bm{\nabla}^a \lp(\mathbb{T}_{ab}(f)-\frac12 \bm g_{ab}\,\mr{tr}_{\bm g}\mathbb{T}(f)\rp)=0\,,\label{eq:Einstein-equation-2}
%\end{equation}
%where $\bm \nabla$ is the covariant derivative respect to the metric $\bm g$, must hold. In other words, for $(\bm{\mc M},\bm g)$ to be a solution to the Einstein equations in General Relativity and thus be considered an effective model, we must both identify the underlying matter field $f$ by computing $\mr{Ric}(\bm g)$ in \eqref{eq:Einstein-equation-1} \textit{and} show that the matter satisfies a field equation \eqref{eq:matter-equation} compatible with \eqref{eq:Einstein-equation-2}. 

As we have just seen, the effective model depends crucially on the convergence assumptions for the sequence $\bm g_\e$. In this paper we are concerned with the so-called \textit{high-frequency limit}, in which small amplitude but high-frequency waves propagate on a fixed background. The high-frequency limit was studied in the physics literature \cite{Brill1964,Choquet-Bruhat1969, Green2013, Hogan1993,Isaacson1968a,Isaacson1968, MacCallum1973,Szybka2014,Szybka2016}
 and we rely here on Burnett's \cite{Burnett1989} and Green--Wald's framework \cite{Green2011}. More precisely, we assume that there is a smooth Lorentzian metric $\bm g$ such that, for each compact set $K\subset \bm{\mc M}$ with a fixed coordinate chart, there is a sequence $\lambda_\e\searrow 0$ such that
\begin{equation}
\label{eq:Burnett-conditions}
\|\p^k (\bm g_\e-\bm g)\|_{L^\infty(K)} \leq C(K)\lambda_\e^{1-k}, \qquad \text{for } k=0,1,2.
\end{equation}
Under these assumptions, Burnett \cite{Burnett1989} conjectured that the effective model is Einstein--massless Vlasov. We borrow a more precise formulation of Burnett's conjecture from the recent work of Huneau--Luk \cite{Huneau2019}:

\begin{conjecture}[Burnett]
\label{conj:Burnett}
Let $(\bm g_\e)_{\e>0}$ and $\bm g$ be smooth Lorentzian metrics on $\bm{\mc M}$ satisfying \eqref{eq:Burnett-conditions}. There is a finite, non-negative Radon measure measure $\bm\mu$ in $\bm{T^* \mc M}$ such that $(\bm{\mc M}, \bm g, \bm \mu)$ is a solution to the Einstein--massless Vlasov system, that is:
\begin{enumerate}
\item the Einstein equation \eqref{eq:Einstein-equation-1} holds, where $\mb T$ is defined by its action on a test vector field $Y$ as
$$\int_{\bm{\mc M}}\mb T(Y,Y)\d\tp{Vol}_{\bm{g}} = \frac{1}{8\pi} \int_{\bm{T^*\mc M}} \xi_a \xi_b Y^a Y^b \d \bm{\mu}; $$
\item $\bm \mu$ is solves the massless Vlasov equation with respect to $\bm g$:
 \begin{enumerate}[ref={\alph{enumi}\textsubscript{\arabic*}},label={\normalfont(\alph{enumi}\textsubscript{\arabic*})}, itemsep=0em]
\item $\bm \mu$ is supported on the zero mass shell $\{(x,\xi)\in \bm{T^*\mc M}: \bm g^{ab}(x)\xi_a\xi_b=0\};$
\item the Vlasov equation holds distributionally: for any $a\in C^\infty_c(\bm{T^* \mc M}\exc\{0\})$,
$$\int_{\bm{T^* \mc M}} \left(\bm g^{ab}\xi_b \p_{x^a} a -\frac 1 2 \p_{x^c} \bm g^{a b} \xi_a\xi_b \p_{\xi_c} a\right)\d \bm{\mu}=0\,.$$
\end{enumerate}
\end{enumerate}
\end{conjecture}

We refer the reader to \cite{Andreasson2011,Rendall1997} for a general introduction to the Einstein--Vlasov model.
According to Conjecture \ref{conj:Burnett}, lack of compactness in the Ricci tensor manifests itself as massless matter, which is propagated along the null directions of spacetime without collisions. In fact, Burnett went further and conjectured that, conversely, all Einstein--massless Vlasov systems can be realized as the weak limit of a sequence of vacuum spacetimes satisfying \eqref{eq:Burnett-conditions}. We refer the reader to Huneau--Luk \cite{Huneau2018a,Huneau2018} for progress in that direction, as well as Touati \cite{Touati2021} in a lower regularity setting.

Let us emphasize that although \eqref{eq:Burnett-conditions} are indeed weak convergence assumptions they forbid the occurrence of concentrations. For a setting where  concentrations are allowed, and without symmetry assumptions, a complete characterization of the weak closure of the Einstein vacuum equations was recently obtained by Luk--Rodnianski \cite{Luk2020}; see also \cite{LeFloch2019, LeFloch2020} in $\mb T^2$-symmetry.

The purpose of this paper is to prove Conjecture \ref{conj:Burnett} under symmetry and gauge assumptions:

\begin{maintheorem}\label{thm:main}
Conjecture \ref{conj:Burnett} is true when all the metrics have $U(1)$-symmetry and can be put in elliptic gauge with respect to a fixed chart on $\bm{\mc M}$.
\end{maintheorem}

See Theorem~\ref{thm:luk-huneau} below for a more precise statement. We recall that a manifold $\bm{\mc M}$ has $U(1)$-symmetry if it has a one-dimensional spacelike group of isometries; the elliptic gauge conditions are more involved, see Section~\ref{sec:quasilinear} and Appendix~\ref{sec:appendix}.

A version of Theorem \ref{thm:main} where \eqref{eq:Burnett-conditions} was assumed up to $k=4$ was proved earlier by Huneau--Luk \cite{Huneau2019}. It is desirable to make assumptions only up to $k=2$ derivatives in \eqref{eq:Burnett-conditions}, since the Einstein equations are second order. Besides this improvement, our proof is perhaps simpler, while remaining completely self-contained. On the way to Theorem \ref{thm:main} we also obtain results of independent interest for wave maps, see Section \ref{sec:semilinear}.
Our proof consists of three steps:
\begin{enumerate}[label=\arabic*., itemsep=0pt]
\item \textbf{The quasilinear terms:} the $U(1)$-symmetry and gauge assumptions can be used to show that oscillations in the quasilinear terms in \eqref{eq:Einstein-equation-e-firstappearance} do not contribute to the effective model in any way. Indeed, the massless Vlasov matter is produced by the oscillations in a \textit{semilinear} wave map system. 
\item \textbf{The semilinear terms:} to understand the oscillations in wave map systems, we rely on essentially classical bilinear and trilinear compensated compactness results due to Murat and Tartar \cite{Murat1978,Tartar2005}, of which we give simple proofs in Section \ref{sec:linear-wave-non-oscillating}. Through these results, and thanks to the \textit{Lagrangian structure} of wave maps, we find that, surprisingly, the heart of the problem is to understand oscillation effects in a \textit{linear} scalar wave equation with respect to an oscillating metric.
\item \textbf{The linear terms:} to study linear scalar wave equations with respect to oscillating metrics, our strategy is to take the metric oscillations as \textit{sources} for a wave equation with respect to the limit metric. It turns out that the metric oscillations contribute to the propagation of lack of compactness via a commutator: this is shown by a careful integration by parts argument relying on the parity of the Vlasov equation. We estimate this commutator through a fine frequency analysis in Fourier space which exploits simultaneously the gauge choice, the cancellations encoded in the commutator and some of the rate assumptions \eqref{eq:Burnett-conditions}.
\end{enumerate}
The remainder of the introduction discusses each of these steps in detail.

\subsection{The quasilinear terms}
\label{sec:quasilinear}

%This paper is concerned with a quasilinear Einstein--wave map system:
%\begin{align}
%\begin{dcases}
%\Box_g \psi +\frac12 e^{-4\psi}g^{-1}(\tp{d}\omega,\tp{d}\omega) =0\,,\\
%\Box_g \omega-4g^{-1}(\tp{d}\psi,\tp{d}\omega)=0\,,\\
%[\mr{Ric}(g)]_{\alpha\beta}=2\p_\alpha\psi\p_\beta\psi +\frac{1}{2}e^{-4\psi}\p_\alpha \omega\p_\beta\omega  \,,
%\end{dcases} \label{eq:Einstein-wave-map-firstappearance}
%\end{align}
%on a Lorentzian domain $(\mc{M},g)$. Let $(\psi_\e,\omega_\e,g_\e)_\e$ be a sequence of solutions of \eqref{eq:Einstein-wave-map-firstappearance} converging to some  $(\psi,\omega,g)$. Each $(\psi_\e,\omega_\e,g_\e)$ defines a vacuum solution, $({^{(4)}\mc{M}},{^{(4)}g_\e})$, to the Einstein equations under $U(1)$-symmetry, see Section~\ref{sec:intro-Burnett}, and whenever $(\psi_\e,\omega_\e,g_\e)\to (\psi,\omega,g)$ strongly in $C^0_\tp{loc}\cap W^{1,2}_\tp{loc}$, so does $(\psi,\omega,g)$.

We begin by describing in further detail the setup of the Main Theorem. Fix a manifold $\bm{\mc{M}}$ which can be trivialized along one direction, $\bm{\mc{M}}=\mc{M}\times \mathbb{R}$. Here, $\mc{M}$ is also a fixed manifold of trivial topology, i.e.\ $\mc{M}=(0,T)\times\mathbb{R}^{2}$ for some $T>0$. We take global coordinates $(t\equiv x^0,x^1,x^2)$ on $\mc{M}$, which we denote with greek indices, and coordinates  $(x^0,x^1,x^2,x^3)$ on $\bm{\mc{M}}$. Henceforth, all derivatives indicated by $\p$, as well as all Sobolev norms, are considered with respect to this fixed chart.

Now take a sequence of Lorentzian metrics $(\bm{g}_\e)_{\e>0}$ on $\bm{\mc{M}}$ of the form
\begin{equation} \label{eq:U1-metric}
\bm g_\e\equiv e^{-2\psi_\e}g_\e+e^{2\psi_\e}\lp(\tp{d}x^3+\mathfrak{A}_{\alpha,\e} \tp{d}x^\alpha\rp)^2\,,
\end{equation}
 where $g_\e$ are Lorentzian metrics on $\mc{M}$ and $\psi_\e$ and $\mathfrak{A}_\e\equiv \mathfrak A_{\a,\e} \d x^\a$ are, respectively, real-valued functions and 1-forms on $\mc{M}$. These conditions ensure that the vector field $\p_{x^3}$ generates a one-dimensional spacelike group of isometries on each $(\bm{\mc{M}},\bm g_\e)$, i.e.\ that these spacetimes are $U(1)$-symmetric. 
 
In order to prove the Main Theorem we first note that, if $\bm g_\e$ is bounded in $W^{1,\infty}_\tp{loc}$ and converges locally uniformly to some $\bm{g}\equiv \bm{g}_0$, then the weak limit $(\bm{\mc{M}},\bm g)$ also has $U(1)$-symmetry. Using the $U(1)$-symmetric metric ansatz \eqref{eq:U1-metric}, we can show that if $(\bm{g}_\e)_{\e>0}$ are vacuum, then $[\mathbf{Ric}(\bm{g}_\e)]_{\alpha 3}=0$ is a linear differential constraint, see \cite{Andersson2004} and  \cite[Chapter XVI.3]{Choquet-Bruhat2008} for details, which is therefore preserved in the limit:
\begin{align}
[\mathbf{Ric}(\bm{g}_\e)]_{\alpha 3}=0 \implies \tp{d}\mathfrak{A}_\e =  e^{-4\psi_\e}\star_{g_\e}\, \tp{d}\omega_\e \,, \quad \e\geq 0\,,  \label{eq:introducing-omega}
\end{align}
where $\omega_\e$ are functions on $\mc{M}$.  We are now ready to state our main result precisely:

%
%--- EXPLANATION; DO NOT DELETE
%Why? The components of the Ricci tensor  of $\bm g$ in the directions $(\alpha,3)$ satisfy
%\begin{align}
%[{\rm Ric}(\bm g)]_{\alpha 3}&=\frac{1}{2}e^{-\psi}\nabla^\beta\lp(e^{4\psi}\mathfrak{F}_{\alpha\beta}\rp)=-\frac{1}{2}e^{-\psi}\lp[\star_g\, \tp{d} \star_g(e^{4\psi}\mathfrak{F})\rp]_\alpha\,,
%\label{eq:Ricci-coefficient-closed}
%\end{align}
%where we use the shorthand notation $\mathfrak{F}_{\alpha\beta}\equiv (\tp{d}\mathfrak{A})_{\alpha\beta}=\p_\alpha\mathfrak{A}_\beta-\p_\beta\mathfrak{A}_\alpha$. Here $\nabla$ is the covariant derivative with respect to $g$. We refer the reader to \cite[Chapter XVI.3]{Choquet-Bruhat2008} and \cite{Andersson2004} for further details.
%
%Let us now \textit{assume} that
%$$[{\rm Ric}(\bm g)]_{\alpha 3}=0\,.$$
%From \eqref{eq:Ricci-coefficient-closed}, it follows  that $\star_g(e^{4\psi}\mathfrak{F})$ is a closed form and hence, by the Poincaré lemma, exact:
%\begin{equation}
%\star_g(e^{4\psi}\mathfrak{F})=-\tp{d}\omega\quad \iff\quad
%\mathfrak{F}=\tp{d}\mathfrak{A} =  e^{-4\psi}\star_g\, \tp{d}\omega \label{eq:U1-metric-A-condition}
%\end{equation}
%for a real valued scalar field $\omega$ on $\mc{M}$.
%By definition of $\Box_g$  and using the closedness of $\mathfrak{F}$, we derive 
%\begin{equation}\label{eq:omega-eq-deriv}
%\Box_g \omega=-\star_g\,\tp{d}\star_g\,\tp{d}\omega=-\star_g\,\tp{d}(e^{4\psi} \mathfrak{F}) =-\star_g \lp[4(\tp{d}\psi)\wedge (\star_g\,\tp{d}\omega) +e^{4\psi}\tp{d}\mathfrak{F}\rp]=4g^{-1}(\tp{d}\psi,\tp{d}\omega)\,.
%\end{equation}
%------
%

\begin{hypotheses} \label{hyp:Einstein} Let $(\bm{g}_\e)_{\e>0}\equiv (g_\e, \psi_\e, \omega_\e)_{\e>0}$ and $\bm{g}\equiv (g\equiv g_0,\psi,\omega)$ satisfy:
\begin{enumerate}
\item \label{it:hypmetrics-Einstein}  in the fixed chart we have introduced, the eigenvalues of $g_\e$ are uniformly bounded above and away from zero and $g$ is a smooth metric such that $g_\varepsilon\to g$ in $C^0_{\tp{loc}}$ as $\e\to 0$ and $g_\e$ is bounded in $W^{1,\infty}_\tp{loc}$; furthermore, for $\e\geq 0$,  $g_\e$ are in an \textit{elliptic gauge}, i.e.\ 
\begin{enumerate}[ref=(\alph{enumi}\textsubscript{\arabic*}),label={\normalfont(\alph{enumi}\textsubscript{\arabic*})}, topsep=0pt]
\item $g_\e$ has the form 
$$g_\e=-N^2_\e(\tp d x^0)^2+\tilde{g}_{ij,\e}(\tp d x^i+\beta^i_\e\tp d x^0)(\tp d x^j+\beta^j_\e \tp d x^0)\,,$$ where $N_\e$ and $\beta_\e$ are, respectively, functions and vectors on $\mc{M}$ and $\tilde{g}_\e$ is a Riemannian metric on $\mathbb{R}^2$ which we can, and do, take to be conformally flat;
\item $x^0$ hypersurfaces are maximal, i.e.\ they have zero mean curvature;
\end{enumerate}
\item $\psi_\varepsilon\to \psi$ in $C^0_\tp{loc}$, $\psi_\varepsilon\w \psi$ in $W^{1,4}_\tp{loc}$, and similarly replacing $\psi_\e,\psi$ with $\omega_\e,\omega$;
\item $\lVert g_\e^{\alpha\beta}-g^{\alpha\beta}\rVert_{L^\infty(K)}\left(\lVert \p^2 (\psi_\e-\psi)\rVert_{L^4(K)}+\lVert \p^2 (\omega_\e-\omega)\rVert_{L^4(K)}\right)\lesssim_K 1$ for every compact $K\subset \mc{M}$.
\label{it:hypsols-Einstein}
\end{enumerate}
\end{hypotheses}

We note that Hypotheses \ref{hyp:Einstein} are 
\textit{strictly weaker} than the high-frequency limit conditions \eqref{eq:Burnett-conditions}, when specialized to the $U(1)$-symmetric and elliptic gauge case.

\begin{arabictheorem} \label{thm:luk-huneau} 
Let $(\bm{g}_\e)_{\e>0}$ and $\bm{g}$ satisfy Hypotheses~\ref{hyp:Einstein} and assume that, for $\e>0$, $\bm{g}_\e$ solve \eqref{eq:Einstein-equation-e-firstappearance}. Then there is a non-negative Radon measure $\nu$ on $S^*\mc M$ such that $(\bm{\mc{M}},\bm g,\nu)$ is a radially averaged measure-valued solution of the restricted Einstein--Vlasov equations in $U(1)$-symmetry. More precisely:
\begin{enumerate}
\item {\normalfont \textbf{Limit equation:}} For every vector field $Y\in C^{\infty}_0(\mc{M})$, the tensor $\mathbf{Ric}(\bm g)$ satisfies
 \label{it:thm-cor-luk-huneau-T}
\begin{gather}\label{eq:thm-cor-luk-huneau-T}
\begin{gathered}
[\mathbf{Ric}(\bm g)]_{\alpha3}=0, 
\qquad
[\mathbf{Ric}(\bm g)]_{33}=0, \qquad
\int_{\mc{M}}[\mathbf{Ric}(\bm g)]_{\a\beta}Y^\a Y^\beta\d \mr{Vol}_{g}=
\int_{S^*\mc{M}} \xi_\a\xi_\beta Y^\a Y^\beta\d\nu\,.
\end{gathered}
\end{gather} 
\item {\normalfont \textbf{Vlasov equation:}}  \label{it:thm-cor-luk-huneau-Vlasov}  $(\mc M,g,\nu)$ is a radially averaged measure-valued solution of massless Vlasov:
  \begin{enumerate}[ref={\alph{enumi}\textsubscript{\arabic*}},label={\normalfont(\alph{enumi}\textsubscript{\arabic*})}, itemsep=0em]
\item {\normalfont Support property:} $\nu$ is supported on the zero mass shell of $g$, i.e.\ for all $\varphi\in C_0^\infty (\mc{M})$
   \label{it:thm-cor-luk-huneau-localization}
\begin{gather*}
\int_{S^*\mc M} \varphi(x)g^{\a\beta}\xi_\alpha\xi_\beta \d\nu =0\,.
\end{gather*}
\item {\normalfont Propagation property:} 
for all $ \tilde a\in C_0^\infty (S^*\mc{M})$, extended as a positively 1-homogeneous function to $T^*\mc M\exc\{0\}$, the measure $\nu$ satisfies 
\label{it:thm-cor-luk-huneau-propagation}
\begin{gather}
\int_{S^*\mc M} \lp[g^{\alpha\beta}\xi_\alpha\p_{x^\beta}\tilde a-\frac12\p_{x^\mu}g^{\alpha\beta}\xi_\alpha\xi_\beta\p_{\xi_\mu}\tilde a\rp] \tp d\nu 
= 0
\,.
\label{eq:thm-cor-luk-huneau-Vlasov}
\end{gather}
\end{enumerate}
\end{enumerate}
\end{arabictheorem}

We note that, as long as  $(\bm{\mc M},\bm g)$ is globally hyperbolic, $(\bm{\mc{M}},\bm g,\nu)$ naturally induces a non-radially averaged solution to the Einstein--massless Vlasov system, see \cite[Section 2]{Huneau2019}.

\begin{remark}[Beyond the vacuum case] Our methods allow for an extension of Theorem~\ref{thm:luk-huneau} to a case where $\bm{g}_\e$ are not vacuum but are sourced by a tensor $\mb T_\e$. To be precise, we require that the $(\alpha,3)$ components  of $\mb T_\e$ must vanish and that there is a smooth tensor $\mb T$ such that $\mb T_\e\to \mb T$ in $C^0_\tp{loc}$ and $\mb T_\e\w \mb T$ in $L^4_\tp{loc}$.  In that case, the analogue of \eqref{eq:thm-cor-luk-huneau-T} reads as
\begin{gather*}
[\mathbf{Ric}(\bm g)]_{\alpha3}=0, 
\qquad
[\mathbf{Ric}(\bm g)]_{33}=8\pi\,\mathbb{T}_{33}, \\
\int_{\mc{M}}[\mathbf{Ric}(\bm g)]_{\a\beta}Y^\a Y^\beta\d \mr{Vol}_{g}=
\int_{\mc{M}}8\pi\lp[ \mathbb{T}_{\alpha\beta}+\mathbb{T}_{33}e^{-2\psi}g_{\alpha\beta}\rp]Y^\alpha Y^\beta \d \tp{Vol}_{g}+
\int_{S^*\mc{M}}  \xi_\a\xi_\beta Y^\a Y^\beta\d\nu\,,
\end{gather*} 
and the Vlasov equation in \eqref{eq:thm-cor-luk-huneau-Vlasov} has a source term related to the failure of compactness in $\mb T_\e$.

\end{remark}

To understand the proof of Theorem~\ref{thm:luk-huneau}, let us begin by computing the curvature of the limit spacetime $(\bm{\mc{M}},\bm{g})$; we again use the notation $\bm g_0\equiv \bm g$. As we have seen above, this spacetime also has $U(1)$-symmetry, so our computations rely on the form of $U(1)$-metrics given in \eqref{eq:U1-metric}.

\vspace{-.66\baselineskip}
\paragraph{Curvature in the $U(1)$-symmetry directions.}  We have already seen that the vacuum condition passes to the limit in the $(\alpha,3)$ direction, motivating us to introduce functions $\omega_\e$, $\e\geq 0$, on $\mc{M}$ as in \eqref{eq:introducing-omega}. One can further show, see \cite{Andersson2004} and  \cite[Chapter XVI.3]{Choquet-Bruhat2008}, that  
\begin{align}
[\mathbf{Ric}(\bm{g}_\e)]_{\alpha 3}=0=\Box_{g_\e} \omega_\e-4g^{-1}_\e(\tp{d}\psi_\e,\tp{d}\omega_\e)\,, \qquad \tp{for all }\e\geq 0\,. \label{eq:Ricci-coefficient-alpha3}
\end{align}
In the $(3,3)$ direction, \eqref{eq:Einstein-equation-e-firstappearance} leads to a nonlinear wave equation, but the nonlinear terms are weakly continuous, see Lemma \ref{lemma:intro-CC}, and hence
\begin{align}
[\mathbf{Ric}(\bm g_\e)]_{33}= 0=\Box_{g_\e} \psi_\e +\frac12 e^{-4\psi_\e}g^{-1}_\e(\tp{d}\omega_\e,\tp{d}\omega_\e) \,, \qquad\tp{for all } \e\geq 0\,.\label{eq:Ricci-coefficient-33}
\end{align}
Thus, the $(\alpha,3)$ and $(3,3)$ directions provide no contributions to any matter produced in the limit. Moreover, from \eqref{eq:Ricci-coefficient-alpha3} and \eqref{eq:Ricci-coefficient-33} we obtain the wave map equation
\begin{align} \label{eq:wave-maps-from-Einstein}
\begin{dcases}
\Box_{g_\e} \psi_\e +\frac12 e^{-4\psi_\e}g^{-1}_\e(\tp{d}\omega_\e,\tp{d}\omega_\e)=0\,,\\
\Box_{g_\e} \omega_\e-4g^{-1}_\e(\tp{d}\psi_\e,\tp{d}\omega_\e)=0\,,
\end{dcases}
\end{align}
from $(\mc{M},g_\e)$ to the Poincaré plane $(\mathbb{R}^2,\mathfrak{g})$, where $\mathfrak{g}=2(\tp{d}\psi)^2+\frac12e^{-4\psi}(\tp{d}\omega)^2$.  We recall that \eqref{eq:wave-maps-from-Einstein}, being a wave map system, is the Euler--Lagrange equation for a Lagrangian on the domain $(\mc{M},g_\e)$; in this case, the Lagrangian density is
\begin{align*}
\mathbb{L}_{\alpha\beta}[\psi_\e,\omega_\e]\equiv 2\p_\alpha\psi_\e\p_\beta\psi_\e +\frac{1}{2}e^{-4\psi_\e}\p_\alpha \omega_\e\p_\beta\omega_\e\,, \qquad \e\geq 0.
\end{align*}

\vspace{-.66\baselineskip}
\paragraph{Curvature in the non-symmetric directions.} Finally, we turn to the curvature in the $(\alpha,\beta)$ directions. From the vacuum condition \eqref{eq:Einstein-equation-e-firstappearance} on $\bm{g}_\e$ and its $U(1)$-symmetry, we find that 
\begin{equation}
[{\rm Ric}(g_\e)]_{\alpha\beta}= \mathbb{L}_{\alpha\beta}[\psi_\e,\omega_\e]=2\p_\alpha\psi_\e\p_\beta\psi_\e +\frac{1}{2}e^{-4\psi_\e}\p_\alpha \omega_\e\p_\beta\omega_\e\,, \qquad \tp{only for }\e>0\,. \label{eq:quasilinear-condition}
\end{equation}
Thus, for the $U(1)$-symmetric weak limit $\bm{g}$, we easily compute
\begin{align}
[\mathbf{Ric}(\bm g)]_{\alpha\beta}&= [{\rm Ric}(g)]_{\alpha\beta}-\mathbb{L}_{\alpha\beta}[\psi,\omega]\nonumber\\
&=\underbracket[0.140px]{\text{w*-}\lim_{\e \to 0}\mathbb{L}_{\alpha\beta}[\psi_\e,\omega_\e]-\mathbb{L}_{\alpha\beta}[\psi,\omega]}_{\text{semilinear terms}}
+\underbracket[0.140px]{[{\rm Ric}(g)]_{\alpha\beta}-\text{w*-}\lim_{\e \to 0}[{\rm Ric}(g_\e)]_{\alpha\beta}}_{\text{quasilinear terms}}\,. \label{eq:Ricci-coefficient-alphabeta}
\end{align}
From the symmetry assumptions alone, we find in \eqref{eq:Ricci-coefficient-alphabeta} that there are two different types of contributions to the matter created in the limit in the $(\alpha,\beta)$ directions: those arising from the \textit{semilinear} wave map equation \eqref{eq:wave-maps-from-Einstein} for $\e>0$, and those arising from the \textit{quasilinear} condition \eqref{eq:quasilinear-condition} which makes $g_\e$ in the wave map equation depend on the solution itself. However, it is easy to see that the latter contributions are \textit{forbidden} under the gauge conditions we impose:
\begin{lemma}
\label{lemma:ricconv}
 If Hypotheses~\ref{hyp:Einstein}(\ref{it:hypmetrics-Einstein}) hold, then
$\mr{Ric}(g_\e)\wstar \mr{Ric}(g)$ in the sense of distributions.
\end{lemma}
For the convenience of the reader, we reprove this standard fact about elliptic gauge in Appendix~\ref{sec:appendix}. Thus, the quasilinear terms do not contribute to the matter created in the limit, and \eqref{eq:Ricci-coefficient-alphabeta} becomes
\begin{align}
[\mathbf{Ric}(\bm g)]_{\alpha\beta}&= \text{w*-}\lim_{\e \to 0}\mathbb{L}_{\alpha\beta}[\psi_\e,\omega_\e]-\mathbb{L}_{\alpha\beta}[\psi,\omega]
\,. \label{eq:Ricci-coefficient-alphabeta-simpler}
\end{align}
We conclude that, in order to prove Theorem~\ref{thm:luk-huneau}, it is enough to characterize the failure of compactness in the Lagrangian density associated to a \textit{semilinear} wave map equation such as \eqref{eq:wave-maps-from-Einstein}.

\begin{remark}[Decoupling of the Einstein part]\label{rmk:killing-quasilinearity}  Equation \eqref{eq:Ricci-coefficient-alphabeta-simpler} shows that, from the point of view of Theorem~\ref{thm:luk-huneau}, the wave map and the Einstein parts of the system composed of \eqref{eq:wave-maps-from-Einstein} and \eqref{eq:quasilinear-condition} decouple \textit{completely} thanks to the elliptic gauge conditions. Notice that this is in stark contrast with other types of analysis of the system composed of \eqref{eq:wave-maps-from-Einstein} and \eqref{eq:quasilinear-condition}, such as understanding its well-posedness, see e.g.\ \cite{Huneau2018a,Touati2021}: there, the  quasilinearity is the main difficulty and it cannot be removed by any gauge condition. 
\end{remark}

\subsection{The semilinear terms}
\label{sec:semilinear}

In the previous section we have shown that, in spite of the quasilinear nature of the Einstein equation (\ref{eq:Einstein-equation-1}),  Theorem~\ref{thm:luk-huneau} \textit{de facto} reduces to understanding the semilinear wave map equation \eqref{eq:wave-maps-from-Einstein}. The study of oscillations in solutions to wave map equations in fact has much broader applications, as these are very widely studied systems of nonlinear hyperbolic PDEs, see e.g.\ the classical reference \cite{Shatah2000}. 
Accordingly, for $\e>0$  consider a wave map from a Lorentzian manifold $(\mc M,g_\e)$ to a \textit{fixed} Riemannian manifold $(\mc N,\mathfrak{g})$:
\begin{equation}
\Box_{g_\e} u_\e^{I} +
\Gamma_{JK}^I(u_\e)g^{-1}(\tp{d} u_\e^J,\d u_\e^K)= f_\e^I\,,
\qquad u_\e^I,f_\e^I\colon\mc{M}\to \mathbb{R}\,, \qquad  {I,J,K}\in\{1,\dots,N\}\,.
\label{eq:wave-map-system-intro}
\end{equation}
Here, $\Gamma_{JK}^I\colon \mathbb{R}\to \mathbb{R}$ are the Christoffel symbols of the Riemannian metric $\mathfrak{g}$ and depend continuously on $u^I$. 
For simplicity, we take $\mc M\subset \R^{1+n}$ and $\mc N\subset \R^N$ to be domains, and $\{x^0,x^1,\dots, x^n\}$ to be coordinates on $\mc{M}$ represented with greek indices or, if $x^0$ is excluded, roman indices; however, in light of the assumptions ensuing, this restriction is without loss of generality.  Indeed, we will assume:

\begin{hypotheses}
\label{hyp:main}
Let $\mathscr U_\e\equiv (g_\e, (u_\varepsilon^I)_{I=1}^N, (f_\varepsilon^I)_{I=1}^N)$ and $\mathscr U\equiv (g,(u^I)_{I=1}^N, (f^I)_{I=1}^N)$ satisfy:
\begin{enumerate}[topsep=3pt]
\item\label{it:hypmetrics} the eigenvalues of $g_\e$ are uniformly bounded above and away from zero and $g$ is a smooth metric such that $g_\varepsilon\to g$ in $C^0_{\tp{loc}}$, $g_\e$ is bounded in $W^{1,\infty}_\tp{loc}$, $\p_0 (g_\e)_{ij}\to \p_0 (g_\e)_{ij}$ strongly in $L^2_\tp{loc}$, and $\delta^{ij}\p_{ij}g_\e^{\alpha\beta}$ is bounded in $L^2_\tp{loc}$;
\item $u^I_\varepsilon$ converges to $u^I$ uniformly in $C^0_\tp{loc}$ and weakly in $W^{1,4}_\tp{loc}$; \label{it:hypsols}
\item  $\lVert g_\e^{\alpha\beta}-g^{\alpha\beta}\rVert_{L^\infty(K)}\lVert \p^2 (u_\e^I-u^I)\rVert_{L^4(K)}\lesssim_K 1$ for every compact $K\subset \mc{M}$; \label{it:hyprates}
\item  $f^I_\e\w f^I$ in $L^4_\tp{loc}$.\label{it:hypsource}
\end{enumerate}
\end{hypotheses}

\begin{remark}
For $n=2$, Hypotheses \ref{hyp:main}\eqref{it:hypmetrics} are implied by Hypotheses \ref{hyp:Einstein}\eqref{it:hypmetrics-Einstein}, see Appendix \ref{sec:appendix}.
\end{remark}

The convergence of $\mathscr U_\e$ assumed in  Hypotheses~\ref{hyp:main} is strong enough to  easily ensure that $(g,u^1, \dots, u^N)$ is itself a wave map. This is a substantially more difficult task under weaker hypotheses, see for instance \cite{Bahouri1999,Freire1997,Freire1998,Gerard1996} for several examples of oscillation and concentration effects in semilinear wave equations in lower regularity, albeit in  settings where  $g_\e=g$ is the Minkowski metric. On the other hand, Hypotheses~\ref{hyp:main} are weak enough that general quadratic quantities in the solutions, such as the Lagrangian density
\begin{align}  \label{eq:wave-map-energy-density}
\mathbb{L}_{\alpha\beta}[u_\e]\equiv \mathfrak{g}_{IJ}(u_\e)\p_\alpha u^I_\e\p_\beta u^J_\e\,, 
\end{align}
which features in the variational principle from which \eqref{eq:wave-map-system-intro} is derived, are not preserved in the limit as $\e\to 0$. With Theorem~\ref{thm:luk-huneau} and, specifically, \eqref{eq:Ricci-coefficient-alphabeta-simpler} in view, our goal is precisely to characterize the failure of compactness in \eqref{eq:wave-map-energy-density}, i.e.\ to identify the compactness singularities \textit{and} describe how they are propagated. For simplicity, we state our main result only for wave maps \eqref{eq:wave-map-system-intro} without sources:

\begin{arabictheorem} \label{thm:wave-map} 
Let $\Gamma^I_{JK}$ be continuous Christoffel symbols arising from a Riemannian metric 
$$\mf g=\mf g_{IJ}(y) \d y^I\otimes \d y^J\,.$$ Let $\mathscr U_\e$ be a sequence of solutions to  \eqref{eq:wave-map-system-intro} with $f_\e^I\equiv 0$. There is a Radon measure $\nu$ on $S^*(\mc M)$ such that:
\begin{enumerate}
\item {\normalfont \textbf{Limit equation.}} \label{it:thm-wave-map-limit-equation}  $\mathscr U$ is a distributional solution of \eqref{eq:wave-map-system-intro} and its Lagrangian energy density satisfies
\begin{align*}
\lim_{\e\to 0} \int_{\mc M} \mathbb{L}_{\alpha\beta}[u_\e]Y^\alpha Y^\beta \d \tp{Vol}_{g_\e} 
=
\int_\mc M \mathbb{L}_{\alpha\beta}[u]Y^\alpha Y^\beta \d \tp{Vol}_{g} 
+\int_{S^*\mc M}  \xi_\alpha \xi_\beta Y^\alpha Y^\beta  \d \nu\,, \quad \forall\, Y\in C^{\infty}_0(\mc{M})\,.
\end{align*}
\item \label{it:thm-wave-map-energy-density}
{\normalfont \textbf{Vlasov equation.}} The measure $\nu$  is a (radially averaged) measure-valued solution of a massless Vlasov equation with respect to $g$, in the sense that properties (\ref{it:thm-cor-luk-huneau-localization}) and (\ref{it:thm-cor-luk-huneau-propagation}) of Theorem \ref{thm:luk-huneau} hold.
\end{enumerate}
\end{arabictheorem}

Strictly speaking, in Theorem~\ref{thm:wave-map}, as well as in Theorem~\ref{thm:linear-wave-oscillating} below, one may need to pass to a subsequence in $\bm{g}_\e$. In fact, throughout the paper we always work modulo subsequences.
We also note that the case $f_\e^I\not\equiv 0$ is very similar: \eqref{it:thm-wave-map-limit-equation} still holds, and in \eqref{it:thm-wave-map-energy-density} the massless Vlasov equation becomes inhomogeneous with source related to the failure of compactness of $(f_\e^I)_{\e>0}$.

The measure in Theorem \ref{thm:wave-map} is essentially an \textit{H-measure} induced by the sequence $\mathscr U_\e$, see Section \ref{sec:prelims-H-measures-CC}. H-measures, often known as \textit{microlocal defect measures} in the literature, were introduced independently by Gérard \cite{Gerard1991} and Tartar \cite{Tartar1990}. H-measures are ideal tools for proving Theorem \ref{thm:wave-map}: like other popular tools to study the failure of strong convergence, such as Young measures, they can be used to compute the difference  between $\mb L_{\alpha\beta}[u]$ and $\lim_{\e\to 0}\mb L_{\alpha\beta}[u_\e]$, but crucially they also capture the way in which this difference propagates.
We refer the reader to  \cite{Rindler2015} for a comparison between  Young measures and H-measures. 

Any sequence $\lp( (\p_{0} u^I_\e, \p_1 u^I_\e, \dots, \p_n u^I_\e, f^I_\e)_{I=1}^N\rp)_\e
$ bounded in $L^2_\tp{loc}$ induces an H-measure
\begin{equation}
\label{eq:defHmeasureintro}
\lp(
\begin{bmatrix}
\tilde \nu^{IJ} & \tilde \lambda^{IJ} \\
(\tilde \lambda^{IJ})^* & \mu^{IJ}
\end{bmatrix}
\rp)_{I,J=1}^N
\,,
\end{equation}
which is valued in $N\times N$ block-matrices. The measures $\tilde \nu^{IJ}$ takes values in $(n+1)\times (n+1)$ matrices, while the measures $\mu^{IJ}$ are scalar; they are essentially computed by respectively evaluating the limits
$$\lim_{\e\to 0}\langle A (\tp d u_\e^I-\tp d u^I),
\tp d u_\e^J-\d u^J\rangle \qquad \tp{ and }\qquad  
\lim_{\e\to 0} \langle B (f_\e^I-f^I), f_\e^J-f^J\rangle\,.$$ 
Here and throughout $\langle\cdot,\cdot\rangle$ denotes the $L^2$-inner product with respect to $\tp d x$, while $A$ and $B$ are zeroth order pseudo-differential operators. Finally, the measures $\tilde \lambda^{IJ}$ capture the interaction between $\tp d u_\e^I$ and $f_\e^J$.

Our strategy to prove Theorem \ref{thm:wave-map} is to rewrite \eqref{eq:wave-map-system-intro} as
\begin{align*}
\Box_{g_\e} u_\e^{I} =Q_\e^I+f_\e^I\,, \qquad Q_\e^I\equiv -\Gamma_{JK}^I(u_\e)g^{-1}(\tp{d} u_\e^J,\d u_\e^K)\,,
\end{align*}
and to interpret the semilinearities in the wave map equations as source terms for a linear wave equation on an oscillating background. As will become clearer in the next subsection, source terms contribute to Theorem \ref{thm:wave-map} only through the H-measure $\tilde \lambda^{IJ}$. Hence, our goal is to compute
\begin{align}
\lim_{\e\to 0} \langle A \p(u_\e^I-u^I), Q_\e^L-Q^L\rangle\,, \label{eq:intro-lambda-semilinearities}
\end{align}
where $Q^L$ denotes the weak limit of $Q_\e^L$ in $L^2_\tp{loc}$. Note that the uniform convergence of $u_\e^I$ ensures that, in $Q_\e^L$, only the null forms $g^{-1}(\tp d u_\e^J,\tp d u_\e^K)$ are important. 
The null structure of the wave map nonlinearities translates into a div-curl structure 
both for bilinear and trilinear terms:
\begin{lemma}[Murat and Tartar \cite{Murat1978, Tartar2005}] \label{lemma:intro-CC} Under Hypotheses~\ref{hyp:main}, we have:
\begin{enumerate}
\item\label{it:bilinCC} $Q^L\equiv \tp{w-}\lim_{\e\to 0} Q_\e^L=-\Gamma_{JK}^I(u)g^{-1}(\tp{d} u^J,\d u^K)$;
\item\label{it:trilinCC} if $u=0$ then $\tp{w-}\lim_{\e\to 0} \p u_\e^I\, g^{-1}(\tp d u_\e^J,\d u_\e^K)=0$, where $\p$ denotes an arbitrary partial derivative.
\end{enumerate}
\end{lemma}
We reprove this classical result in Section \ref{sec:wave-CC} below using the geometric version of the div-curl lemma from \cite{Robbin1987} and the usual geometric framework of energy identities for covariant wave equations. We also alert the reader that \eqref{it:trilinCC} is referred to as \textit{three-wave compensated compactness} in \cite{Huneau2019}.

When $u= 0$, Lemma~\ref{lemma:intro-CC} easily shows that \eqref{eq:intro-lambda-semilinearities} vanishes. However, this is not the case in general, as the trilinear quantity in \eqref{it:trilinCC} is weakly continuous \textit{only at zero}. That such quantities even exist is only possible because $\Box_g$, thought of as a first-order operator acting on $\p u$, does not have constant rank, c.f.\ \cite{Guerra2019} and Remark \ref{remark:constantrank}. The upshot is that in the general case $u^I\not\equiv 0$ the nonlinearities create a \textit{coupling} between the behavior of the measures $\tilde \nu^{IJ}$ and $\tilde \nu^{KL}$, so the lack of compactness in general quadratic quantities associated to wave maps does not admit a simple characterization. 

For the particular quantity we are interested in, the Lagrangian density \eqref{eq:wave-map-energy-density}, something surprising occurs: the couplings between the different measures are added up so as to precisely cancel! Hence, through the classical Lemma~\ref{lemma:intro-CC}, the nonlinear terms can be easily shown not to contribute to the failure of compactness of $\mathbb{L}_{\alpha\beta}[u_\e]$ nor to its propagation. We conclude that, to establish Theorem~\ref{thm:wave-map}, it is enough to characterize the failure of compactness in quadratic quantities associated to a \textit{linear} scalar wave equation with oscillating coefficients.

\subsection{The linear terms}
\label{sec:linear}

We have reduced the proofs of Theorems~\ref{thm:luk-huneau} and \ref{thm:wave-map} to understanding oscillations in a scalar linear wave equation with respect to oscillating background metrics. In other words, we take $N=1$ in Hypotheses~\ref{hyp:main} and hence, for simplicity, we drop the superscripts. 

\begin{arabictheorem} \label{thm:linear-wave-oscillating} 
Let $\mathscr U_\e=(g_\e,u_\e,f_\e)$ be a sequence satisfying Hypotheses \ref{hyp:main} and such that $\Box_{g_\e} u_\e = f_\e.$
\begin{enumerate}
\item {\normalfont\textbf{Limit equation.}} The triple $\mathscr U=(g,u,f)$ is a solution of $\Box_g u =f$. \label{it:lin-wave-eq}
\item {\normalfont \textbf{Vlasov equation.}} \label{it:thm-lin-wave-energy-density}
There are  Radon measures $\nu$, $\lambda$ such that $\tilde \nu_{\a\beta}=\xi_\alpha\xi_\beta\nu$, $\tilde{\lambda}_{\gamma}=\xi_\gamma \lambda$. Moreover, $\nu$ is a (radially averaged) measure-valued solution of an inhomogeneous massless Vlasov equation, in the sense that property (\ref{it:thm-cor-luk-huneau-localization}) of Theorem \ref{thm:luk-huneau} holds, and 
for all $ \tilde a\in C_0^\infty (S^*\mc{M})$, extended as a positively 1-homogeneous function to $T^*\mc M\exc\{0\}$, the measure $\nu$ satisfies 
\begin{gather}
\int_{S^*\mc M} \lp[g^{\alpha\beta}\xi_\alpha\p_{x^\beta}\tilde a-\frac12\p_{x^\mu}g^{\alpha\beta}\xi_\alpha\xi_\beta\p_{\xi_\mu}\tilde a\rp] \tp d\nu 
= -\int_{S^*\mc M} \tilde a\, \tp{d}(\Re \lambda)
\,.
\label{eq:propagationprop}
\end{gather}
\end{enumerate}
\end{arabictheorem}

\begin{remark}[Initial value formulation] The transport equation \eqref{eq:propagationprop} in Theorem~\ref{thm:linear-wave-oscillating}\eqref{it:thm-lin-wave-energy-density} naturally inherits a suitable set of initial conditions in terms of initial conditions for $\Box_{g_\e}u_\e=f_\e$, see \cite[Section 3.4]{Tartar1990} as well as \cite{Francfort1992} for a detailed study when $g_\e=g$ is fixed. In other words, the failure of compactness seen in the evolution may be characterized in terms of failure of compactness of the initial data. 
\end{remark}

\begin{remark}[Regularity of $g$] It is natural to ask whether $W^{1,\infty}_\tp{loc}$-bounds on $g_\e$ in can be weakened to $W^{1,q}_\tp{loc}$-bounds, for some $q<\infty$. This would affect the expected regularity of $g$, which would drop below $C^1$. Such a level of regularity seems problematic: indeed, the integrand in the left-hand side of \eqref{eq:propagationprop} is the Poisson bracket between the symbol of $\Box_g$ and $\tilde{a}$, which in turn is the symbol of a commutator between the corresponding pseudo-differential operators that ought to be at least bounded, c.f.\  
Remark \ref{remark:smoothness}.
\end{remark}

Let us give an outline of the proof of Theorem~\ref{thm:linear-wave-oscillating}. For  a  \textit{fixed} Lorentzian metric, a full characterization of the H-measure associated to the linear wave equation is already essentially contained in Tartar's original paper \cite{Tartar1990}, as well as in \cite{Francfort1992}. For the sake of completeness, in Section~\ref{sec:linear-wave-non-oscillating}, we extend these proofs to general covariant wave equations, relying on a standard geometric version of the energy identity, see e.g.\ \cite{Alinhac2010}. 

The case of \textit{oscillating} metrics $g_\e$, which takes up the entirety of Section~\ref{sec:linear-wave-oscillating} here, is much more involved, as predicted by Francfort--Murat \cite{Francfort1992}; it is, nonetheless, very natural from the point of view of Homogenization Theory \cite{Cioranescu1999}. An obvious additional difficulty of this case is that it is not clear what  is the appropriate notion of convergence for the metrics. Though this is an interesting problem, we do not investigate it here: it turns out that Hypotheses~\ref{hyp:main} provide sets of convergence conditions under which the oscillations of $g_\e$ do not contribute to the propagation of non-compactness. With stronger conditions on the rates of convergence, as mentioned above, this remarkable fact is one of the key observations of Huneau--Luk \cite{Huneau2019}, and it served as inspiration for our work.

Our strategy for dealing with the oscillations of $g_\e$ is to reduce to the case where $g$ is fixed, so we write
$$\Box_{g_\e} u_\e= f_\e\qquad
\implies 
\qquad \Box_g u_\e =  -H_\e+f_\e,\quad \tp{where } H_\e\equiv (\Box_g- \Box_{g_\e})u_\e \,.$$
Determining the contribution of the oscillations of $g_\e$ to the Vlasov equation amounts to calculating 
$$\lim_{\e\to 0}\langle H_\e, A e_0 (u_\e-u)\rangle,\qquad \tp{where } e_0\equiv \p_0+\frac{g^{0i}}{g^{00}}\p_i.$$
Here $A\in \Psi^0_c$ is an arbitrary pseudo-differential operator corresponding to the test function $\tilde{a}$ in \eqref{eq:propagationprop} and the upper indices denote components of the inverse metrics. A parity argument shows that we can assume that the symbol of $A$ is real and even; then, by a careful integration by parts argument, we obtain
\begin{equation}\lim_{\e\to 0} \langle H_\e, A e_0 (u_\e-u)\rangle= \frac 1 2\int_{\R^{1+n}}\p_\alpha (u_\e-u) [g_\e^{\alpha\beta}-g^{\alpha\beta}, A]\p_\beta e_0 (u_\e-u)\d x\,,
\label{eq:intro-commutator}\end{equation}
see Lemma \ref{lemma:commutator-reduction}.
By the Calderón commutator estimate, if $g_\e\to g$ strongly in $W^{1,\infty}_{\tp{loc}}$, then \eqref{eq:intro-commutator} vanishes in the limit. However, even if all derivatives but one converge strongly, this simple proof fails, as the Calderón commutator estimate requires Lipschitz bounds. This is the case in  Hypotheses~\ref{hyp:main}: the assumptions imply that spatial derivatives of $g_\e$ convergence strongly, with $\e_0g_\e$ converging only weakly. 

As is common in compensated compactness, see e.g.\ \cite[Theorem 5.3.2]{Hormander1997}, we examine the failure of compactness in $e_0 g_\e$ in Fourier space, and we denote by  $\Lambda$ the region where the symbol of $e_0$ vanishes. This naturally induces a partition of Fourier space as follows, see Figure~\ref{fig:frequency-space}.

\begin{figure}[htbp]
\centering
\includegraphics[scale=1]{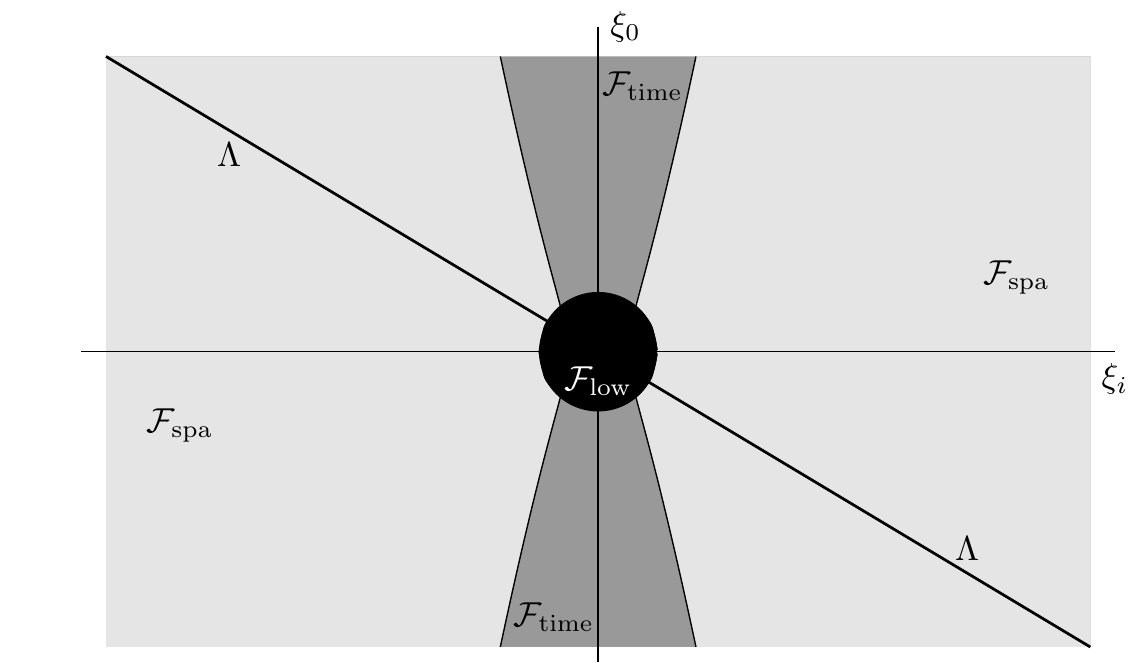}
\caption{The frequency space picture.}
\label{fig:frequency-space}
\end{figure}

\vspace{-0.66\baselineskip}
\paragraph{Low frequencies ($\mc{F}_{\rm low}$).} In bounded regions of frequency space, $W^{1,2}_\tp{loc}$ and $L^2_{\tp{loc}}$ norms are comparable, hence $e_0 g_\e$  is, in fact, \textit{compact} in this range. Indeed, as a general principle, failure of compactness is a high-frequency phenomenon.

\vspace{-0.66\baselineskip}
\paragraph{High frequencies close to $\Lambda$ ($\mc{F}_{\rm space}$).} In this region, $e_0$ is not invertible, so the fact that $u_\e$ appear in the commutator does not help. We instead  \textit{compensate} for the \textit{lack of compactness} in $e_0 g_\e$ by using the fact that the spatial laplacians of $g_\e$ are bounded in $L^2_\tp{loc}$, see Hypotheses~\ref{hyp:main}\eqref{it:hypmetrics}. We alert the reader that this, as well as the argument laid out in the next frequency regime, are referred to as \textit{elliptic-wave compensated compactness}  in \cite{Huneau2019}.

\vspace{-0.66\baselineskip}
\paragraph{High frequencies away from $\Lambda$ ($\mc{F}_{\rm time}$).} 
This is the most difficult regime and, in some sense, the heart of the proof. To illustrate our strategy, let us consider the simple case where the limit metric $g$ is the Minkowski metric and $A$ is a multiplier, i.e.\ its symbol is merely a function $m(\xi)$ for $\xi\in S^*\mc{M}$ which is 0-homogeneous and even. Let us write $w_\e\equiv u_\e-u$ and $h_\e^{\alpha\beta}\equiv g_\e^{\alpha\beta}-g^{\alpha\beta}$. Then, from Plancherel and the parity of $m$, \eqref{eq:intro-commutator} becomes
\begin{align}
\frac12\int\p_\alpha w_\e [h_\e^{\alpha\beta}, A]\p_\beta e_0 w_\e\d x = \frac{i}{4} \iint \xi_i\eta^i(\xi_0+\eta_0)\widehat{h_\e^{\alpha\beta}}(\xi-\eta) \widehat{w_\e}(-\xi)\widehat{w_\e}(\eta)[m(\xi)-m(\eta)] \d\xi \d\eta\,.
\label{eq:modelcalculation}
\end{align}
We now manipulate the symbol  $(\xi_0+\eta_0)\xi_i\eta^i$ as follows. First, we multiply and divide by the symbol of $e_0$ when acting on $h_\e^{\alpha\beta}$, which is $(\xi_0-\eta_0)$; this is allowed since $|\xi_0-\eta_0|\gg 1$ in the frequency regime we are considering. Then, we regroup terms so as to make the symbol of $\Box$, denoted $\sigma_\Box(\eta)\equiv -\eta_0^2+\eta^i\eta_i$, appear:
\begin{align}\begin{split}
(\xi_0-\eta_0)(\xi_0+\eta_0)\xi_i\eta^i 
&= \lp[\sigma_\Box(\eta)-\sigma_\Box(\xi)+\xi_k \xi^k-\eta_k \eta^k\rp]\xi_i\eta^i\\
&= \sigma_\Box(\eta)\xi_i (\eta^i-\xi^i+\xi^i) -\sigma_\Box(\xi) (\xi_i-\eta_i+\eta_i) \eta^i + (\xi^k\xi_i\eta^i + \xi_i\eta^i\eta^k)(\xi_k-\eta_k) \\
&= \xi_i\xi^i\sigma_\Box(\eta)-\eta_i\eta^i\sigma_\Box(\xi) + 
%
%\lp[\xi^i\xi_k\eta^k + \xi_k\eta^k\eta^i - \xi^i\sigma(\Box)(\eta) - \eta^i\sigma_\Box(\xi) \rp] 
%
\lp(\eta^i\xi_0^2+\xi^i\eta_0^2\rp)
(\xi_i-\eta_i)\,.
\end{split}
\label{eq:symbolmanipulation}
\end{align}
Plugging this identity into \eqref{eq:modelcalculation}, we find   that terms which contain $\p^2 w_\e$ are always paired with $h_\e^{\alpha\beta}\Box u_\e$ or with $\p_i h_\e^{\alpha\beta}\p w_\e$. The latter is obviously compact and the former is compact as well, since
\begin{equation}
\label{eq:boundedbox}
\Box_g u_\e \tp{ is bounded in } L^4_\tp{loc}\,,
\end{equation}
unlike general second order derivatives of $u_\e$.
Hence \textit{lack of compactness} of $e_0 g_\e$ is \textit{compensated} by appealing to a differential condition on $u_\e$. It is the last two manipulations in \eqref{eq:symbolmanipulation} that ensure we have no more than two derivatives on each $w_\e$ and no more than one derivative on $h_\e^{\alpha\beta}$. This extra step means that our Hypotheses~\ref{hyp:main} contain no assumptions on derivatives of order $k>2$, c.f.\ \cite{Huneau2019} where assumptions on up to $k=4$ are imposed. 

\begin{remark}[The role of rate assumptions]
To apply compensated compactness methods it is crucial that we have differential information on the sequence with respect to fixed $\e$-independent differential operators, as in \eqref{eq:boundedbox}. Hypotheses \ref{hyp:main}\eqref{it:hypsource} are not sufficient to deduce \eqref{eq:boundedbox} and so, in the spirit of Conjecture \ref{conj:Burnett}, we require rate assumptions in Hypotheses \ref{hyp:main}\eqref{it:hyprates}.
\end{remark}

The simple proof we have given here for the case where $g$ is Minkowski in fact generalizes to any constant coefficient metric $g$, as long as one still takes $A$ to be a multiplier. When  $A$ is a true pseudo-differential operator with $x$-dependence and/or $g$ is $x$-dependent, an application of Plancherel leads to convolutions, and the division by the symbol of $e_0$, which may itself be $x$-dependent, becomes tricky. One way to remedy this situation is to apply cutoffs to ``freeze'' the $x$-dependence of $g$ and $A$, making them locally constant in $x$: if the balls where the freezing is done shrink in an appropriate way as $\e\to 0$, the above argument works. This is the route taken in \cite{Huneau2019} but it requires additional assumptions: in Hypotheses~\ref{hyp:main}\eqref{it:hyprates}, we  would also need information on the rate of uniform convergence of $u_\e$ compared to $\p^2 u_\e$ and $g_\e$. 

In this paper we do not take the previous approach and instead we implement the strategy outlined by \textit{defining} the inverse of $e_0$ as a pseudo-differential operator, which exists in the frequency regime we are considering. This simplifies the argument considerably and our proof is purely based on integration by parts:
\begin{enumerate}
\item We write $h_\e^{\alpha\beta}=e_0(e_0^{-1}h_\e^{\alpha\beta})$. Integrating by parts brings the extra $e_0$ derivative onto $u_\e$; the trilinear form of \eqref{eq:intro-commutator} is then key.
\item Relying on parity arguments and the structure of the commutator, we can use the extra $e_0$ derivative to fashion $\Box_g u_\e$ out of the second derivatives on $u_\e$ which appear. Further integration by parts ensures that we have only up to two derivatives of $u_\e$ and one derivative of $g_\e$.
\end{enumerate}
Combining the previous two points we show that \eqref{eq:modelcalculation} vanishes as $\e\to 0$, completing the proof.

\medskip

\noindent\textbf{Acknowledgments.} The authors were supported by the EPSRC, respectively grants [EP/L015811/1] and [EP/L016516/1]. We warmly thank Maxime van de Moortel for carefully reading an earlier version of the manuscript and Mihalis Dafermos for useful suggestions. We also thank all of those who came to celebrate the 4th of July with us.

\section{Preliminaries on H-measures and compensated compactness}
\label{sec:prelims-H-measures-CC}

\subsection{Symbols and pseudo-differential operators}
\label{sec:symbols-psidos}

In this section we gather some basic results about pseudo-differential operators. They can be found, for instance, in the books \cite{Hormander2007} and \cite{Grubb2009}. We take $\Omega\subset \R^N$ to be a fixed open set throughout.

\begin{definition}\label{symbols}
For $m\in \R$, a function $a$ is called a \textit{symbol of order} $m$, $a\in S^m\equiv S^m(\Omega,\C^{d\times d})$, if $a \in C^\infty(\Omega\times \R^N, \C^{d\times d})$ and, for each compact set $K\subset \Omega$, 
$$|\p_x^\a\p^\beta_\xi a(x,\xi)|\lesssim_{\a,\beta,K} (1+|\xi|)^{m-|\beta|}.$$
We write $S^{-\infty}=\bigcap_{m\in \R} S^m$.
\end{definition}

The following basic lemma gives meaning to asymptotic expansions of symbols:

\begin{lemma}
For $j\in \N_0$ let $a_j\in S^{m_j}$ and $m_j\searrow -\infty$. There is $a\in S^{m_0}$ such that, for every $k$,
$a-\sum_{j<k} a_{j}\in S^{m_k}$. 
The symbol $a$ is unique modulo $S^{-\infty}$ and we write
$a\sim \sum_{j=0}^\infty a_{j}$ in $S^m.$
\end{lemma}

Each symbol $a\in S^m$ induces an operator $A$ acting on $v\in C^\infty_c(\R^N,\C^d)$ by
$$Av(x)\equiv\int_{\R^N} a(x,\xi) e^{2 \pi i x\cdot \xi} \widehat v(\xi)  \d \xi.$$
We say that $A$ is a \textit{pseudo-differential operator} of order $m$. We write $\sigma(A)\equiv a$ and note that, for any pseudo-differential operator, the symbol $\sigma(A)$ is uniquely determined modulo $S^{-\infty}$.

\begin{lemma}\label{lemma:continuity}
%If $a\in S^m$ is compactly supported in $x$, then $A$ extends as a continuous operator $A\colon H^s(\R^N,\C^d)\to H^{s-m}(\Omega,\C^d)$. 
%\red{aqui talvez fosse melhor uma versão com locs}
If $a\in S^m$ then $A$ extends a continuous operator $A\colon H^s(\R^N,\C^d)\to H^{s-m}_\tp{loc}(\Omega,\C^d)$. In particular, if $m<0$ then $A\colon L^2(\R^n,\C^d)\to L^2_\tp{loc}(\Omega,\C^d)$ is compact.
\end{lemma}

We will work with a more restricted class of pseudo-differential operators, the so-called polyhomogeneous operators. To motivate the next definition, observe that if $a\in C^\infty(\Omega\times \R^N)$ satisfies
\begin{equation*}
a(x,t\xi)=t^m a(x,\xi) \quad \tp{ for all }t,|\xi|\geq 1,
\end{equation*}
then $a\in S^m$. Such functions are said to be \textit{positively $m$-homogeneous in $\xi$} for $|\xi|\geq 1$.

\begin{definition}
A symbol $a\in S^m$ is called \textit{polyhomogeneous} if 
$$a\sim \sum_{j=0}^\infty a_{m-j} \quad \tp{in } S^m$$
where  $a_{m-j}\in C^\infty(\Omega\times \R^N)$ is  positively $(m-j)$-homogeneous  in $\xi$ for $|\xi|\geq 1$.
The term $a_m$ is called the \textit{principal symbol} and is denoted by $\sigma^m(A)$.

The space of pseudo-differential operators with polyhomogeneous symbols in $S^m(\Omega,\C^{d\times d})$ is denoted by $\Psi^m_d(\Omega)$; if their symbols are compactly supported in $x$, we write $\Psi^m_{d,c}(\Omega)$.
\end{definition}

\begin{lemma}\label{lemma:expansions}
Take $P\in \Psi^l_{d}(\Omega)$ and $Q\in \Psi^m_{d}(\Omega)$. Writing $\D\equiv \frac{1}{i} \p$, we have the formulae
\begin{gather*}
\sigma(P^*) \sim \sum_{\a\in \N^n_0} \frac{1}{\a!} \p^\a_x \D^\a_\xi \sigma(P)^* \quad \tp{in } S^m\,,\qquad 
\sigma(PQ)\sim \sum_{\a\in \N^n_0}\frac{1}{\a!}\D^\a_\xi \sigma(P) \,\p^\a_x \sigma(Q) \quad \tp{in } S^m\,.
\end{gather*}
Thus, if $[\sigma(P),\sigma(Q)]=0$, then $[P,Q]\in \Psi^{l+m-1}_d(\Omega)$ with
$\sigma^{l+m-1}([P,Q]) = \frac 1 i \{\sigma^l(P),\sigma^m(Q)\}.$
\end{lemma}

Here, and in the sequel, $[p,q]\equiv pq-qp$ and $\{p,q\}$ denotes the \textit{Poisson bracket}, that is, 
$$\left\{p,q\right\}\equiv \frac{\p p}{\p \xi_j} \frac{\p q}{\p x^j}-\frac{\p p}{\p x^j} \frac{\p q}{\p \xi_j}\,.$$

\begin{theorem}[Calderón Commutator]
\label{thm:Calderon}
Let $P\in \Psi^1_1(\R^N)$ and let $a(x)$ be a Lipschitz function. Then, for any $1<p<\infty$, $[P,a]\colon L^p(\R^N)\to L^p(\R^N)$ is bounded and 
$$\Vert [P,a] f\Vert_{L^p} \leq C_p \Vert \nabla a \Vert_{L^\infty} \Vert f \Vert_{L^p}.$$
Conversely, if $[P,a]\colon L^2(\R^N)\to L^2(\R^N)$ is bounded for $P=\p_{x_j}$, $j=1,\dots, N$, then $a$ is Lipschitz.
\end{theorem}

We refer the reader to \cite{Meyer2000} for a proof of Theorem \ref{thm:Calderon}.
%In the modern approach to the subject, Theorem \ref{thm:Calderon} is usually derived from the T(1) theorem of David and Journé, see 

%Classical pseudo-differential operators have a \textit{principal symbol}:
%
%\begin{lemma}\label{lemma:principalsymbol}
%There is a linear map $\sigma^m\colon \Psi^m(\Omega,\C^{d\times d})\to C^\infty_0(S^* \Omega,\C^{d\times d})$, such that:
%\begin{enumerate}
%\item\label{it:welldef} $\sigma^m$ is onto and $\sigma^m(P)=0$ if and only if $P\in \Psi^{m-1}_c(\Omega, \C^{d\times d})$;
%\item if $P\in \Psi^l_c$ and $Q\in \Psi^m_c$ then $\sigma^{l+m}(PQ)=\sigma^l(P)\sigma^m(Q)$;
%\item\label{it:adjoints} $\sigma^m$ commutes with adjoints, that is, $\sigma^m(P^*)=\sigma^m(P)^*$;
%\item if $P=\sum_{|\a|\leq m} p_\a(x)(-i\p_x)^\a$ then $\sigma^m(P)=\sum_{|\a|=m} p_\a(x)\xi^\a$.
%\end{enumerate}
%\end{lemma}
%
%We remark that, by Lemma \ref{lemma:principalsymbol}\ref{it:welldef}, the principal symbol is unique modulo $\Psi^m_c$.
%
%\begin{lemma}\label{lemma:furtherprops}
%Take $P\in \Psi^l_{d,c}(\Omega)$ and $Q\in \Psi^m_{d,c}(\Omega)$. We have:
%\begin{enumerate}
%\item\label{it:firstpart}
% $\sigma(P^*)-\sigma(P)^*-\frac 1 i \sum_{j=1}^n \frac{\p^2 \sigma(P)}{\p y^j \p \xi_j} \in S^{m-2};$
%\item\label{it:secondpart} if additionally  $\sigma(P)$ commutes with $\sigma(Q)$, then
%$[P,Q]\equiv PQ-QP\in \Psi^{l+m-1}_c(\Omega,\C^{d\times d})$ and
%$$\sigma(PQ-PQ)= \frac 1 i\left\{\sigma(P),\sigma(Q)\right\}\equiv 
%\frac 1 i \left(\nabla_\xi \sigma(P) \cdot \nabla_y \sigma(Q) - \nabla_y\sigma(P) \cdot \nabla_\xi \sigma(Q)\right).$$
%\end{enumerate}
%\end{lemma}

\subsection{Existence and properties of H-measures}
\label{sec:existence-characterization-H-measure}

In this subsection we recall the definition of H-measures, as well as a few useful properties they possess. H-measures were introduced independently by Tartar \cite{Tartar1990,Tartar2010} and Gérard \cite{Gerard1991}, who called them \textit{microlocal defect measures}. Here we adopt Tartar's terminology and refer the reader to \cite{Tartar2010} for further details.

\begin{theorem}[Existence of H-measures]\label{thm:Hmeasure}
Let $v_\e \w v$ in $L^2(\Omega,\C^d)$. Up to a subsequence, there are Radon measures $\mu_{\a\beta}$, $\a,\beta=1,\dots, d$, such that
$$\mu_{\a\beta}= \overline{\mu_{\beta\a}}, \qquad \mu_{\a\beta}\xi^\a \bar \xi^\beta\geq 0 \tp{ for all } \xi\in \C^d$$
and, for any $A\in \Psi^0_{d,c}(\Omega)$, we have
\begin{equation}
\label{eq:Hmeasurecharact}
\lim_{\e\to 0}\langle A (v^\e-v),v^\e-v\rangle\equiv 
\lim_{\e\to 0} \int_{\Omega} A(v_\e -v) \cdot \overline{v_\e -v} \d x= \int_{S^* \Omega} \sigma^0(A)^{\a\beta} \d\mu_{\a\beta}\equiv 
%\int_{S^* \Omega} \tp{tr}(\sigma^0(A)\mu)\equiv
 \langle \mu, \sigma^0(A)\rangle.
\end{equation}
The matrix-valued measure $\mu=(\mu_{\a\beta})_{\a,\beta}$ is called the {\normalfont H-measure} associated with $(v_\e)$.
\end{theorem}

In Theorem \ref{thm:Hmeasure}, as usual, $S^*\Omega\equiv\Omega\times S^{N-1}$ denotes the cosphere bundle over $\Omega$. Here, and in the rest of the paper, we will always write $\langle f,g \rangle\equiv \int_\Omega f \bar g \d x$ whenever this integral is meaningful.

\begin{remark}
\label{remark:symbolform}
The Stone--Weierstrass Theorem and a standard density argument show that it suffices to test \eqref{eq:Hmeasurecharact} with symbols of the form $\sigma^0(A)(x,\xi)=b(x)m(\xi)$, see also \cite[Remark 2.7]{Francfort2006}.
\end{remark}

The following lemma, although simple, describes a very important property of H-measures. 

\begin{lemma}[Localization property]\label{lemma:localization}
Let $(v_\e)$ be a sequence such that $v_\e\w v$ in $L^2(\Omega,\C^d)$ and let $\mu$ be its H-measure. Given $P\in \Psi^m_{d}(\Omega)$, we have
$$(Pv_\e)\tp{ is compact in } H^{-m}_\tp{loc} \iff \sigma^m(P)\mu=0.$$
\end{lemma}

To conclude this subsection we define a way of generating, in a \textit{non-canonical} fashion, an H-measure for a sequence that converges only locally in $L^2$:

\begin{definition}\label{def:Hmeasurelocal}
By passing to a subsequence, $v_\e\w v$ in $L^2_\tp{loc}(\Omega,\C^d)$ \textit{generates an H-measure} $\mu$,
$$v_\e\wH \mu,$$
as follows. Let $(K_i)_{i=1}^\infty$ be a compact exhaustion of $\Omega$ and let $\chi_i\in C^\infty_c(K_{i+1},[0,1])$ be such that $\chi_i=1$ on $K_i$. Consider a sequence of Radon measures $(\mu_i)$ constructed as follows: $\mu_1$ is the H-measure generated by a subsequence $(\chi_1 v_{\e'})_{\e'}$ of $(\chi_1 v_\e)_\e$, $\mu_2$ is the H-measure generated by a subsequence of $(\chi_2 v_{\e'})_{\e'}$, and so on. We define $\mu$ through its action on $\varphi \in C_c(S^*\Omega)$: let $i$ be such that $\supp \varphi \subset S^*K_i$ and set
$\langle \mu, \varphi \rangle \equiv \langle \mu_i,\varphi \rangle.$
It is easy to see that $\mu$ is well-defined.
\end{definition}

\subsection{Compensated compactness}

%Theorem \ref{thm:Hmeasure} yields a simple proof of a very general version of Tartar's quadratic theorem \cite{Tartar1979}. In order to state it, we consider a general operator
%\begin{equation}Pv(x)\equiv \sum_{|\a|\leq m} \D_\a \left (P^\a(x) v(x)\right),
%\label{eq:operator}
%\end{equation}
%where $P^\a\in C^0(\Omega,\R^{d'\times d})$. 
%
%\begin{theorem}[Bilinear compensated compactness]
%\label{thm:compensatedcompactness}
%Let $Q\in C^0(\Omega,\R^{d\times d})$ satisfy the condition
%$$
%\sigma^m(P)(x,\xi)\lambda = 0 \implies \lambda^\tp{T} Q(x) \lambda=0, \qquad
%\text{for all } (x,\xi,\lambda)\in S^*\Omega\times \R^d.
%$$
%Then, for an operator $P$ as in \eqref{eq:operator}, we have
%$$\begin{rcases}
%v_\e\w v \text{ in } L^2_\tp{loc}(\Omega,\R^d)\\
%P v_\e \text{ is compact in } W^{-m,2}_\tp{loc}
%\end{rcases}
%\quad \implies \quad 
%\langle Q v_\e,v_\e\rangle \wstar \langle Q v, v\rangle \text{ in the sense of distributions.}$$ 
%\end{theorem}

The next theorem, which is due to Robbin--Rogers--Temple \cite{Robbin1987} and generalizes an earlier result of Murat and Tartar \cite{Murat1978}, is the main compensated compactness result that we will use:

\begin{theorem}[Generalized div-curl lemma]
\label{thm:p-q-divcurl}
Let $p_1,p_2\in (1,\infty)$ be such that $\frac{1}{p_1} + \frac{1}{p_2}=1$. For differential forms $\omega_{i,\e}$ over $\Omega$ of degree $k_i$, $i=1,2$, such that $k_1+k_2\leq N$,
$$
\begin{rcases}
\omega_{i,\e} \w \omega_i  \text{ in } L^{p_i}_\tp{loc}(\Omega)\\
\d \omega_{i,\e} \text{ is compact in } W_\tp{loc}^{-1,p_i}(\Omega)
\end{rcases} \quad \implies \quad \omega_{1,\e} \wedge \omega_{2,\e} \wstar \omega_{1}\wedge \omega_2 \text{ in } \mathscr D'(\Omega).
$$
\end{theorem}

The case $p=q=2$ can be proved easily using H-measures, but for the general case one needs to use the H\"ormander--Mihlin multiplier theorem, which is applicable since the differential constraint in Theorem \ref{thm:p-q-divcurl} has \textit{constant rank} \cite{Murat1981}. We refer the reader to \cite{GuerraRaita2020,Raita2019} for characterizations of constant rank operators and to \cite{Guerra2019,Guerra2020} for generalizations of Theorem \ref{thm:p-q-divcurl} to this setting. 
 
The $L^p$-theory of compensated compactness, even in the bilinear setting, is extremely useful to deal with higher-order nonlinearities, and in fact Theorem \ref{thm:p-q-divcurl} extends straightforwardly to the general multilinear setting.
However, it is worthwhile noting that the $L^p$-theory in the non-constant rank case is still poorly understood. The classical wave operator $\Box \equiv -\p_{tt}+\Delta_x$, if rewritten as a first-order system, is an important example of such an operator but, due to the particular structure of $\Box$, Theorem \ref{thm:p-q-divcurl} will be enough for our purposes.

\section{The linear covariant wave equation}
\label{sec:linear-wave-non-oscillating}

This section is concerned with a linear covariant wave equation
\begin{equation}
\Box_{g} u =f\,,  \qquad u,f\colon \mc M\to \R\,,
 \label{eq:linear-wave-system}
\end{equation}
where $g$ is a smooth Lorentzian metric on an open domain $\mc M\subset \R^{1+n}$. 
Recall that
\begin{equation}
\Box_g u \equiv \frac{1}{\sqrt{|g|}} \p_\a\left(\sqrt{|g|}g^{\a\beta} \p_{\beta} u\right)=\nabla^\a \nabla_\a u\,, \label{eq:covariant-wave-op}
\end{equation}
where $g^{\a\beta}\equiv (g^{-1})^{\a\beta}$, $|g|\equiv |\tp{det} \,g|$ and $\nabla^\a$ is the covariant derivative with respect to $g$. We will also write
$\tp{dVol}_g\equiv \sqrt{|g|} \d x$ for the volume form induced by $g$.

It will be convenient to work with a diagonalized form of the wave operator. To this end, define
\begin{align}
\label{eq:Cauchyframe}
\beta^i\equiv -\frac{g^{0i}}{g^{00}}\,, \qquad e_0\equiv \p_0-\beta^i\p_i\,, \qquad \tilde{g}^{ij}\equiv g^{ij}-\frac{g^{0i}g^{0j}}{g^{00}}\,, 
\end{align}
The symbol of the timelike vector field $e_0$ appears naturally in relation to the zero mass shell of $g$: indeed,
\begin{equation}
\label{eq:lightcone}
g^{\a\beta}\xi_\a\xi_\beta = g^{00} (\xi_0-\beta^k\xi_k)^2 + \tilde{g}^{ij}\xi_i\xi_j\,.
\end{equation}

In order to use Stokes' theorem, we define some useful geometric quantities associated with the covariant wave operator. Given functions $u_1,u_2\colon \mc M\to \R$ and a smooth vector field $X$ on $\mc M$, let us write
\begin{equation}
\label{eq:defTJ}
\begin{split}
T_{\a\beta}[u^1, u^2] & \equiv \p_\a u^1 \,\p_\beta u^2-\frac 1 2 g_{\a\beta}g^{\mu\nu}\p_\mu u^1\,\p_\nu u^2\,,\\
J^X_\a[u^1, u^2] & \equiv \frac 1 2 \lp[Xu^1\,\p_\a u^2 + X u^2\, \p_\a u^1- X_\a g^{-1}(\tp{d} u^1,\d u^2)\rp]\,.\\
\end{split}
\end{equation} 
The energy-momentum tensor $T$ and the associated current $J^X$ are related by the \textit{energy identity}:
\begin{equation}
\nabla^\a J^X_\a[u^1,u^2] = \frac 1 2(X u^1\, \Box_g u^2 + Xu^2\,\Box_g u^1)
 + T_{\a\beta}[u^1, u^2]\,\nabla^\a X^\beta.
 \label{eq:energyidentity}
\end{equation}
When $u^1=u^2=u$ we recover the standard energy identity, see e.g.\ \cite{Alinhac2010, Dafermos2008} for further details. 
%In that case we simply write $T_{\a\beta}[u]\equiv T_{\a\beta}[u,u]$ and likewise for $J^X$.

In this section we study the limiting behavior of sequences of solutions to \eqref{eq:linear-wave-system}. For the convenience of the reader, we state here a simplified form of Hypotheses \ref{hyp:main}:
\begin{hypotheses}
Let $u_\varepsilon,f_\varepsilon\colon \mc M\to \mathbb{R}$ be sequences such that $(u_\e,f_\e)$ satisfy,  for each $\varepsilon>0$, the linear wave equation \eqref{eq:linear-wave-system}.
We consider the following regularity conditions:
\begin{enumerate}
\item $g$ is smooth; \label{assum:smooth-metric}
\item $u_\varepsilon\w u$ in $W^{1,2}_\tp{loc}(\mc M)$; \label{assum:bounds-ueps-LinWave}
\item $f_\e\w f$ in $L^2_\tp{loc}(\mc M)$.
 \label{assum:bounds-box-ueps-LinWave}
\end{enumerate}
\label{hyp:linear-wave-system}
\end{hypotheses}

According to Definition \ref{def:Hmeasurelocal} and Hypotheses~\ref{hyp:linear-wave-system},  we may pass to a subsequence so that 
\begin{equation}
\lp( \p_0 u_\e, \p_1 u_\e, \dots, \p_n u_\e, f_\e \rp)_\e
\wH \begin{bmatrix} \tilde \nu & \tilde \lambda \\ 
\tilde \lambda^* & \mu
\end{bmatrix}\,,
\label{eq:defHmeasure}
\end{equation}
where $\tilde \nu$ is a $\C^{(n+1)\times (n+1)}$-valued measure, generated by $(\p_0 u_\e,\dots, \p_n u_\e)$, and $\tilde \lambda$ is $\C^{n+1}$-valued.

\subsection{The H-measure and its properties}

We are now ready to state the main result of this section, which describes the structure, support and propagation properties of the H-measure defined in \eqref{eq:defHmeasure}.
\begin{theorem}\label{thm:linearwave}
Let $(u_\varepsilon,f_\e)$ satisfy Hypotheses~\ref{hyp:linear-wave-system} and define $\tilde \nu$ and $\tilde{\lambda}$ as in \eqref{eq:defHmeasure}. Then:
\begin{enumerate}
\item \label{it:limit-eq-lin-wave}
{\normalfont \textbf{Limit equation.}} $(u,f)$ satisfy \eqref{eq:linear-wave-system} in the dense of distributions.
\item
{\normalfont \textbf{Energy density.}}
\label{it:structure} There are Radon measures $\nu$ and $\lambda$ on $S^*\mc M$ such that $\tilde \nu_{\alpha\beta}=\xi_\alpha\xi_\beta\nu$ and $\tilde{\lambda}_{\gamma}=\xi_\gamma \lambda$.
Furthermore, $\nu$ and $\lambda$ satisfy the following conditions:
\begin{enumerate}
[ref={\alph{enumi}\textsubscript{\arabic*}},label={\normalfont(\alph{enumi}\textsubscript{\arabic*})}, itemsep=0em,start=0]
\item \label{it:parity} {\normalfont Parity:} $\nu$ is even and $\lambda$ is odd, i.e.\
$\langle \nu, \tilde a\rangle =0$ for any $\tilde a\in C^\infty_c(S^*\mc M)$ which is odd in $\xi$,
and likewise for $\lambda$.
\item\label{it:support} {\normalfont Support property:}  for all $\varphi \in C^\infty_c(\mc M)$, $\nu$ and $\lambda$ satisfy
\begin{align*}
\langle \nu,\varphi(x) g^{\alpha\beta}(x)\xi_\alpha\xi_\beta\rangle=0,
\qquad
\langle \lambda,\varphi(x) g^{\alpha\beta}(x)\xi_\alpha\xi_\beta\rangle=0\,.
\end{align*}
\item\label{it:propagation} {\normalfont Propagation property:} for all $\tilde a(x,\xi)\in C^\infty_c(S^*\mc M)$, though of as positively 1-homogeneous functions in $\xi$, the measure $\nu$ satisfies
$$\langle \nu, \{g^{\alpha\beta}(x)\xi_\alpha \xi_\beta,\tilde a\}\rangle = - \langle \Re\lambda,\tilde a\rangle.$$
\end{enumerate}
\end{enumerate}
\end{theorem}

Theorem \ref{thm:linearwave} follows by standard methods, and similar statements have appeared in \cite[Theorem 3.12]{Tartar1990} and \cite{Antonic1996, Francfort2006}. 
%In \cite{Francfort1992}, a particular case of Theorem \ref{thm:linearwave} is used to describe the local behavior of the energy density in the limit. 
The main novelty here is that our proof holds for a general \textit{covariant wave operator} where, unlike in these references, the coefficients of the operator are allowed to depend both on $x^0=t$ and $(x^1,\dots,x^n)$. 

Before proceeding with the core of the proof, we show that we may assume that the convergence in Hypotheses \ref{hyp:linear-wave-system} is \textit{global} and not just local:

\begin{proof}[Reduction to compact supports]
Let $\chi \in C^\infty_c(\mc M)$ satisfy $\chi=1$ on a compact set $K$. Then
$$\Box_g(\tilde u_\e)\equiv \Box_g(\chi u_\e)=  f_\e\chi + 2 g^{-1}(\tp{d}u_\e, \d \chi) + u_\e \Box_g \chi \equiv \tilde f_\e.$$
Suppose that, for every such $\chi$, the conclusion of Theorem \ref{thm:linearwave} holds, with $\tilde \nu$ and $\tilde \lambda$ being now the H-measures generated according to \eqref{eq:defHmeasure}, but with $u_\e$ replaced with $\tilde u_\e$ and $f_\e$ replaced with $\tilde f_\e$. Since $(u_\e,f_\e)=(\tilde u_\e,\tilde f_\e)$ on $K$, it is then clear, recalling Definition \ref{def:Hmeasurelocal}, that the original H-measure generated by $(u_\e,f_\e)$ also satisfies the conclusion of Theorem \ref{thm:linearwave}.
\end{proof}

Thus, from now onwards, we assume that the sequence $(u_\e,f_\e)_\e$  has uniformly bounded support.

\begin{proof}[Proof of Theorem \ref{thm:linearwave}(\ref{it:limit-eq-lin-wave},\ref{it:parity},\ref{it:support})]
Part \eqref{it:limit-eq-lin-wave} follows from the divergence structure of $\Box_g$, see Proposition \ref{prop:limit-equation-LinWaveOsc} for a more general statement.

Noting that  
$\tp{D}_\a \p_\beta u_\e = \tp{D}_\beta\p_\a u_\e$, Lemma \ref{lemma:localization} yields 
$\xi_\a \tilde \nu_{\beta \gamma} = \xi_\beta \tilde \nu_{\a \gamma}$. It follows that $\tilde \nu_{\a\beta} = \xi_\a {\rho}_\beta$ for some $\C^d$-valued Radon measure $\rho$. Since $\mu$ is Hermitian and non-negative, we must have $\rho = \xi \nu$ for another non-negative Radon measure $\nu$. Likewise, $\tilde{\lambda}_{\gamma}=\xi_\gamma \lambda$ for some Radon measure $\lambda$.

The support property of $\nu$ in (\ref{it:support}) follows by applying again Lemma \ref{lemma:localization}: since $\Box_g u_\e= f_\e$, by Hypotheses \ref{hyp:linear-wave-system}(\ref{assum:bounds-box-ueps-LinWave})  we see that the sequence of vector fields $(\sqrt{|g|} g^{\a \beta} v_{\beta,\e})_\a$ has a divergence which is compact in $H^{-1}_\tp{loc}$ and so
$g^{\a \beta} \xi_\a \xi_\beta \nu = 0$.
%The second part now follows by applying again Lemma \ref{lemma:localization}, noting that $\sigma^2(\Box_g)=g^{-1}$.
%For the second part, by \ref{assum:bounds-box-ueps} or Lemma~\ref{lemma:reduction-g0}, $$\Box_{g_0} u_\e = \frac{1}{\sqrt{|g_0|}} \p_\a\left(g^{\a\beta}\sqrt{|g_0|} (v_\e)_\beta\right) \text{~compact in $H^{-1}_\tp{loc}$}\,.$$
%Thus we can again apply Lemma \ref{lemma:localization} to conclude that  $(g_0^{-1})_{\a\beta} \nu =0$.  
In turn, the support of $\lambda$ is contained in the support of $\nu$. Indeed, from  \eqref{eq:defHmeasure} and the basic properties of H-measures, for any measurable set $E\subset S^*\mc M$,
$$
M\equiv 
\begin{bmatrix}
\tilde{\nu}(E) & \tilde \lambda(E)\\
\tilde \lambda^*(E) & \mu(E)
\end{bmatrix}
$$
is a positive semi-definite matrix and $\mu(E)\geq 0$, hence $\mu(E)\geq 0\implies \tilde\lambda(E)=0$.

To prove part \eqref{it:parity} we consider a real symbol $a(x,\xi)=b(x)m(\xi)$; the general case follows according to Remark \ref{remark:symbolform}. Suppose that $m$ is odd: then, using Plancherel's identity,
\begin{align*}
\langle Ae_0u_\e, e_0 u_\e\rangle
%& =\int [\hat b*(m\, \widehat{e_0u_\e})](\xi) \overline{\widehat{e_0u_\e}}(\xi) \d \xi\\
& =\iint \hat b(\xi-\eta) m(\eta)\widehat{e_0 u_\e}(\eta) \overline{\widehat{e_0u_\e}}(\xi)\d \xi \d\eta\\
& =- \iint  \hat b(\eta-\xi) m(\eta)\overline{\widehat{e_0 u_\e}}(\eta)\widehat{e_0 u_\e}(\xi) \d\xi\d\eta=-\langle Ae_0 u_\e, e_0u_\e\rangle\,,
\end{align*}
where in the last line we made the change of variables $(\xi,\eta)\mapsto -(\xi,\eta)$, used the fact that $m$ is odd and that all functions are real. Hence
$$\langle \nu, (\xi_0-\beta^i\xi_i)^2 a\rangle
=\lim_{\e\to 0} \langle Ae_0u_\e, e_0 u_\e\rangle = 0.
$$
Note that, by \eqref{eq:lightcone}, $\xi_0-\beta^i \xi_i$ never vanishes on the zero mass shell $\{g^{\a\beta}\xi_\a\xi_\beta=0\}$ where, according to \eqref{it:support}, $\nu$ is supported. Hence we have shown that $\langle \nu, a\rangle =0$ whenever $a$ is odd in $\xi$.  An identical argument for $\lambda$, which is also supported in the zero mass shell, concludes the proof.
\end{proof}

The proof of part (\ref{it:propagation}) is more involved but follows essentially the outline of \cite[Theorem 3.12]{Tartar1990}. The crucial technical ingredient is contained in the following lemma:

\begin{lemma}\label{lemma:commutator-wave}
Let $g$ be a smooth Lorentzian metric and take $A\in \Psi^0_{1,c}$. Then $[\Box_g,A]\in \Psi^1_1$ and
%$$\sigma^1([\Box_g,A])= i(2g^{\a\beta}\xi_\beta \p_{x^\a} a - \p_{x^\mu} g^{\a\beta}\xi_{\a}\xi_{\beta} \p_{\xi_\mu} a).$$
$$\sigma^1([\Box_g,A])= \sigma^1(i P^\a \D_\a)= \sigma^1( P^\a\p_\a)\,,$$
where $P^\a\in \Psi_1^0$ is such that
$$\sigma^0(P^\a) \equiv 2 g^{\a \beta} \p_{x^\beta} a -\p_{x^\mu} g^{\a\beta}\xi_\beta \p_{\xi^\mu} a\,.$$
\end{lemma}

Since $g$ is assumed to be smooth, Lemma \ref{lemma:commutator-wave} follows at once from the last part of Lemma \ref{lemma:expansions}. Nonetheless, the result still holds if $g\in C^1$, although this is much more difficult:

\begin{remark}\label{remark:smoothness} 
The Calderón Commutator (Theorem \ref{thm:Calderon}) shows that $[\Box_g, A]\colon H^1\to L^2$ is bounded, even when $g$ is just $C^1$, but this assumption cannot be substantially weakened, c.f.\ \cite[pages 336-337]{Tartar2010} and \cite{Coifman1976, Uchiyama1978}. 
\end{remark}

\begin{proof}[Proof of Theorem \ref{thm:linearwave}(\ref{it:propagation})]
Let us take $A\in \Psi^0_1$ to be a multiplier, so $a(x,\xi)\equiv m(\xi)$. We begin by applying $A$ and $\overline{A}$ to \eqref{eq:linear-wave-system} to get, respectively,
\begin{equation}
\label{eq:energyidentitiesA}
\Box_g (A u_\e) = A f_\e +[\Box_g,A] u_\e,
\qquad 
\Box_g (\overline{A u_\e}) = \overline{A f_\e} +\overline{[\Box_g,A] u_\e}\,.
\end{equation}
Given a smooth vector field $X$, we multiply the first equation by $X (\overline{A u_\e})$, the second equation by $X(Au_\e)$, and sum the two. Using the energy identity \eqref{eq:energyidentity} we get
\begin{equation}
\label{eq:auxtransport}
\begin{split}
& \nabla^\a J_\a^X[Au_\e,\overline{A u_\e}] - T_{\a\beta}[A u_\e,\overline{A u_\e}] \nabla^\a X^\beta
 \\
& \qquad = \frac 1 2 \lp[X(\overline{A u_\e}) Af_\e+X(Au_\e)\overline{A f_\e}\rp] + 
\frac 1 2\lp[ X(\overline{A u_\e}) [\Box_g, A]u_\e+X(Au_\e) \overline{[\Box_g,  A]u_\e}\rp].
\end{split}
\end{equation}
Now let $\varphi\in C^\infty_c(\R^{1+n})$ and integrate \eqref{eq:auxtransport} against $\varphi$ with respect to $\tp{d} \tp{Vol}_g$. We deal with each of the corresponding terms separately. 

\medskip
\noindent\textbf{Step 1:} the left hand side of \eqref{eq:auxtransport}. For the first term, we integrate by parts and recall \eqref{eq:defTJ}:
\begin{equation}
\label{eq:firsterm}
\begin{split}
\int &\nabla^\a J_\a^X[Au_\e,\overline{A u_\e}] \varphi \d\tp{Vol}_g\\
&=-\frac 1 2 \int\lp( X(Au_\e) \p_\a (\overline{A u_\e}) + X(\overline{A u_\e}) \p_\a (Au_\e)-X_\a g^{\beta\gamma}\p_\beta(Au_\e)\p_\gamma(\overline{A u_\e})\rp) \nabla^\a \varphi \d \tp{Vol}_g.
\end{split}
\end{equation}
Using the fact that $A$ is a multiplier and that $\sigma^0(A^*)=\sigma^0(A)^*$, we have
\begin{align*}
\lim_{\e\to 0}\int X(Au_\e)\p_\a (\overline{A u_\e}) \nabla^\a \varphi \d \tp{Vol}_g 
 & = \lim_{\e\to 0}\int A^*\bigg(A(\p_\beta u_\e) X^\beta  \nabla^\a \varphi \sqrt{|g|}\bigg) \p_\a u_\e \d x\\
& = \lim_{\e \to 0} \int A^* A(\p_\beta u_\e) \p_\a u_\e X^\beta \nabla^\a\varphi\sqrt{|g|}\d x\\
& =\langle \tilde\nu_{\a\beta},|m(\xi)|^2 X^\beta \nabla^\a \varphi\sqrt{|g|} \rangle
= \langle \nu,g^{\a\gamma} \xi_\a \xi_\beta|m(\xi)|^2 X^\beta \nabla_\gamma \varphi\sqrt{|g|} \rangle,
\end{align*}
where we also used the fact that $[A^*,X^\beta \nabla^\a\varphi \sqrt{|g|}]\colon L^2\to L^2$ is compact, c.f.\ Lemma \ref{lemma:continuity}.
The second term on the right-hand side of \eqref{eq:firsterm} is treated identically and has the same limit.
Finally, the last term on the right-hand side of \eqref{eq:firsterm} vanishes in the limit: indeed, arguing as before,
\begin{align*}
\lim_{\e \to 0}\int X_\a g^{\beta \gamma}\p_\beta(Au_\e)\p_\gamma(\overline{Au_\e})\n^\a\varphi \d\tp{Vol}_g
& =\langle \nu, X_\a g^{\beta\gamma} \xi_\beta \xi_\gamma |m(\xi)|^2 \nabla^\a \varphi \sqrt{|g|} \rangle =0\,,
\end{align*}
using the support condition on $\nu$. 
%Hence we have shown that
%\begin{equation*}
%\lim_{\e\to 0} \int\nabla^\a J_\a^X[Au_\e,A^*u_\e] \varphi \d\tp{Vol}_g
%=-\langle \nu, g^{\a\gamma}\xi_\a\xi_\beta |m(\xi)|^2 X^\beta \n_\gamma \varphi \sqrt{|g|}\rangle\,.
%\end{equation*}

For the second term in \eqref{eq:auxtransport}, similar arguments yield
\begin{align*}
\lim_{\e\to 0}\int T_{\a\beta}[Au_\e,\overline{Au_\e}]\n^\a X^\beta \varphi\d\tp{Vol}_g
&=
\langle \nu, (\xi_\a\xi_\beta \n^\a X^\beta -\frac 1 2 g_{\a\beta}g^{\mu\nu} \xi_\mu \xi_\nu)|m(\xi)|^2\varphi\sqrt{|g|}\rangle\\
 &= \langle \nu, g^{\a\gamma}\xi_\a\xi_\beta \n_\gamma X^\beta|m(\xi)|^2 \varphi\sqrt{|g|}\rangle.
\end{align*}
Setting $\Phi^\beta(x,\xi)\equiv X^\beta |m(\xi)|^2 \varphi(x)\sqrt{|g|}$, $\Phi\equiv\xi_\beta\Phi^\beta$, and using the fact that $\nabla_\gamma g=0$, we have calculated the limit of the left-hand side of \eqref{eq:auxtransport}:
$$
\lim_{\e\to 0}
\int \lp(\nabla^\a J_\a^X[Au_\e,\overline{Au_\e}]-T_{\a\beta}[Au_\e,\overline{Au_\e}]\n^\a X^\beta\rp)\varphi\d\tp{Vol}_g
=
-\langle \nu, g^{\a\gamma}\xi_\a  \xi_\beta \n_{\gamma}\Phi^\beta\rangle.$$

\medskip
\noindent\textbf{Step 2:} the right hand side of \eqref{eq:auxtransport}. For the first term we have
\begin{align*}
\lim_{\e\to 0} \frac 1 2\int\lp( X(\overline{Au_\e}) A f_\e + X(Au_\e) \overline{Af_\e}\rp)\varphi \d \tp{Vol}_g 
 & = \lim_{\e\to 0}\frac 1 2\int \lp(X u_\e A^*Af_\e + X(A^*Au_\e) f_\e\rp)\varphi \d\tp{Vol}_g\\
  &= \frac 1 2\langle X^\beta( \tilde \lambda_{\beta} +  \tilde \lambda_{\beta}^*),|m(\xi)|^2  \varphi\sqrt{|g|} \rangle
 = \langle \Re \lambda,\Phi\rangle \,.
\end{align*}
According to Lemma \ref{lemma:commutator-wave}, the last term yields
\begin{align*}
&\lim_{\e\to 0}\frac 1 2
\int \lp(X(\overline{Au_\e}) [\Box_g, A]u_\e+X(Au_\e) \overline{[\Box_g,  A]u_\e}\rp)\varphi\tp{\,dVol}_g\\
& \qquad = \lim_{\e\to 0} -\langle \nu,\p_{x^\mu} g^{\a\gamma} \xi_\a  \xi_\gamma \xi_\beta  X^\beta(m \,\p_{\xi_\mu} \overline m +\overline m\, \p_{\xi_\mu}m)\varphi \sqrt{|g|}\rangle=-\frac 1 2\langle \nu,\p_{x^\mu} g^{\a\gamma} \xi_\a \xi_\gamma \xi_\beta \p_{\xi_\mu}\Phi^\beta\rangle\,.
\end{align*}

\medskip
\noindent\textbf{Step 3:} putting everything together. Combining the last three computations we find that
$$\langle \nu, -g^{\a\gamma}\xi_\a(\xi_\beta\n_{x^\gamma} \Phi^\beta)
+\frac 1 2 \p_{x^\mu} g^{\a\gamma}\xi_\a\xi_\gamma(\xi_\beta \p_{\xi_\mu}\Phi^\beta)\rangle = \langle\Re\lambda,\xi_0\Phi\rangle.$$
The left-hand side can be simplified further: note that, as $\nabla_\mu$ is the Levi-Civita connection,
$$0=\nabla_\mu g^{\a\gamma} = \p_{x^\mu} g^{\a\gamma} +\delta^\a_\beta\Gamma^\beta_{\mu\delta} g^{\delta\gamma} +\delta^\gamma_\beta\Gamma^\beta_{\mu\delta} g^{\a\delta}\quad \implies \quad
\frac 1 2 \p_{x^\mu} g^{\a\gamma} \xi_\a\xi_\gamma +g^{\a\gamma}\xi_\a\xi_\beta \Gamma^{\beta}_{\gamma\mu}=0.$$
Combining this identity with the two equations
$$\p_{\xi_\mu} \Phi = \xi_\beta \p_{\xi_\mu} \Phi^\beta + \Phi^\mu,
\qquad
\p_{x^\gamma} \Phi = \xi_\beta \n_{x^\gamma} \Phi^\beta - \Gamma^{\beta}_{\gamma\mu} \Phi^\mu \xi_\beta$$
we find that 
\begin{equation*}
\label{eq:claim}
-g^{\a\gamma}\xi_\a(\xi_\beta\n_{x^\gamma} \Phi^\beta)
+\frac 1 2 \p_{x^\mu} g^{\a\gamma}\xi_\a\xi_\gamma(\xi_\beta \p_{\xi_\mu}\Phi^\beta)
=-g^{\a\gamma}\xi_\a\p_{x^\gamma} \Phi+\frac 1 2\p_{x^\mu} g^{\a\gamma} \xi_\a\xi_\gamma\p_{\xi_\mu} \Phi
=-\{g^{\a\beta}\xi_\a\xi_\beta,\Phi\}
 \,.
\end{equation*}
While the previous calculations hold for an arbitrary vector field $X$, we now take $X=e_0$, so that $X^\beta\xi_\beta = \xi_0 - \beta^i\xi_i$. As before we note that $\xi_0-\beta^i \xi_i$ never vanishes on the zero mass shell, where $\nu$ is supported. Hence, we have shown that part (\ref{it:propagation}) of the theorem holds whenever $\tilde a$ is of the form $\tilde a(x,\xi)=\Phi(x,\xi)=b(x)q(\xi)$ with $q$ real and positively 1-homogeneous. The case of a general test function follows by considerations analogous to the ones in Remark \ref{remark:symbolform}.
\end{proof}

\subsection{Two compensated compactness lemmas}
\label{sec:wave-CC}

This subsection contains two compensated compactness results for solutions of the wave system \eqref{eq:linear-wave-system} which follow readily from the very classical Theorem \ref{thm:p-q-divcurl}. We begin with a \textit{bilinear} result:

\begin{lemma}\label{lemma:compensatedcompactnessWave}
Null forms are weakly continuous, i.e.\
$$
\text{for } I=1,2,\quad
\begin{rcases}
u_\e^{I}\w u^I \text{ in } W_\tp{loc}^{1,2}\\
(\Box_g u_\e^I)_\e \text{ is compact in } W^{-1,2}_\tp{loc}
\end{rcases} 
\implies 
g^{-1}(\tp d u_\e^1, \d u_\e^2) \wstar g^{-1}(\tp d u^1, \d u^2) \tp{ in } \mathscr D'.
$$
\end{lemma}

\begin{proof}
It suffices to consider the case $u_\e^1 = u_\e^2$: indeed, one can use the polarization identity
\begin{align*}
g^{-1}(\tp d u_\e^1, \d u_\e^1) + 2 g^{-1}(\tp d u_\e^1, \d u_\e^2)+ g^{-1}(\tp d u_\e^2, \d u_\e^2)=
g^{-1}(\tp d u_\e^1+\d u_\e^2, \d u_\e^1+\d u_\e^2)
\end{align*}
and pass to the limit on both sides to see that $g^{-1}(\tp d u_\e^1, \d u_\e^2)\wstar g^{-1}(\tp d u^1, \d u^2)$ in the sense of distributions.
We thus drop all superscripts from the sequences.

Let $\star$ be the Hodge star with respect to the metric $g^{-1}$. We have
$$g^{-1}(\tp d u_\e,\d u_\e) \d \tp{Vol}_{g^{-1}}= \d u_\e\wedge (\star \d  u _\e)$$
and, since $u_\e$ is scalar, $\Box_g u_\e = \star \d \star \tp{d} u_\e$. The conclusion follows from Theorem \ref{thm:p-q-divcurl}.
\end{proof}

The next result is \textit{trilinear} and was essentially known to Tartar: see \cite[Lemma I.5]{Tartar2005}, where it is proved when $g$ is the Minkowski metric. The proof given below is the natural adaptation of Tartar's proof, now in the language of geometric wave equations introduced at the beginning of the section. See also  \cite[Proposition 12.2]{Huneau2019} for an alternative proof.

\begin{lemma}\label{lemma:trilinear}
Let $X$ be a smooth vector field. Then 
$$
\text{for } I=1,2,3,\quad \begin{rcases}
u^I_\e\w 0 \text{ in } W^{1,3}_\tp{loc}\\
(\Box_g u^I_\e)_\e \text{ is bounded in } L^3_\tp{loc}
\end{rcases}
\implies 
X u^1_\e\, g^{-1}(\tp d u^2_\e,\d u^3_\e)\wstar 0 \quad \tp{ in } \mathscr D'.$$
\end{lemma}

\begin{proof}
The assumptions imply that the sequence $J^X_\a[u^1_\e,u^2_\e]$ is bounded in $L^{3/2}_\tp{loc}$ and, recalling \eqref{eq:energyidentity}, that $\nabla^\a J^X_\a[u^1_\e,u^2_\e]$  is compact in $W^{-1,3/2}_\tp{loc}$. We note that  
$$
2 J_\a^X[u^1, u^2] \p_\beta u^3\, g^{\a \beta}
=
X u^1 \, g^{-1}(\tp{d} u^2, \d u^3)+
X u^2 \, g^{-1}(\tp{d} u^1, \d u^3)-
X u^3 \, g^{-1}(\tp{d} u^1, \d u^2)\,,
$$
where the left-hand side is a div-curl product. Using the polarization identity, as in Lemma \ref{lemma:compensatedcompactnessWave}, to prove the conclusion we can take $u^2_\e=u^3_\e$ without loss of generality. Thus
$$
2 J_\a^X[u^1_\e, u^2_\e] \p_\beta u^2_\e\, g^{\a \beta}=X u^1_\e \, g^{-1}(\tp{d} u^2_\e, \d u^2_\e)
$$
or, equivalently, writing again $\star$ for the Hodge star with respect to $g^{-1}$,
$$
X u^1_\e \, g^{-1}(\tp{d} u^2_\e, \d u^2_\e) \d \tp{Vol}_{g^{-1}}
=
g^{-1}(2J^X[u^1_\e,u^2_\e], \d u^2_\e) \d \tp{Vol}_{g^{-1}} =
 2J^X[u_\e^1, u_\e^2] \wedge \star \d u^2 _\e\,.$$
Since $\d u^2_\e\w 0$ in $L^3_\tp{loc}$, we can again use Theorem \ref{thm:p-q-divcurl} to pass to the limit.
\end{proof}

\begin{remark}\label{remark:constantrank}
Taking $u^1_\e=u^2_\e=u^3_\e$ in Lemma \ref{lemma:trilinear}, we note that the trilinear quantity is weakly continuous \textit{solely at zero}.
That this happens is only possible  because $\Box_g$, thought of as a first-order operator acting on $\tp d u$, does not have constant rank. Indeed, it is shown in \cite{Guerra2019} that, under constant rank constraints, nonlinearities which are weakly continuous at a point are necessarily weakly continuous \textit{everywhere}. Furthermore, regardless of rank conditions, nonlinearities which are weakly continuous everywhere are polynomials with degree not exceeding the dimension of the domain, i.e.\ $n+1$, see also \cite{Murat1981}. In contrast,  Lemma \ref{lemma:trilinear} is of course valid even when $n=1$. See also \cite{Joly1995,Muller2000} for other trilinear Compensated Compactness results without constant rank assumptions.
\end{remark}

\section{The linear covariant wave equation with oscillating coefficients}
\label{sec:linear-wave-oscillating}

This section is devoted to the proof of Theorem \ref{thm:linear-wave-oscillating}. Our strategy is to reduce the analysis of the limiting behavior of sequences of solutions to
\begin{equation}
\Box_{g_\e} u_\e =f_\e\,,  \qquad u_\e,f_\e \colon (0,T)\times\mathbb{R}^n\to \R\, \label{eq:linear-wave-system-2}
\end{equation}
to the case where $(u_\e,f_\e)$ are solutions of a fixed wave equation, as in the previous section. 
Hence, we will frequently recast \eqref{eq:linear-wave-system-2} in the form of \eqref{eq:linear-wave-system}, i.e.\
\begin{align}
\Box_{g}u_\varepsilon= \lp(\Box_{g}-\Box_{g_\varepsilon}\rp)u_\varepsilon+f_\e\,, \label{eq:reduction-linear-to-nonoscillating}
\end{align}
Note that, by Hypothesis \ref{hyp:main}\eqref{it:hypsols},
\begin{equation}
\label{eq:bddbox}
\Box_g u_\e \tp{ is uniformly bounded in } L^4_\tp{loc}.
\end{equation}

We begin by noting that part (\ref{it:lin-wave-eq}) of Theorem \ref{thm:linear-wave-oscillating} poses no difficulty, as the covariant wave operator is an operator in divergence form. For later use, we state the result explicitly:

\begin{proposition}[Limit equation] \label{prop:limit-equation-LinWaveOsc}
Let $(g_\e,u_\e,f_\e)_\e$ be a sequence satisfying Hypotheses~\ref{hyp:main} and solving \eqref{eq:linear-wave-system-2}. Then $\Box_g u = f$ in the sense of distributions.
\end{proposition}

\begin{proof}
Note that $\Box_{g_\e} u_\e\wstar \Box_{g} u$ in $\mathscr D'$, and hence by Hypotheses \ref{hyp:main}\eqref{it:hypsource} also weakly in $L^2_\tp{loc}$. Indeed, take a test function $\varphi$; then, using the local uniform convergence of $g_\e$ and integrating by parts, we find that
\begin{align*}
\lim_{\e \to 0} \int \varphi \,\Box_{g_\e} u_\e \d\tp{Vol}_{g} 
& = \lim_{\e \to 0} \int \varphi \,\Box_{g_\e} u_\e \d\tp{Vol}_{g_\e}\\
& =- \lim_{\e \to 0} \int\p_\a \varphi\, \p_\beta u_\e \, g_\e^{\a\beta}\d \tp{Vol}_{g_\e}
=  -\int \p_\a \varphi\, \p_\beta u\, g^{\a\beta}\d \tp{Vol}_{g}
=\int \varphi\, \Box_{g} u\, \d \tp{Vol}_{g},
\end{align*}
since we have the product of weakly convergent terms with strongly convergent ones. Recall that $\d \tp{Vol}_{g}\equiv \sqrt{|\det g|} \d x$ and \eqref{eq:covariant-wave-op}. As $\Box_{g_\e} u_\e = f_\e \w f$ in $L^2_\tp{loc}$, by uniqueness of limits we see that $\Box_g u = f$.
\end{proof}

For part (\ref{it:thm-lin-wave-energy-density}), our starting point is identity \eqref{eq:reduction-linear-to-nonoscillating}. We set
\begin{align*}
h_\e^{\alpha\beta}\equiv g_\e^{\alpha\beta}-g^{\alpha\beta}\,, \qquad H_\e \equiv \lp(\Box_{g_\e}-\Box_{g}\rp)u_\varepsilon\,;
\end{align*}
by \eqref{eq:bddbox} and Proposition~\ref{prop:limit-equation-LinWaveOsc}, $H_\e$ converges weakly in $L^2_\tp{loc}$ to zero. Besides the H-measures defined in \eqref{eq:defHmeasureintro}, we will need the H-measure generated when $\tp d u_\e$ is combined with the right-hand side in \eqref{eq:reduction-linear-to-nonoscillating}:
$$
\lp(\p_0 u_\e,\p_1 u_\e, \dots, \p_n u_\e, H_\e + f_\e\rp)_\e \wH 
\begin{bmatrix}
\tilde \nu & \tilde \sigma \\
\tilde \sigma^* & \star
\end{bmatrix}
\,.
$$

\subsection{Elementary reductions}
\label{sec:oscillating-wave-elementary-reductions}

Before proceeding with the core of the proof, we make a few basic observations. Firstly, both the structure of the H-measure and the localization part of Theorem \ref{thm:linear-wave-oscillating}\eqref{it:thm-lin-wave-energy-density} follow  as in Section \ref{sec:linear-wave-non-oscillating} since, by \eqref{eq:bddbox}, $\Box_g u_\e$ is bounded in $L^2_\tp{loc}$. Likewise, $\tilde \sigma_\gamma = \xi_\gamma \sigma$ for some Radon measure $\sigma$.
Moreover, arguing once more as in Section \ref{sec:linear-wave-non-oscillating}, we can and will assume that the sequence $u_\e$ is supported  on a fixed bounded set $\Omega$. Hence we can and will also assume that $g_\e=g$ for all $\e$, outside a neighborhood of $\Omega$. 

The final remark that we make here concerns the parity in $\xi$ of equation \eqref{eq:propagationprop}: according to the parity of $\nu$ and $\lambda$, established in Theorem \ref{thm:linearwave}, we  only need to test \eqref{eq:propagationprop} against 1-homogeneous functions $\tilde a$ which are odd in $\xi$, which corresponds to testing against symbols $a$ which are 0-homogeneous and even in $\xi$. In particular, in the rest of the proof we will use implicitly the following straightforward lemma:

\begin{lemma}\label{lemma:hypsymbol}
For $A\in \Psi^0$ such that
\begin{equation}
\label{eq:hypsymbol}
\sigma^0(A)(x,\xi) \text{ is real and even in } \xi\,,
\end{equation} 
$A\varphi$ and $A^*\varphi$ are real whenever $\varphi\in L^2$ is real.
\end{lemma}

%\begin{proof}
%The proof is straightforward: 
%$$\mc F(A\varphi)(\xi)=a(x,\xi) \hat \varphi(\xi)
%=a(x,-\xi) \overline{\hat \varphi(-\xi)}
%=\overline{a(x,-\xi) \hat \varphi(-\xi)}=\overline{\mc F(A\varphi)(-\xi)},
%$$
%hence $A\varphi$ is real.
%\end{proof}

Due to Theorem~\ref{thm:linearwave} our task is to show that, as $\e\to 0$, $H_\e$ does not contribute to the transport equation.

\subsection{A warm-up: the case of strong convergence of the metrics}
\label{sec:easycase}

In this section we show that if we knew that $g_\e\to g$ \textit{strongly} in $W^{1,\infty}_\tp{loc}$ then Theorem~\ref{thm:linear-wave-oscillating} would follow easily. The first step is a reduction to estimating some commutators. The basic idea is to integrate by parts in order to try to distribute the derivatives in such a way that two derivatives do not land on the same term; this cannot be achieved completely, but the remaining terms have a commutator structure.

\begin{lemma} \label{lemma:commutator-reduction-easy} Let $A\in\Psi_{1,c}^0$ satisfy \eqref{eq:hypsymbol}. If $g_\e\to g$ in $W^{1,\infty}_\tp{loc}$ and Hypotheses \ref{hyp:main}\eqref{it:hypsols} hold, then
$$2\lim_{\e\to 0}\langle H_\e, A \p_\gamma (u_\e-u)\rangle = 
\lim_{\e\to 0} \int \p_\gamma (u_\e-u) \lp[A, h_\e^{\alpha\beta}\rp]\p_{\alpha\beta}^2(u_\e-u) \d x\,.$$ 
\end{lemma}
\begin{proof} Since derivatives of the metric coefficients converge strongly,
$H_\e = h_\e^{\alpha\beta}\p_{\alpha\beta}^2u_\e+o_{L^2}(1)$, where $o_{L^2}(1)$ denotes a remainder which is compact in $L^2$. We begin by noting that
\begin{align*}
\langle  h_\e^{\alpha\beta}\p_{\alpha\beta}^2u_\e, A \p_\gamma (u_\e-u)\rangle &=\langle  h_\e^{\alpha\beta}\p_{\alpha\beta}^2(u_\e-u), A \p_\gamma (u_\e-u)\rangle +\langle  h_\e^{\alpha\beta}\p_{\alpha\beta}^2u, A \p_\gamma (u_\e-u)\rangle\\
&=\langle  h_\e^{\alpha\beta}\p_{\alpha\beta}^2(u_\e-u), A \p_\gamma (u_\e-u)\rangle+o(1)\,.
\end{align*}
Now, we evaluate the remaining term, setting $w_\e\equiv u_\e-u$. First, we integrate by parts in $\p_\alpha$:
\begin{align*} 
\langle  h_\e^{\alpha\beta}\p_{\alpha\beta}^2w_\e, A \p_\gamma w_\e \rangle &= 
%\p_\alpha\lp[ h_\e^{\alpha\beta} \p_\beta A \p_\gamma w_\e\rp]
-\langle \p_\alpha h_\e^{\alpha\beta} \p_\beta w_\e, A \p_\gamma w_\e\rangle 
-\langle h_\e^{\alpha\beta}\p_\beta w_\e,[\p_\alpha,A]\p_\gamma w_\e\rangle
-\langle  h_\e^{\alpha\beta}\p_\beta w_\e, A( \p_{\gamma\alpha}^2 w_\e)\rangle\\
&= -\langle  h_\e^{\alpha\beta}\p_\beta w_\e, A( \p_{\gamma\alpha}^2 w_\e)\rangle +o(1)\,.
\end{align*}
Then, we integrate the remaining term by parts along $\p_\gamma$, and obtain
\begin{align*} 
-\langle  h_\e^{\alpha\beta}\p_\beta w_\e, A\lp(\p_\alpha \p_\gamma w_\e\rp)\rangle &= 
-\langle  h_\e^{\alpha\beta}\p_\beta w_\e, [A,\p_\gamma]\p_\alpha w_\e\rangle 
%-\p_\gamma\lp[ h_\e^{\alpha\beta}\p_\beta w_\e A \p_\alpha w_\e\rp]
+\langle \p_\gamma  h_\e^{\alpha\beta}\p_\beta w_\e, A \p_\alpha w_\e\rangle  
+\langle  h_\e^{\alpha\beta}\p_\beta \p_\gamma w_\e, A \p_\alpha w_\e \rangle\\
&= \langle  h_\e^{\alpha\beta}\p_\beta \p_\gamma w_\e, A \p_\alpha w_\e \rangle+o(1)\,.
\end{align*}
Finally, integrating the remaining term along $\p_\beta $, 
\begin{align*}
\langle  h_\e^{\alpha\beta} \p_\beta \p_\gamma w_\e, A\p_\alpha w_\e\rangle &= 
%\p_\beta\lp[ h_\e^{\alpha\beta} \p_\gamma w_\e A\p_\alpha w_\e\rp]
-\langle \p_\beta h_\e^{\alpha\beta} \p_\gamma w_\e, A\p_\alpha w_\e\rangle 
-\langle  h_\e^{\alpha\beta} \p_\gamma w_\e, [\p_\beta,A]\p_\alpha w_\e\rangle
- \langle  h_\e^{\alpha\beta} \p_\gamma w_\e, A \p_{\alpha\beta}^2 w_\e\rangle\\
&=  \langle \p_\gamma w_\e, [A, h_\e^{\alpha\beta}]\p_{\alpha\beta}^2 w_\e\rangle 
- \langle \p_\gamma w_\e, A \lp( h_\e^{\alpha\beta}\p_{\alpha\beta}^2 w_\e\rp)\rangle+o(1)\,. 
\end{align*}
Combining the expressions above yields the identity:
\begin{align*}
\langle  h_\e^{\alpha\beta}\p_{\alpha\beta}^2w_\e, A \p_\gamma w_\e \rangle+\langle \p_\gamma w_\e, A \lp( h_\e^{\alpha\beta}\p_{\alpha\beta}^2 w_\e\rp)\rangle = \langle \p_\gamma w_\e, [A, h_\e^{\alpha\beta}]\p_{\alpha\beta}^2 w_\e\rangle  +o(1)\,.
\end{align*}
Since $A$ has real symbol and hence is self-adjoint, up to a compact operator, we conclude the proof.
\end{proof}

Due to our strong-convergence assumptions, the Calderón commutator immediately yields:
\begin{proposition} For all $ A\in \Psi^0_{1,c}$ satisfying \eqref{eq:hypsymbol},
if $g_\e\to g$ in $W^{1,\infty}_\tp{loc}$ and Hypotheses \ref{hyp:main}\eqref{it:hypsols} hold, then
 \label{prop:case-i-no-metric-osc-contribution}
 $$\lim_{\e\to 0}\langle H_\e, A\p_\gamma (u_\e-u)\rangle = 0\,.$$
\end{proposition}

\begin{proof}
We need only observe that
$$[A,h_\e^{\a\beta}]\p^2_{\a\beta} (u_\e-u) = [A\p_\a,h_\e^{\a\beta}] \p_\beta (u_\e-u) + o_{L^2}(1).$$
By Theorem \ref{thm:Calderon} and the fact that $\Vert \p_\a h_\e^{\a\beta}\Vert_{L^\infty}\to 0$, the $L^2$-norm of the commutator on the right-hand side goes to zero. It now suffices to appeal to Lemma \ref{lemma:commutator-reduction-easy} and use H\"older's inequality.
\end{proof}

\subsection{The general case}

The remainder of this section deals with the more complicated case where we do not know that $g_\e\to g$ strongly in $W^{1,\infty}_\tp{loc}$. Similarly to the previous subsection, we are required to establish the following:

\begin{proposition} \label{prop:case-ii-no-metric-osc-contribution}
 Under Hypotheses~\ref{hyp:main}, for all $ A\in \Psi^0_{1,c}$ satisfying \eqref{eq:hypsymbol},
 $$\lim_{\e\to 0}\langle H_\e, A e_0 (u_\e-u)\rangle = 0\,.$$
\end{proposition}

Before proceeding further, let us outline the proof of Proposition \ref{prop:case-ii-no-metric-osc-contribution}:
\begin{itemize}[itemsep=0pt]
\item as in Section \ref{sec:easycase}, we start by reducing Proposition~\ref{prop:case-ii-no-metric-osc-contribution} to estimating a commutator (Lemma \ref{lemma:commutator-reduction});
\item we then estimate this commutator by choosing an $\e$-dependent partition of Fourier space and estimating each regime independently (Lemmas \ref{lemma:case-ii-no-metric-osc-contribution-lowfreq} to \ref{lemma:case-ii-no-metric-osc-contribution-timefreq}).
\end{itemize}

In what follows, we use the frame introduced in \eqref{eq:Cauchyframe}, in which we have
\begin{align}
\begin{split}
\Box_g&=g^{00}e_0^2+\tilde{g}^{ij}\p^2_{ij} + \frac12 \p_0 g^{00}e_0+\p_i g^{i\beta}\p_\beta+ \frac{g^{i\beta}\p_i\sqrt{|g|}}{\sqrt{|g|}} \p_\beta -g^{0i}\p_i\lp(\frac{g^{0j}}{g^{00}}\rp)\p_j -\frac{(g^{00})^2}{2 q(\tilde{g}^{ij})}\p_0 \lp(q(\tilde{g}^{ij})\rp) e_0\,,
\end{split}\label{eq:Box-alternative-identity}
\end{align}
where $q=q(\tilde{g}^{ij})$ denotes the polynomial in $\tilde{g}^{ij}$ determined implicitly by $|g^{-1}| = -g^{00}q(\tilde{g}^{ij})$ (the existence of such a polynomial is readily verified by considering the LDU decomposition of the matrix-field $g$).

Under Hypotheses~\ref{hyp:main}, general first derivatives of the metric coefficients do not converge strongly; however, \textit{spatial} first derivatives of the metric coefficients do: since we assume that $g_\e=g$ outside a neighborhood of $\Omega$, by integration by parts and our hypotheses,
\begin{equation}
\lVert \p_k h_\e^{\a\beta}\rVert_{L^2}
\lesssim  
\Vert \p_{kk}^2 h_\e^{\a\beta}\Vert_{L^2}^\frac 1 2 
\Vert h_\e^{\a\beta}\Vert_{L^2}^\frac 1 2
\lesssim \Vert h_\e^{\a\beta}\Vert_{L^2}^\frac 1 2\,, \label{eq:convergence-spatial-derivatives}
\end{equation}
which converges to zero. We recall that $\p_0 (g_\e)_{ij}$, and hence, $\p_0 \tilde{g}_\e^{ij}$ (one may check that $\tilde{g}_\e^{ij}$ is the inverse of the Riemannian metric $(g_\e)_{ij}$) also converge strongly.
%\begin{remark} \red{About how the \eqref{eq:convergence-spatial-derivatives} only makes sense with compact support} 
%\end{remark}
It is now easy to see that, under our assumptions, the last four terms in  \eqref{eq:Box-alternative-identity} only involve strongly converging derivatives of the metric coefficients.

 The proof of the next lemma follows the strategy used for Lemma \ref{lemma:commutator-reduction-easy}, but it is much more involved:

\begin{lemma}[Reduction to commutators]\label{lemma:commutator-reduction}
 Under Hypotheses~\ref{hyp:main}, let $A\in\Psi_{1,c}^0$ satisfy \eqref{eq:hypsymbol}. Then,%
$$2\lim_{\e\to 0}\langle H_\e, A e_0 (u_\e-u)\rangle = 
\lim_{\e\to 0} \int \p_\alpha (u_\e-u) [A, h_\e^{\alpha\beta}]\p_\beta e_0 (u_\e-u) \d x
\,.$$ 
%where we have defined $h_\e^{\a\beta}\equiv g_\e^{\a\beta}-g^{\a\beta}$.
%\begin{gather*}
%h_\e^{00}\equiv g_\e^{00}-g^{00}\,, \qquad h_\e^{0i}\equiv g^{0i}-g_\e^{0i}\,,\qquad 
% h_\e^{ij}\equiv g^{ij}_\e-g^{ij}\,.
%\end{gather*}
\end{lemma}

\begin{proof}Let us denote
\begin{gather}
\tilde h^{00}_\e\equiv g_\e^{00}-g^{00}\,, \qquad \tilde h^{0i}_\e\equiv - g_\e^{00}(\beta^i_\e-\beta^i)\,,\qquad 
\tilde h^{ij}_\e\equiv g_\e^{00}(\beta_\e^i -\beta^i)(\beta^j_\e-\beta^j) + \tilde{g}_\e^{ij}-\tilde{g}^{ij}\,.
\label{eq:defhtilda}
\end{gather}
From \eqref{eq:Box-alternative-identity},  we compute
\begin{align*}
H_\e&=\lp(\Box_{g_\e}-\Box_g\rp) u_\e
= \tilde h_\e^{ij}\p_{ij}^2u_\e +\tilde h^{00}_\e e_0^2u_\e+ \frac12 e_0 \tilde h^{00}_\e e_0u_\e + \tilde h_\e^{0i}e_0\p_iu_\e + e_0(\tilde h_\e^{0i}\p_iu_\e) + o_{L^2}(1)\\
&= \tilde h_\e^{ij}\p_{ij}^2(u_\e-u) +\lp[\tilde h^{00}_\e e_0+ \frac12 e_0 \tilde h^{00}_\e \rp]e_0(u_\e-u) + \tilde h_\e^{0i}e_0\p_i(u_\e-u) + e_0(\tilde h_\e^{0i}\p_i(u_\e-u)) \numberthis \label{eq:Box-difference-alternative-identity-1}\\
&\qquad+e_0 \tilde h_\e^{00}e_0 u +e_0 \tilde h_\e^{0i}\p_i u + o_{L^2}(1)\,, \numberthis \label{eq:Box-difference-alternative-identity-2}
\end{align*}
where $o_{L^2}(1)$ denotes a remainder which is strongly converging in $L^2$. The proof now proceeds in several steps. Step 5 deals with \eqref{eq:Box-difference-alternative-identity-2}. In steps 1 through 4, we deal with \eqref{eq:Box-difference-alternative-identity-1} and we set $w_\e\equiv u_\e-u$ to simplify the notation. We will also find it convenient to note the following identities for $\e\to 0$: letting $\mathfrak{h}_\e$ be a suitably regular function with $\mathfrak{h}_\e \to 0$ in $L^\infty$ and $\mathfrak{h}_\e\p^2_{\alpha\beta}u_\e$ uniformly bounded in $L^2$, and using Lemma \ref{lemma:continuity},
\begin{align}
\begin{split}
\label{eq:commutator-beta-with-h-right}
\langle \p_\gamma u_\e, [A, \mathfrak{h}_\e](\beta^k \p_k e_0u_\e)\rangle &= \langle \p_\gamma u_\e, [A, \beta^k  \mathfrak{h}_\e]\p_k e_0u_\e\rangle -\langle \p_\gamma u_\e \mathfrak{h}_\e,[\beta^k, A\p_k]e_0u_\e+A(\p_k\beta^k e_0u_\e)\rangle\\
&=\langle \p_\gamma u_\e, [A, \beta^k  \mathfrak{h}_\e]\p_k e_0u_\e\rangle +o(1)\,, 
\end{split}\\
\begin{split}
\langle \beta^k\p_k u_\e, [A, \mathfrak{h}_\e]\p_\gamma e_0u_e \rangle &= \langle \p_k u_\e, [A, \beta^k  \mathfrak{h}_\e]\p_\gamma e_0u_\e\rangle +\langle \p_k u_\e ,[\beta^k, A](\mathfrak{h}_\e\p_\gamma e_0u_\e)\rangle\\
&= \langle \p_k u_\e, [A, \beta^k  \mathfrak{h}_\e]\p_\gamma e_0u_\e\rangle +o(1)\,.
\label{eq:commutator-beta-with-h-left}
\end{split}
\end{align}

\medskip
\noindent\textbf{Step 1:} first term in \eqref{eq:Box-difference-alternative-identity-1}. By its special structure, both time and spatial derivatives of $\tilde{h}^{ij}_\e$ converge strongly. Hence, by a straightforward adaptation of the proof of Lemma~\ref{lemma:commutator-reduction-easy}, we find that 
%the only contribution of this term is through
\begin{align*}
\lim_{\e\to 0}\left\langle \tilde{h}_\e^{ij}\p_{ij}^2 w_\e, A e_0 w_\e\right\rangle = 
\lim_{\e \to 0}\frac 1 2\int e_0 w_\e [A,\tilde{h}_\e^{ij}]\p_{ij}^2 w_\e \d x\,. 
\end{align*}
Integrating by parts in $e_0$ and in $\p_j$ and using the compactness of  derivatives of $\tilde h^{ij}_\e$, we get
\begin{align*}
\lim_{\e \to 0}\frac 1 2\int e_0 w_\e [A,\tilde h_\e^{ij}]\p_{ij}^2 w_\e \d x = \lim_{\e \to 0}\frac 1 2\int \p_i w_\e [A, \tilde h_\e^{ij}]\p_{j}e_0 w_\e \d x\,.
\end{align*}

\medskip
\noindent \textbf{Step 2:} second term in \eqref{eq:Box-difference-alternative-identity-1}. An integration by parts in $e_0$ (which requires both an integration by parts in $\p_t$ and in a spatial direction) leads to
\begin{align*}
\langle \tilde h_\e^{00}e_0^2 w_\e, A e_0 w_\e \rangle
&= 
\langle \p_k\beta^k \tilde  h_\e^{00}e_0 w_\e, A e_0w_\e\rangle 
-\langle \tilde h_\e^{00}e_0 w_\e, [e_0,A]e_0w_\e\rangle\\
&\qquad -\langle e_0 \tilde h_\e^{00}e_0 w_\e, A e_0 w_\e \rangle
 -\langle \tilde h_\e^{00}e_0 w_\e, A e_0^2 w_\e \rangle\,.
\end{align*}
Thus, we have the identity:
\begin{align*}
\langle \tilde h_\e^{00}e_0^2 w_\e, A e_0 w_\e \rangle+\langle e_0 w_\e, A (\tilde h_\e^{00} e_0^2 w_\e) \rangle+\langle e_0 \tilde h_\e^{00}e_0 w_\e, A e_0 w_\e \rangle=\langle e_0 w_\e, [A,\tilde h_\e^{00}] e_0^2 w_\e\rangle  + o(1)\,.
\end{align*}
Using the self-adjointness of $A$ on the second term on the left hand side, we conclude that
\begin{align*}
\lim_{\e\to 0}\langle 2\tilde h_\e^{00}e_0^2 w_\e+ e_0h_\e^{00}e_0 w_\e, Ae_0 w_\e\rangle &= \lim_{\e\to 0} \int  e_0 w_\e [A,\tilde h_\e^{00}] e_0^2 w_\e  \d x\\
&=\lim_{\e\to 0} \int  \lp(\p_0 w_\e [A,\tilde h_\e^{00}] \p_0e_0 w_\e + \p_i w_\e [A,\tilde h_\e^{00}\beta^i\beta^j] \p_j e_0 w_\e\rp)  \d x\\
&\qquad-\lim_{\e\to 0} \int  \lp(\p_i w_\e [A,\beta^i \tilde h_\e^{00}] \p_0e_0 w_\e + \p_0 w_\e [A,\tilde h_\e^{00}\beta^i] \p_i e_0 w_\e\rp)  \d x\,,
\end{align*}
where we have applied both \eqref{eq:commutator-beta-with-h-right} and \eqref{eq:commutator-beta-with-h-left} to conclude.

\medskip
\noindent \textbf{Step 3:} third term in \eqref{eq:Box-difference-alternative-identity-1}.  An integration by parts in $\p_i$ leads to
\begin{align*}
\langle \tilde h_\e^{0i} e_0\p_i w_\e, Ae_0 w_\e\rangle &= 
\langle \tilde h_\e^{0i} [e_0,\p_i] w_\e, Ae_0 w_\e\rangle
-\langle \p_i \tilde h_\e^{0i} e_0 w_\e, A e_0 w_\e\rangle 
-\langle h_\e^{0i} e_0 w_\e, [\p_i,A] e_0 w_\e\rangle \\
&\qquad -\langle \tilde h_\e^{0i} e_0 w_\e, A[\p_i,e_0] w_\e\rangle 
+ \langle e_0 w_\e, [A, \tilde h_\e^{0i}] e_0\p_i w_\e\rangle 
- \langle e_0 w_\e, A(\tilde h_\e^{0i} e_0\p_i w_\e)\rangle\,.
\end{align*}
Thus, we have the identity
\begin{align*}
\langle \tilde h_\e^{0i} e_0\p_i w_\e, Ae_0 w_\e \rangle +\langle e_0 w_\e, A(\tilde h_\e^{0i} e_0\p_i w_\e)\rangle = \langle e_0 w_\e [A,\tilde h_\e^{0i}] e_0\p_i w_\e\rangle +o(1)\,.
\end{align*}
Using the self-adjointness of $A$, modulo a compact operator, \eqref{eq:commutator-beta-with-h-left}, and interchanging $e_0$ with $\p_i$, we conclude
\begin{align*}
\lim_{\e\to 0}\left\langle \tilde h_\e^{0i}e_0\p_i w_\e, A e_0w_\e\right\rangle &= \lim_{\e\to 0}\frac12 \int  e_0 w_\e [A,\tilde h_\e^{0i}] \p_i e_0 w_\e  \d x\\
&=\lim_{\e\to 0}\frac12 \int  \p_0 w_\e [A,\tilde h_\e^{0i}] \p_i e_0 w_\e  \d x
-\lim_{\e\to 0}\frac12 \int  \p_i w_\e [A,\beta^i\tilde h_\e^{0j}] \p_j e_0 w_\e  \d x\,.
\end{align*}

\medskip
\noindent \textbf{Step 4:} fourth term in \eqref{eq:Box-difference-alternative-identity-1}.  To begin, recall that, for any $\mathfrak{h}_\e \to 0$ in $L^\infty$, by \eqref{eq:Box-alternative-identity},
\begin{align}
\mathfrak{h}_\e e_0^2 w_\e + \mathfrak{h}_\e\frac{\tilde{g}^{jk}}{g^{00}}\p_{jk}^2w_\e=  \mathfrak{h}_\e\frac{\Box_g w_\e}{g^{00}} +o_{L^2}(1) = o_{L^2}(1)\,, \label{eq:exchange-e0-i}
\end{align}
since $\Box_g w_\e$ is bounded in $L^2$ by \eqref{eq:bddbox}. Consider the term $e_0(\tilde h_\e^{0i}\p_i w_\e)$; an integration by parts in $e_0$ yields
\begin{equation}
\begin{split}
\langle e_0(\tilde h_\e^{0i}\p_i w_\e),A e_0 w_\e\rangle  &=
-\langle \p_k\beta^k \tilde h_\e^{0i}\p_i w_\e, Ae_0w_\e\rangle 
-\langle \tilde h_\e^{0i}\p_i w_\e ,[e_0,A]e_0 w_\e\rangle  
-\langle \tilde h_\e^{0i} \p_i w_\e, A e_0^2 w_\e\rangle \\
&= - \langle \tilde h_\e^{0i} \p_i w_\e, A e_0^2 w_\e\rangle+o(1)\,.
 \label{eq:commutator-intermediate}
\end{split}
\end{equation}
In the remaining term, we apply \eqref{eq:exchange-e0-i} to replace $e_0^2w_\e$ with $\p_{jk}^2w_\e$ and we commute $A$ with $\tilde g^{jk}/g^{00}$:
\begin{align*}
-\langle \tilde h_\e^{0i} \p_i w_\e, A  e_0^2 w_\e\rangle 
% &=- \langle \tilde h_\e^{0i}  \p_i w_\e, A \frac{\Box_g w_\e}{g^{00}}\rangle +\lp\langle \tilde h_\e^{0i} \p_i w_\e,  A \lp(\frac{\tilde{g}^{jk}}{g^{00}}\p_{jk}^2w_\e\rp)\rp\rangle +o(1)\\
= \langle  \tilde h_\e^{0i} \frac{\tilde{g}^{jk}}{g^{00}} \p_i w_\e,  A \p_{jk}^2w_\e \rangle+o(1) \,.
\end{align*}
Integrating by parts in $j$, then $i$ and then $k$, as in Step 1, we obtain
\begin{align*}
\langle \tilde h_\e^{0i} \frac{\tilde{g}^{jk}}{g^{00}} \p_iw_\e, A \p_{jk}^2w_\e \rangle &= - \langle  \tilde h_\e^{0i}  \frac{\tilde{g}^{jk}}{g^{00}}\p_{jk}^2 w_\e, A \p_{i}w_\e \rangle +o(1)= \langle  \tilde h_\e^{0i} e_0^2 w_\e,A \p_i w_\e \rangle +o(1)\,,
\end{align*}
where the last step follows from another application of \eqref{eq:exchange-e0-i}. Combining the previous results, we finally arrive at the identity:
\begin{align*}
-\langle \tilde h_\e^{0i} \p_i w_\e, A e_0^2 w_\e\rangle-\langle e_0^2 w_\e, A( \tilde h_\e^{0i} \p_i w_\e)\rangle =\langle \tilde h^{0i}_\e e_0^2 w_\e, A\p_i w_\e \rangle-\langle e_0^2 w_\e, A( \tilde h_\e^{0i} \p_i w_\e)\rangle +o(1)\,.
\end{align*}
Now we use the self-adjointness of $A$, modulo a compact operator, on all of the terms of the last expression, excluding the first term:
\begin{align*}
-2\langle \tilde h_\e^{0i} \p_i w_\e, A e_0^2 w_\e\rangle =\langle  \p_i w_\e, A (\tilde h^{0i}_\e e_0^2 w_\e) \rangle-\langle \tilde h_\e^{0i} \p_i w_\e, A e_0^2 w_\e \rangle +o(1) = \langle\p_i w_\e, [A,\tilde h_\e^{0i}]e_0^2w_\e\rangle +o(1)\,.
\end{align*}
Recalling \eqref{eq:commutator-intermediate}, we arrive at
\begin{align*}
\lim_{\e\to 0}\left\langle e_0(\tilde h_\e^{0i}\p_i w_\e),A e_0 w_\e \right\rangle &= \lim_{\e\to 0} \frac12 \int  \p_i w_\e [A,\tilde h_\e^{0i}]e_0^2 w_\e  \d x\\
&=\lim_{\e\to 0} \frac12 \int  \p_i w_\e [A,\tilde h_\e^{0i}]\p_0e_0 w_\e  \d x-\lim_{\e\to 0} \frac12 \int  \p_i w_\e [A,\tilde h_\e^{0i}\beta^j]\p_j e_0 w_\e  \d x
\,,
\end{align*}
where we use \eqref{eq:commutator-beta-with-h-right} in the last equality.

\medskip
\noindent \textbf{Step 5:} the two terms in \eqref{eq:Box-difference-alternative-identity-2}. We integrate by parts in $e_0$:
\begin{align*}
&\langle e_0 \tilde h_\e^{00}e_0 u +e_0 \tilde h_\e^{0k}\p_k u, A e_0 (u_\e-u) \rangle\\
&\quad=- \langle \p_j\beta^j (\tilde h_\e^{00}e_0 u+h_\e^{0k}\p_k u), A e_0(u_\e-u)\rangle 
-\langle \tilde h_\e^{00}e_0 u+\tilde h_\e^{0k}\p_k u, [e_0,A]e_0(u_\e-u)\rangle\\
&\quad \qquad -\langle \tilde h_\e^{00}e_0^2 u +\tilde h_\e^{0k}e_0\p_k u, A e_0 (u_\e-u) \rangle
 -\langle \tilde h_\e^{00}e_0 u + \tilde h_\e^{0k}\p_k u, A e_0^2 (u_\e-u) \rangle\\
&\quad =-\langle \tilde h_\e^{00}e_0 u +\tilde h_\e^{0k}\p_k u, A e_0^2 (u_\e-u) \rangle +o(1)\,,
\end{align*}
where the last line follows by the uniform convergence of $\tilde h_\e^{\alpha\beta}$. For the remaining term, we may apply the same reasoning as in the previous step: from \eqref{eq:exchange-e0-i}, we have
\begin{align*}
-\langle \tilde h_\e^{00}e_0 u +\tilde h_\e^{0k}\p_k u, A e_0^2 (u_\e-u) \rangle &=\langle \tilde h_\e^{00}e_0 u +\tilde h_\e^{0k}\p_k u, A \frac{\tilde{g}^{ij}}{g^{00}}\p_{ij}^2(u_\e-u)\rangle +o(1)\\
&=-\langle \p_i \tilde h_\e^{00}e_0 u +\p_i \tilde h_\e^{0k}\p_k u, A \frac{\tilde{g}^{ij}}{g^{00}}\p_{j}(u_\e-u)\rangle +o(1) =o(1)\,,
\end{align*} 
with the second line following from an integration by parts. Thus, \eqref{eq:Box-difference-alternative-identity-2} does not contribute to the limit.
\medskip

\noindent \textbf{Step 6:} conclusion. Combining the previous steps yields 
\begin{align*}
\langle H_\e, A e_0 w_\e\rangle &= \int \p_0 w_\e[A,\tilde h_\e^{00}]\p_0 e_0 w_\e\d x 
+ \int \lp(\p_0 w_\e[A,\tilde h_\e^{0i}-\tilde h_\e^{00}\beta^i]\p_i+\p_i w_\e[A,\tilde h_\e^{i0}-\tilde h_\e^{00}\beta^i]\p_0 \rp)e_0 w_\e\d x\\
&\qquad+ \int \p_i w_\e[A,\tilde h_\e^{ij} -\tilde h_\e^{0i}\beta^j -\tilde h_\e^{0j}\beta^i+\tilde h_\e^{00}\beta^i\beta^j]\p_j e_0 w_\e\d x + o(1)\,,
\end{align*}
and using the definitions in \eqref{eq:defhtilda} the conclusion follows.
\end{proof}

By passing to subsequences if need be, by Hypothesis~\ref{hyp:main}\eqref{it:hypsols} we may find a sequence $\omega_\e\searrow 0$ such that 
\begin{align*}
\sup_{\alpha,\beta}\lVert \p_{\alpha\beta}^2 u_\e \rVert_{L^4(\Omega)} \lesssim \omega_\e^{-1}\,,\qquad  \sup_{\alpha,\beta}\lVert h_\e^{\alpha\beta}  \rVert_{L^\infty(\Omega)} \lesssim \omega_\e\,.
\end{align*}
In order to prove Proposition~\ref{prop:case-ii-no-metric-osc-contribution}, we move to Fourier space.  Let $\zeta\colon \mathbb{R}^+_0\to[0,1]$  be a smooth function such that $\zeta(x)=1$ for $x\leq 1$ and $\zeta=0$ for $x\geq 2$. 

We consider an $\e$-dependent partition of frequency space into low frequencies, spatially-dominated high frequencies and time-dominated high frequencies, c.f.\ Figure \ref{fig:frequency-space}. This partition is associated to the smooth functions $0 \leq \Theta_{\tp{low},\,\e}, \Theta_{\tp{spa},\,\e},\Theta_{\tp{time},\,\e}\leq 1$ defined by
\begin{align*}
\Theta_{\tp{low},\,\e}(\rho)\equiv\zeta\lp(\omega_\e^{\delta_1}|\rho_{\rm tot}|\rp)&\implies \supp \Theta_{\tp{low},\,\e} 
\subseteq \lp\{ |\rho_{\rm tot}| \leq 2\omega_\e^{-\delta_1}\rp\}, \\
\Theta_{\tp{spa},\,\e}(\rho)\equiv (1-\Theta_{\tp{low},\,\e})(\rho)\lp[1-\zeta\lp(\frac{|\rho_{\rm spa}|}{|\rho_{\rm tot}|^{\delta_2}}\rp)\rp] &\implies \supp \Theta_{\tp{spa},\,\e} \subseteq \lp\{|\rho_{\rm spa}|\geq |\rho_{\rm tot}|^{\delta_2} \geq \omega_\e^{-\delta_1\delta_2}\rp\}
,\\
\Theta_{\tp{time},\,\e}(\rho)\equiv(1-\Theta_{\tp{low},\,\e})(\rho)\zeta\lp(\frac{|\rho_{\rm spa}|}{|\rho_{\rm tot}|^{\delta_2}}\rp)&\implies \supp \Theta_{\tp{time},\,\e} \subseteq\lp\{\begin{array}{ll}  \rho_0^2 \geq |\rho_{\rm tot}|^2-4|\rho_{\rm tot}|^{2\delta_2} \\ \tp{ and } |\rho_{\rm tot}| \geq \omega_\e^{-\delta_1} \end{array} \rp\}
,
\end{align*}
where $|\rho_{\rm tot}|^2=\sum_\alpha\rho_\alpha^2$ and $|\rho_{\rm spa}|^2=\sum_i\rho_i^2$. 
Here, $\delta_1,\delta_2>0$ are parameters to be fixed. Clearly
$$\Theta_{\rm{low},\,\e}+\Theta_{\rm{spa},\,\e}+\Theta_{\rm{time},\,\e}=1\,.$$
For an $L^2$ function $\mathfrak{h}_\e$, we define its projections on a range of frequencies according to
\begin{align*}
\mc{P}_{\rm range,\, \e}[\mathfrak{h}_\e] &\equiv \mc{F}^{-1}\lp(\Theta_{\rm range,\, \e}\mc{F}(\mathfrak{h}_\e)\rp)\,,\qquad \mathrm{range} \in \{\rm low, spa, time\}\,,
\end{align*}
where $\mc{F}$ denotes the Fourier transform. These projections are linear and commute with derivatives.

We focus on the low frequencies first. Note that, if the frequency parameter is capped, then derivatives, which in frequency space correspond to multiplication by the frequency variable, are not complete. Thus, the strategy of Section \ref{sec:easycase} still works under the current convergence assumptions on $g_\e$ as long as one restricts to low frequencies:

\begin{lemma} \label{lemma:case-ii-no-metric-osc-contribution-lowfreq}  Under Hypotheses~\ref{hyp:main}, as long as $\delta_1<1$,
\begin{align*}
\lim_{\e\to 0} \int \p_\alpha(u_\e-u) [A, \mc{P}_{\rm low,\, \e}[h_\e^{\alpha\beta}]]\p_\beta e_0 (u_\e-u)  \d x=0\,.
\end{align*}
\end{lemma}

\begin{proof} Without loss of generality, set $u\equiv 0$. Consider the identity
\begin{align*}
\p_\alpha u_\e[A,\mc{P}_{\rm low,\, \e}[h_\e^{\alpha\beta}] ]\p_{\beta}e_{0} u_\e = \p_\alpha u_\e[A \p_\beta,\mc{P}_{\rm low,\, \e}[h_\e^{\alpha\beta}] ]e_{0} u_\e -\p_\alpha u_\e A( \p_\beta\mc{P}_{\rm low,\, \e}[ h_\e^{\alpha\beta}] e_0 u_\e)\,.
\end{align*}
We estimate the second term directly and, for the first term, apply the Theorem \ref{thm:Calderon}: for small $\e$,
\begin{align*}
\left|\int \p_\alpha u_\e[A,\mc{P}_{\rm low,\, \e}[h_\e^{\alpha\beta}] ]\p_\beta e_{0} u_\e \d x \right|
&\lesssim  \sup_{\alpha,\beta}\lVert \mc{P}_{\rm low,\, \e}[h_\e^{\alpha\beta}] \rVert_{W^{1,\infty}}\lVert \p u_\e \rVert_{L^{2}}^2
\lesssim \omega_\e^{-\delta_1}\sup_{\alpha,\beta}\lVert h_\e^{\alpha\beta}\rVert_{L^\infty}
\lesssim \omega_\e^{1-\delta_1}\,,
%\lesssim \omega_\e^{\delta_1(1+\frac{n+1}{p_1})-\frac{n+1}{p_1}}\to 0\,,
\end{align*}
where we use Bernstein's inequality $\hat f \subset B_R(0) \implies 
\Vert \D f\Vert_\infty \lesssim R \Vert f \Vert_{\infty}$ in the second inequality.
\end{proof}

For high frequencies this method fails, as we do not have sufficient control over $h_\e^{00}$ and $h_\e^{0i}$. If the spatial frequencies dominate, however, we can compensate for this issue by appealing to control on higher order spatial derivatives of $h_\e^{\alpha\beta}$. This is independent of the commutator structure.

\begin{lemma} \label{lemma:case-ii-no-metric-osc-contribution-spafreq} Under Hypotheses~\ref{hyp:main}, as long as $\delta_1\delta_2>\frac 1 2$,
\begin{align*}
\lim_{\e\to 0} \int \p_\alpha(u_\e-u) [A, \mc{P}_{\rm spa,\, \e}[h_\e^{\alpha\beta}]]\p_\beta e_0 (u_\e-u)  \d x=0\,.
\end{align*}
\end{lemma}

\begin{proof}  Without loss of generality, set $u\equiv 0$. We have
\begin{align*}
\left|\int \p_\alpha u_\e[A,\mc{P}_{\rm spa,\, \e}[h_\e^{\alpha\beta}] ]\p_{\beta}e_{0} u_\e \d x\right|
&\leq \lp|\int \p_\alpha u_\e A \lp(\mc{P}_{\rm spa,\, \e}[h_\e^{\alpha\beta}]\p_{\beta}e_{0} u_\e \rp)\d x\rp|+ \lp|\int \mc{P}_{\rm spa,\, \e}[h_\e^{\alpha\beta}] \p_\alpha u_\e A \p_{\beta}e_{0} u_\e\d x \rp|\\
&\lesssim \sup_{\alpha,\beta,\gamma}\lVert \p_{\beta}e_{0} u_\e\rVert_{L^4} \lVert \p u_\e\rVert_{L^{4}}\lVert  \mc{P}_{\rm spa,\, \e}[h_\e^{\alpha\beta}]\rVert_{L^2}\lesssim \sup_{\alpha,\beta}\omega_\e^{-1}  \lVert\mc{P}_{\rm spa,\, \e}[h_\e^{\alpha\beta}]\rVert_{L^2} \,.
\end{align*}
Using the assumptions directly would imply that the term above is bounded, but not necessarily converging to zero. However,  by Plancherel's theorem, 
%\red{note that here really needed to take $p_2=2$ and couldn't insert the second derivatives of u inside what follows.}
\begin{align*}
 \lVert\mc{P}_{\rm spa,\, \e}[h_\e^{\alpha\beta}]\rVert_{L^2}^2
 &= \int \lp|\Theta_{\rm spa,\, \e}\widehat{h^{\alpha\beta}_\e}\rp|^2 \tp d\xi 
 =\int \frac{|\Theta_{\rm spa,\, \e}(\xi)|^2}{|\xi_\tp{spa}|^4}\lp|\delta^{ij}\xi_i\xi_j\widehat{h^{\alpha\beta}_\e}(\xi)\rp|^2 \tp d\xi 
 \lesssim \omega_\e^{2\delta_1\delta_2}\lVert  \delta^{ij}\p_{ij}^2h_\e^{\alpha\beta}\rVert_{L^2}^2\,,
\end{align*}
since $|\xi_\tp{spa}|^2\gtrsim |\xi|^{2\delta_2}\gtrsim \omega_\e^{-2\delta_1\delta_2}$ in the support of $\Theta_{\tp{spa},\,\e}(\xi)$. By the boundedness of the spatial laplacian of the metric coefficients, we obtain our result.  
\end{proof}

Finally, we are left with the regime of high frequencies where it is the time frequency which dominates. Here, the lack of control over $\p h_\e^{\alpha\beta}$ is compensated by control over $\Box_g u_\e$, see \eqref{eq:bddbox}. Crucial to the argument is the commutator structure yielded by Lemma~\ref{lemma:commutator-reduction} and the  invertibility of $e_0$ in this frequency regime. 
%Through integrations by parts, we are able to shift any $e_0$ derivative of the metric coefficients onto $e_0u_\e$. By our usual arguments, using $\Box_g u_\e$ (recall \eqref{eq:bddbox}), we may exchange these second time derivatives of $u_\e$ by second spatial derivatives of $u_\e$. Then, integration by parts allows us to shift these spatial derivatives onto the metric coefficients and take advantage of their improved convergence, \eqref{eq:convergence-spatial-derivatives}.

%In this more complicated regime, we manipulate the  Fourier symbol associated to the commutators, which are computed through the following simple lemma:

%\begin{lemma} \label{lemma:computing-fourier-symbols} Suppose $\sigma(A)=m(\xi)$ is real and even. Let $\mf h_\e(x)$ be a real, suitably regular function. Then
%\begin{align*}
%2\int \p_\alpha u_\e [A,\mathfrak{h}_\e] \p^2_{\beta\gamma} u_\e \d x  &=  \iint  \widehat{\p_\gamma u_\e}(\eta) \lp[\widehat{\p_\alpha\mathfrak{h}_\e}(\xi-\eta) \overline{\widehat{\p_\beta u_\e}}(\xi)-\widehat{\p_\beta\mathfrak{h}_\e}(\xi-\eta) \overline{\widehat{\p_\alpha u_\e}}(\xi)\rp][m(\xi)-m(\eta)] \d\xi\d\eta\\
%&\qquad + i\iint  \eta_\alpha\xi_\beta(\eta_\gamma+\xi_\gamma)\reallywidehat{\mathfrak{h}_\e}(\xi-\eta) \widehat{u_\e}(\eta)\overline{\widehat{u_\e}}(\xi)[m(\xi)-m(\eta)] \d\xi\d\eta\,.
%\end{align*}
%\end{lemma}

%We are now ready to prove our final lemma, which is the heart of the proof.

\begin{lemma} \label{lemma:case-ii-no-metric-osc-contribution-timefreq}
Under Hypotheses~\ref{hyp:main} and assuming that $A$ satisfies \eqref{eq:hypsymbol}, if  $\delta_1>\frac 1 2$  and $\delta_2<1$,
\begin{align*}
\int \p_\alpha(u_\e-u) [A, \mc{P}_{\rm time,\, \e}[h_\e^{\alpha\beta}]]\p_\beta e_0 (u_\e-u)  \d x=0\,.
\end{align*}
\end{lemma}

\begin{proof} Without loss of generality, set $u\equiv 0$. 
Throughout the proof, we let  $\mathfrak{h}_\e^{\alpha\beta}\equiv \mc{P}_{\rm time}[h_\e^{\alpha\beta}]$ and we assume  that $\delta_1>0$ and $\delta_2<1$.
We also note that, for sufficiently small $\e_0$,
 \begin{align*}
|\beta^j(x)\xi_j/\xi_0|\ll 1\,, \quad \lp|\xi_0+\beta^j(x)\xi_j\rp| \gtrsim \omega_\e^{-\delta_1}\gg 1\qquad \text{ when } \xi\in \supp\Theta_{\rm time,\, \e}\,,
\end{align*}
whenever $\e\leq\e_0$.
Hence we may find an operator $Q\in \Psi^{-1}_{1,c}$ and $R\in \Psi^{0}_{1,c}$ such that
\begin{align*}
\sigma(Q)(x,\xi)&=q(x,\xi)\equiv \lp[g^{00}(x)(\xi_0+\beta^j(x)\xi_j)\rp]^{-1} = \frac{1+ \sigma({R})(x,\xi)}{g^{00}(x)\xi_0} \,,\quad
\sigma({R})(x,\xi)\equiv  \sum_{\ell=1}^\infty \lp(-\beta^j(x)\frac{\xi_j}{\xi_0}\rp)^\ell\,,
\end{align*}
whenever $(x,\xi)\in \Omega\times \supp \Theta_{\tp{time},\,\e_0}$. 

It is now easy to see that we have the estimates
\begin{gather}
\lVert  Q\mathfrak{h}_\e^{\alpha\beta}\rVert_{L^2(\Omega)} \lesssim \omega_\e^{1+\delta_1}\,, \label{eq:time-dominated-Qh-estimate}\\
\lVert  \p_0(Q\mathfrak{h}_\e^{\alpha\beta})\rVert_{L^2(\Omega)} \lesssim \omega_\e^{\delta_1}\,, 
\qquad {\textstyle \sup_j}\lVert  \p_j (Q\mathfrak{h}_\e^{\alpha\beta})\rVert_{L^2 (\Omega)} \lesssim \omega_\e^{\frac12 +\delta_1} \,.
\label{eq:time-dominated-Qh-estimate-derivative}
\end{gather}
Indeed, for \eqref{eq:time-dominated-Qh-estimate}, we  compute
\begin{align*}
\lVert Q \mf h_\e^{\a\beta}\rVert_{L^2(\Omega)} & \lesssim \lVert (g^{00})^{-1} \rVert_{L^\infty(\Omega)} \lVert \xi_0^{-1}\rVert_{L^\infty_\xi(\supp \Theta_{\e,\, \rm time})}\lVert h_\e^{\alpha\beta} \rVert_{L^2} \sum_{\ell=0}^\infty \lp[\lVert \beta^{i} \rVert_{L^\infty(\Omega)} \lVert \xi_j\xi_0^{-1}\rVert_{L^\infty_\xi(\supp \Theta_{\e,\, \rm time})}\rp]^\ell \\
&\lesssim \lVert \xi_0^{-1}\rVert_{L^\infty_\xi(\supp \Theta_{\e,\, \rm time})} \lVert h_\e^{\alpha\beta} \rVert_{L^2} \lesssim \omega_\e^{1+\delta_1}\,.
\end{align*}
The estimates in \eqref{eq:time-dominated-Qh-estimate-derivative} follow similarly. Note that we only require $L^2$ norms in $\Omega$ in what follows as we will always be testing against $u_\e$ and its derivatives, which have compact support in $\Omega$.

Using $Q$, we may rewrite our commutator as
\begin{align}
\langle \p_\alpha u_\e, [A, \mathfrak{h}_\e^{\alpha\beta}]\p_\beta e_0 u_\e\rangle 
&= \langle \p_\alpha u_\e, A\lp( g^{00}e_0Q\mathfrak{h}_\e^{\alpha\beta}\p_\beta e_0 u_\e\rp)\rangle -\langle \p_\alpha u_\e g^{00}e_0Q\mathfrak{h}_\e^{\alpha\beta}, A \p_\beta e_0 u_\e \rangle\,. \label{eq:time-dominated-commutator}
\end{align}

\medskip
\noindent\textbf{Step 1:} integration by parts in $e_0$. In this step, we show:
\begin{align}
\lim_{\e\to 0}\langle \p_\alpha u_\e, [A, \mathfrak{h}_\e^{\alpha\beta}]\p_\beta e_0 u_\e\rangle 
&=\lim_{\e\to 0}\langle \p^2_{\alpha\beta} u_\e, [A,Q\mathfrak{h}_\e^{\alpha\beta}](g^{00} e_0^2 u_\e) \rangle+\lim_{\e\to 0}\langle \p_{\alpha} u_\e, [A,\p_\beta(Q\mathfrak{h}_\e^{\alpha\beta})](g^{00} e_0^2 u_\e) \,. \label{eq:time-dominated-claim-1}
\end{align}

To begin, we seek to move the $e_0$ derivative on $Q\mathfrak{h}_\e^{\alpha\beta}$ in \eqref{eq:time-dominated-commutator} onto $u_\e$ through integration by parts in $e_0$. Note that, whenever a derivative hits a coefficient of the limit metric $g$, that term is $o(1)$: using \eqref{eq:time-dominated-Qh-estimate},
\begin{align}
\langle \p g\, Q\mathfrak{h}_\e^{\alpha\beta}\p^2 u_\e, A \p u_\e \rangle \lesssim \lVert \p^2 u_\e \rVert_{L^4}\lVert \p u_\e \rVert_{L^4}\lVert Q\mathfrak{h}_\e^{\alpha\beta} \rVert_{L^2} \lesssim \omega_\e^{\delta_1}\to 0\,,\label{eq:time-dominated-derivatives-on-g}
\end{align}
where $\p$ denotes an arbitrary partial derivative and $g$ an arbitrary metric coefficient. We note that the order of the terms in the left-hand side is unimportant. We will use \eqref{eq:time-dominated-derivatives-on-g} and its variants implicitly in the sequel.

The first term of \eqref{eq:time-dominated-commutator} becomes
\begin{align*}
\langle \p_\alpha u_\e, A\lp( g^{00}e_0Q\mathfrak{h}_\e^{\alpha\beta}\p_\beta e_0 u_\e\rp)\rangle 
&= \langle \p_\alpha u_\e, [A,g^{00}e_0]\lp( Q\mathfrak{h}_\e^{\alpha\beta}\p_\beta e_0 u_\e\rp)\rangle+\langle \p_\alpha u_\e g^{00}, e_0 A\lp( Q\mathfrak{h}_\e^{\alpha\beta}\p_\beta e_0 u_\e\rp)\rangle\\
&\qquad -\langle \p_\alpha u_\e, A \lp( Q\mathfrak{h}_\e^{\alpha\beta}[g^{00}e_0,\p_\beta] e_0 u_\e\rp)\rangle-\langle \p_\alpha u_\e, A \lp( Q\mathfrak{h}_\e^{\alpha\beta}\p_\beta(g^{00} e_0^2 u_\e)\rp)\rangle\\
& = -\langle \p_\alpha u_\e, A \lp( Q\mathfrak{h}_\e^{\alpha\beta}\p_\beta(g^{00} e_0^2 u_\e)\rp)\rangle - \langle  g^{00}e_0 \p_\alpha u_\e , A\lp( Q\mathfrak{h}_\e^{\alpha\beta}\p_\beta e_0 u_\e\rp)\rangle+o(1)\,,
\end{align*}
and the second term yields
\begin{align*}
-\langle \p_\alpha u_\e g^{00}e_0Q\mathfrak{h}_\e^{\alpha\beta}, A \p_\beta e_0 u_\e \rangle 
%
%&\quad= \langle \p_t(g^{00}\p_\alpha u_\e Q\mathfrak{h}_\e^{\alpha\beta}), A \p_\beta e_0 u_\e \rangle-\langle \p_k( \beta^k g^{00}\p_\alpha u_\e Q\mathfrak{h}_\e^{\alpha\beta}), A \p_\beta e_0 u_\e \rangle\\
%&\quad\qquad+\langle e_0\p_\alpha u_\e g^{00} Q\mathfrak{h}_\e^{\alpha\beta}, A \p_\beta e_0 u_\e \rangle +\langle (e_0g^{00}+\p_k \beta^k g^{00})\p_\alpha u_\e  Q\mathfrak{h}_\e^{\alpha\beta}, A \p_\beta e_0 u_\e \rangle\\
%
&= \langle \p_\alpha u_\e Q\mathfrak{h}_\e^{\alpha\beta}, [g^{00} e_0,A\p_\beta]  e_0 u_\e \rangle+\langle \p_\alpha u_\e Q\mathfrak{h}_\e^{\alpha\beta}, A\p_\beta(g^{00} e_0^2 u_\e) \rangle\\
&\qquad +\langle g^{00} e_0\p_\alpha u_\e  Q\mathfrak{h}_\e^{\alpha\beta}, A \p_\beta e_0 u_\e \rangle +\langle (e_0g^{00}+\p_k \beta^k g^{00})\p_\alpha u_\e  Q\mathfrak{h}_\e^{\alpha\beta}, A \p_\beta e_0 u_\e \rangle\\
&= \langle \p_\alpha u_\e Q\mathfrak{h}_\e^{\alpha\beta}, A\p_\beta(g^{00} e_0^2 u_\e) \rangle +\langle g^{00} e_0\p_\alpha u_\e  Q\mathfrak{h}_\e^{\alpha\beta}, A \p_\beta e_0 u_\e \rangle+o(1)\,.
\end{align*}
Combining the previous computations gives 
\begin{align*}
\langle \p_\alpha u_\e, [A, \mathfrak{h}_\e^{\alpha\beta}]\p_\beta e_0 u_\e\rangle 
&=\langle \p_\alpha u_\e, [Q\mathfrak{h}_\e^{\alpha\beta}, A]\p_\beta(g^{00} e_0^2 u_\e) \rangle +\langle g^{00} e_0\p_\alpha u_\e ,[ Q\mathfrak{h}_\e^{\alpha\beta}, A] \p_\beta e_0 u_\e \rangle+o(1)\\
&=\langle \p_\alpha u_\e, [Q\mathfrak{h}_\e^{\alpha\beta}, A]\p_\beta(g^{00} e_0^2 u_\e) \rangle +o(1)\,, \numberthis \label{eq:time-dominated-claim-intermediate}
\end{align*}
because, by the symmetry of $\mathfrak{h}_\e^{\alpha\beta}$, the self-adjointness of $A$ (up to a compact operator), and \eqref{eq:time-dominated-derivatives-on-g}, we have
\begin{align*}
&\langle g^{00} e_0\p_\alpha u_\e ,[ Q\mathfrak{h}_\e^{\alpha\beta}, A] \p_\beta e_0 u_\e \rangle \\
\numberthis
\label{eq:symmetry1}
&\quad=\langle g^{00} e_0\p_\beta u_\e ,[ Q\mathfrak{h}_\e^{\alpha\beta}, A] \p_\alpha e_0 u_\e \rangle = \langle A(e_0\p_\beta u_\e g^{00} Q\mathfrak{h}_\e^{\alpha\beta}),  \p_\alpha e_0 u_\e \rangle-\langle A( e_0\p_\beta u_\e g^{00}), Q\mathfrak{h}_\e^{\alpha\beta} \p_\alpha e_0 u_\e\rangle\\
&\quad = \langle A(\p_\beta e_0 u_\e Q\mathfrak{h}_\e^{\alpha\beta}),  g^{00} e_0 \p_\alpha u_\e \rangle-\langle A(\p_\beta e_0 u_\e ), Q\mathfrak{h}_\e^{\alpha\beta} g^{00} e_0 \p_\alpha  u_\e\rangle +o(1)\\
&\quad=-\langle g^{00} e_0 \p_\alpha u_\e, [Q\mathfrak{h}_\e^{\alpha\beta}, A]\p_\beta e_0 u_\e \rangle +o(1)\,.
\end{align*}
To obtain \eqref{eq:time-dominated-claim-1}, we need only integrate \eqref{eq:time-dominated-claim-intermediate} by parts in $\p_\beta$.

\medskip 
\noindent\textbf{Step 2:} introducing $\Box_g$. In this step, we show:
\begin{align}
\begin{split}
\langle \p_\alpha u_\e, [A, \mathfrak{h}_\e^{\alpha\beta}]\p_\beta e_0 u_\e\rangle 
&=\langle \p^2_{\alpha\beta} u_\e, [A,Q\mathfrak{h}_\e^{\alpha\beta}]\Box_g u_\e \rangle+\langle \p_{\alpha} u_\e, [A,\p_\beta(Q\mathfrak{h}_\e^{\alpha\beta})]\Box_g u_\e\rangle\\
& \qquad +\langle \p_{\alpha}u_\e, [\p_i(Q\mathfrak{h}_\e^{\alpha\beta}),A]\lp(\tilde{g}^{ij}\p_{j\beta}^2 u_\e\rp)\rangle  +o(1) \,. 
\end{split}\label{eq:time-dominated-claim-2}
\end{align}

From \eqref{eq:Box-alternative-identity}, it is clear that terms  $g^{00}e_0^2$ in \eqref{eq:time-dominated-claim-1}  may be replaced by $\Box_g -\tilde{g}^{ij}\p_{ij}$, as the remaining terms in \eqref{eq:Box-alternative-identity}, which involve derivatives of $g$, do not contribute, c.f.\ \eqref{eq:time-dominated-derivatives-on-g}. Thus, we have
\begin{align*}
&\langle \p_\alpha u_\e, [A, \mathfrak{h}_\e^{\alpha\beta}]\p_\beta e_0 u_\e\rangle\\ 
&\quad=-\langle \p^2_{\alpha\beta} u_\e, [A,Q\mathfrak{h}_\e^{\alpha\beta}](\tilde g^{ij} \p_{ij}^2 u_\e) \rangle -\langle \p_{\alpha} u_\e, [A,\p_\beta(Q\mathfrak{h}_\e^{\alpha\beta})](\tilde g^{ij} \p_{ij}^2 u_\e)\\
&\quad\qquad +\langle \p^2_{\alpha\beta} u_\e, [A,Q\mathfrak{h}_\e^{\alpha\beta}]\Box_g u_\e \rangle+\langle \p_{\alpha} u_\e, [A,\p_\beta(Q\mathfrak{h}_\e^{\alpha\beta})]\Box_g u_\e\rangle   +o(1)\\
&\quad=\langle \p_{\alpha} u_\e, [A,Q\mathfrak{h}_\e^{\alpha\beta}]\p_\beta(\tilde g^{ij} \p_{ij}^2 u_\e) \rangle +\langle \p^2_{\alpha\beta} u_\e, [A,Q\mathfrak{h}_\e^{\alpha\beta}]\Box_g u_\e \rangle+\langle \p_{\alpha} u_\e, [A,\p_\beta(Q\mathfrak{h}_\e^{\alpha\beta})]\Box_g u_\e\rangle   +o(1)\,,
\end{align*}
integrating by parts in $\p_\beta$ to arrive at the final equality. Now, we integrate the first term in $\p_i$:
\begin{align*}
\langle \p_{\alpha} u_\e, [A,Q\mathfrak{h}_\e^{\alpha\beta}]\p_\beta(\tilde g^{ij} \p_{ij}^2 u_\e) \rangle = \langle \p_{i\alpha}^2u_\e, [Q\mathfrak{h}_\e^{\alpha\beta},A]\lp(\tilde{g}^{ij}\p_{j\beta}^2 u_\e\rp)\rangle +\langle \p_{\alpha}u_\e, [\p_i(Q\mathfrak{h}_\e^{\alpha\beta}),A]\lp(\tilde{g}^{ij}\p_{j\beta}^2 u_\e\rp)\rangle+o(1)\,,
\end{align*}
where we use \eqref{eq:time-dominated-derivatives-on-g} as needed. To obtain our claim, it only remains to show that the first term in the above formula vanishes in the limit. To see this, we argue as before, invoking the symmetry of $\mathfrak{h}_\e^{\alpha\beta}$ and $\tilde{g}^{ij}$ in their indices and self-adjointness of $A$ (up to a compact operator):
\begin{align}
\begin{split}
\langle \p_{i\alpha}^2u_\e, [Q\mathfrak{h}_\e^{\alpha\beta},A]\lp(\tilde{g}^{ij}\p^2_{j\beta} u_\e\rp)\rangle &=\langle \p_{j\beta}^2u_\e, [Q\mathfrak{h}_\e^{\alpha\beta},A]\lp(\tilde{g}^{ij}\p^2_{i\alpha} u_\e\rp)\rangle =-\langle \tilde{g}^{ij}\p^2_{i\alpha} u_\e, [Q\mathfrak{h}_\e^{\alpha\beta},A] \p^2_{j\beta} u_\e\rangle\\
&=-\langle \p_{i\alpha}^2u_\e ,[Q\mathfrak{h}_\e^{\alpha\beta},A]\lp(\tilde{g}^{ij}\p^2_{j\beta} u_\e\rp)\rangle+o(1)\,.
\end{split}
\label{eq:symmetry2}
\end{align}

\medskip
\noindent \textbf{Step 3:} conclusion. From \eqref{eq:time-dominated-claim-2}, and estimates \eqref{eq:bddbox}, \eqref{eq:time-dominated-Qh-estimate}, \eqref{eq:time-dominated-Qh-estimate-derivative}, we obtain
\begin{align*}
&\lim_{\e \to 0 }\int \p_\alpha u_\e [A,\mc{P}_{\rm time,\, \e}[{h}_\e^{\alpha\beta}]] \p_{\beta}e_0 u_\e \d x\\
&\quad\lesssim \lVert Q h_\e\rVert_{L^2(\Omega)} \lVert \p^2 u_\e\rVert_{L^4}\lVert \Box_g u_\e\rVert_{L^4} + \lVert \p (Q h_\e)\rVert_{L^2(\Omega)} \lVert \p u_\e\rVert_{L^4}\lVert \Box_g u_\e\rVert_{L^4}+\lVert \p_j(Q h_\e)\rVert_{L^2(\Omega)} \lVert \p^2 u_\e\rVert_{L^4}\lVert \p u_\e\rVert_{L^4}\\
&\quad\lesssim \omega_\e^{\delta_1}+\omega_\e^{\delta_1-\frac12}\to 0\,,
\end{align*}
as long as $\delta_1>\frac12$.
\end{proof}

\begin{proof}[Proof of Proposition~\ref{prop:case-ii-no-metric-osc-contribution}] 
It suffices to pick $\delta_1\in (\frac12, 1)$ and $\delta_2\in (\frac{1}{2\delta_1},1)$; for instance,  $\delta_1=\frac56$ and $\delta_2= \frac45$. Now combine Lemma~\ref{lemma:commutator-reduction} with Lemmas~\ref{lemma:case-ii-no-metric-osc-contribution-lowfreq}, \ref{lemma:case-ii-no-metric-osc-contribution-spafreq} and \ref{lemma:case-ii-no-metric-osc-contribution-timefreq}.
\end{proof} 

\subsection{Conclusion of the proof}

Combining the results from the previous subsections we finish the proof of Theorem \ref{thm:linear-wave-oscillating}.

\begin{proof}[Proof of \eqref{eq:propagationprop} in Theorem \ref{thm:linear-wave-oscillating}(\ref{it:thm-lin-wave-energy-density})]
By Proposition \ref{prop:case-ii-no-metric-osc-contribution}, whenever $A$ satisfies \eqref{eq:hypsymbol},
$$0=\lim_{\e\to 0}\langle Ae_0 (u_\e-u), H_\e\rangle = \langle (\xi_0-\beta^i\xi_i)\sigma, a\rangle -\langle  (\xi_0-\beta^i\xi_i) \lambda,a\rangle\,.$$ Since $\nu$ is supported on the zero mass shell $\{g^{\a\beta}\xi_\a\xi_\beta=0\}$, and since  $\xi_0-\beta^i\xi_i$ never vanishes on that set, see \eqref{eq:lightcone}, it follows that $\langle \lambda, \tilde a\rangle=\langle \sigma, \tilde a\rangle$ for any $\tilde a\in C^\infty_c(S^*\mc M)$ which is \textit{odd and 1-homogeneous in $\xi$}. Thus, according to Theorem \ref{thm:linearwave}\eqref{it:propagation}, for any such $\tilde a$,
\begin{equation}
\label{eq:propagationpropsec4}
\int_{S^*\mc M} \lp[g^{\alpha\beta}\xi_\alpha\p_{x^\beta}\tilde a-\frac12\p_{x^\mu}g^{\alpha\beta}\xi_\alpha\xi_\beta\p_{\xi_\mu}\tilde a\rp] \tp d\nu 
= -\int_{S^*\mc M} \tilde a\, \tp{d}(\Re \sigma)
= -\int_{S^*\mc M} \tilde a\, \tp{d}(\Re \lambda)
\,.
\end{equation}
However, $\nu$ is even and $\lambda$ is odd, c.f.\ Theorem \ref{thm:linearwave}\eqref{it:parity}: thus, whenever $\tilde a$ is even in $\xi$, $\langle \Re\lambda, \tilde a\rangle =0$, and likewise the right-hand side of \eqref{eq:propagationpropsec4} vanishes as well in that case. Hence we see that \eqref{eq:propagationpropsec4} actually holds for \textit{any} $\tilde a\in C^\infty_c(S^*\mc M)$, as wished.
\end{proof}

\section{Nonlinear wave map systems with oscillating coefficients}

%The purpose of this section is to prove Theorem \ref{thm:wave-map}, as well as Corollary \ref{cor:luk-huneau}.

\subsection{Proof of Theorem~\ref{thm:wave-map}}
\label{sec:proof-wave-maps}

We recall that Theorem \ref{thm:wave-map} is concerned with sequences of solutions to
\begin{equation}
\Box_{g_\e} u_\e^{I} =-\Gamma_{JK}^I(u_\e)g_\e^{-1}(\tp{d} u_\e^J,\d u_\e^K)+f_\e^I\,,\qquad u_\e^I,f_\e^I\colon\mathbb{R}^{1+n}\to \mathbb{R}\,, \qquad  {I,J,K}\in\{1,\dots,N\}\,. \label{eq:wave-map-system}
\end{equation}
%This section is concerned with the wave map system \eqref{eq:wave-map-system-intro}, i.e.\
%\begin{equation}
%\Box_g u^{I} =\Gamma_{JK}^I(u)g^{-1}(\tp{d} u^J,\d u^K)+f^I\,,\qquad u^I,f^I:\mathbb{R}^{1+n}\to \mathbb{R}\,, \quad  {I,J,K}\in\{1,\dots,N\}\,, \label{eq:wave-map-system}
%\end{equation}
%where $g$ is a Lorentzian metric on $(0,T)\times\mathbb{R}^{n}$ and $\Gamma_{JK}^I\colon\mathbb{R}^N\to \mathbb{R}$ are continuous functions. 
%
%\subsection{Discussion of the hypotheses}
%
%Consider sequences of functions $u^I_\varepsilon,f^I_\varepsilon\colon (0,T)\times\mathbb{R}^{n}\to \mathbb{R}$ and Lorentzian metrics $g_\varepsilon$ on $(0,T)\times\R^n$  such that, for each $\varepsilon>0$, $(g_\e, u_\varepsilon^1, \dots, u_\varepsilon^N, f_\e^1, \dots, f_\e^N)$ satisfy the wave map system \eqref{eq:wave-map-system}. 
We will reduce the study of the wave map system \eqref{eq:wave-map-system} to the case of wave maps into a flat target, as studied in Section \ref{sec:linear-wave-non-oscillating} and \ref{sec:linear-wave-oscillating}. By repeating the arguments detailed in Section \ref{sec:linear-wave-non-oscillating} we see that, by replacing  $u_\e$ with $\chi u_\e$ for an arbitrary smooth cut-off function $\chi$, there is no loss of generality in assuming that the sequence $(u_\e)$ has uniformly bounded support.

Before proceeding with the proof, let us introduce the notation
$$
H_\e^I\equiv (\Box_{g_\e}-\Box_{g})u_\e^I,\qquad
Q_\e^I \equiv -\Gamma_{JK}^I(u_\e)g_\e^{-1}(\tp{d} u_\e^J,\d u_\e^K), \qquad
Q^I \equiv \tp{w-}\lim_{\e \to 0} Q_\e^I\,.
$$
Hence we may rewrite \eqref{eq:wave-map-system} as
$$\Box_g u^I_\e =F_\e^I,\qquad \tp{where } F_\e^I \equiv -H_\e^I + Q_\e^I + f_\e^I\,.$$
In addition to the H-measures defined in \eqref{eq:defHmeasureintro}, we will need the H-measure 
$$
\lp((\p_0 u^I_\e,\p_1 u^I_\e, \dots, \p_n u^I_\e, F_\e^I)_{I=1}^N\rp)_\e \wH 
\lp(
\begin{bmatrix}
\tilde \nu^{IJ} & \tilde \sigma^{IJ} \\
(\tilde \sigma^{IJ})^* & \star
\end{bmatrix}
\rp)_{I,J=1}^N\,.
$$

We deal with the terms $H_\e^I$ and $Q_\e^I$ separately. For the former, it suffices to apply, with minor modifications, the arguments in Section \ref{sec:linear-wave-oscillating}:

\begin{lemma}
 \label{lemma:rhssection4}
 Under Hypotheses \ref{hyp:main} and assuming that $u_\e\to u$ in $C^0_\tp{loc}$, for any $A\in \Psi^0_{1,c}$ satisfying \eqref{eq:hypsymbol},
  $$\lim_{\e\to0} \langle  A \,e_0(u_\e^I-u^I),\mf g_{IL}(u_\e) H_\e^L\rangle=0.$$
\end{lemma} 

\begin{proof}
The proof consists of a small modification of the arguments used to prove Propositions \ref{prop:case-i-no-metric-osc-contribution} and \ref{prop:case-ii-no-metric-osc-contribution}. Here we only point out the modifications needed in the proof of Proposition \ref{prop:case-ii-no-metric-osc-contribution}, as the former is much simpler. Note that, by the local uniform convergence of $u_\e$, it is enough to show that
$$\lim_{\e\to0} \langle  A \,e_0(u_\e^I-u^I),\mf g_{IL}(u) H_\e^L\rangle=0.$$
For simplicity of notation we suppress the dependence of $\mf g_{IL}$ on $u$.

Similarly to Lemmas \ref{lemma:commutator-reduction-easy} and \ref{lemma:commutator-reduction}, we have
\begin{equation}
2\lim_{\e\to 0}\langle A e_0 (u^I_\e-u^I),  g_{IL}H^L_\e\rangle = 
\lim_{\e\to 0} \int g_{IL} \mf \p_\alpha (u^I_\e-u^I) [A, h_\e^{\alpha\beta}]\p_\beta e_0 (u^L_\e-u^L) \d x 
\,.\label{eq:commutator-reduction-revisited}
\end{equation}
Indeed, the proofs of these lemmas consists of integrating by parts using the self-adjointness of $A$ to produce commutators. With $\mf g_{IL}$ now in the bracket, the integration by parts generates terms with derivatives of $\mf g_{IL}$, which however are compact, as they have one fewer derivative on $u_\e^J$. Using the self-adjointness also yields the same conclusion: e.g.\ in Step 2 of Lemma \ref{lemma:commutator-reduction}, again writing $w_\e^J\equiv u_\e^J-u_\e^J$, we find the commutator
\begin{align*}
\langle \mf g_{IL} e_0 w_\e^L, [A,\tilde h_\e^{00}] e_0^2 w_\e^I\rangle &= \langle \mf g_{IL} \tilde h_\e^{00}e_0^2 w_\e^L, A e_0 w_\e^I \rangle+\langle e_0 w_\e^L, A ( \mf g_{IL} \tilde h_\e^{00} e_0^2 w_\e^I) \rangle + o(1)\\
&=\langle \mf g_{IL} \tilde  h_\e^{00}e_0^2 w_\e^L, A e_0 w_\e^I \rangle+\langle \mf g_{IL} \tilde  h_\e^{00}e_0^2 w_\e^I, A e_0 w_\e^L \rangle + o(1)\\
&= 2\langle  \mf g_{IL} \tilde h_\e^{00}e_0^2 w_\e^L, A e_0 w_\e^I \rangle+o(1)\,,
\end{align*}
due to the symmetry of $\mf g_{IL}$ in $I,L$. Arguing similarly in the other steps, \eqref{eq:commutator-reduction-revisited} is established.

The proofs of Proposition \ref{prop:case-i-no-metric-osc-contribution} and Lemmas \ref{lemma:case-ii-no-metric-osc-contribution-lowfreq} and \ref{lemma:case-ii-no-metric-osc-contribution-spafreq} only require cosmetic modifications. In the proof of Lemma \ref{lemma:case-ii-no-metric-osc-contribution-timefreq}, the fact that $I=L$ is used in a non-trivial way in the arguments involving the symmetry in $\a,\beta$ of $h^{\a\beta}_\e$ in \eqref{eq:symmetry1} and \eqref{eq:symmetry2}. However, since we now sum over all $I,L$ and $\mf g_{IL}=\mf g_{LI}$, these arguments still apply: for instance, the analogue of \eqref{eq:symmetry2} is now
\begin{align*}
\langle \mf g_{IL} \p_{i\alpha}^2u^I_\e, [Q\mathfrak{h}_\e^{\alpha\beta},A]\lp(\tilde{g}^{ij}\p^2_{j\beta} u^L_\e\rp)\rangle &=\langle\mf g_{IL} \p_{j\beta}^2u^L_\e, [Q\mathfrak{h}_\e^{\alpha\beta},A]\lp(\tilde{g}^{ij}\p^2_{i\alpha} u^I_\e\rp)\rangle =-\langle \tilde{g}^{ij}\p^2_{i\alpha} u^I_\e, [Q\mathfrak{h}_\e^{\alpha\beta},A] (\mf g_{IL}\p^2_{j\beta} u_\e^L)\rangle\\
&=-\langle \mf g_{IL}\p_{i\alpha}^2u^I_\e ,[Q\mathfrak{h}_\e^{\alpha\beta},A]\lp(\tilde{g}^{ij}\p^2_{j\beta} u^L_\e\rp)\rangle+o(1)\,,
\end{align*}
where we exchanged $\a$ with $\beta$, $I$ with $L$ and $i$ with $j$ in the first equality. Here, we have also commuted through $\mf g_{IL}$ to place it on the right hand side; this follows similarly as for the commutation of $\tilde{g}^{ij}$, since $\mf g_{IL}$ is independent of $\e$.
\end{proof}
 
In light of Theorem~\ref{thm:linear-wave-oscillating}, the main remaining point in the proof of Theorem \ref{thm:wave-map} is to characterize the contribution of $Q_\e^I$ to the transport equation. This is done in the next lemma.

\begin{lemma}
\label{lemma:contributionsRHSwavemap}
Assuming that Hypotheses~\ref{hyp:main} hold and that $u_\e^I \to u^I$ in $C^0_\tp{loc}$, then for any $A\in \Psi^0_{1,c}$
\begin{align*}
\lim_{\e \to 0} \langle A(\p_\gamma u_\e^I-\p_\gamma u^I), Q^L_\e-Q^L\rangle =
\langle g^{\alpha \beta}\lp[\Gamma_{JK}^L(u)+\Gamma_{KJ}^L(u)\rp] \p_\beta u^K \tilde \nu^{IJ}_{\gamma\alpha},\sigma^0(A)\rangle\,,
\end{align*}
Additionally, if 
$\Gamma_{JK}^I$ are Christoffel symbols with respect to a Riemannian metric $\mathfrak{g}=\mathfrak{g}_{IJ}(u) \d u^I\otimes \tp{d} u^J,$
then%
\begin{equation}
\label{eq:contributionHepsilon}
\begin{split}
\lim_{\e \to 0}  \langle A(\p_\gamma u_\e^I-\p_\gamma u^I) ,& \mf g_{IL}(Q_\e^L-Q^L)\rangle
=\\
&=
-\langle
 g^{\a\beta} \p_{x^\beta} \mf g_{IJ}\, \tilde \nu^{IJ}_{\gamma\a},\sigma^0(A)\rangle
+ 2i
\langle g^{\a\beta}\p_{u^J} \mf  g_{IK}\,  \p_{x^\beta} u^K \Im \tilde\nu^{IJ}_{\gamma\a}, \sigma^0(A)\rangle  \,.
\end{split}
\end{equation}
\end{lemma}

\begin{proof} 
We have that $\tp{w-}\lim Q_\e^L = \tp{w-}\lim -\Gamma_{JK}^L(u) g^{-1}(\d u_\e^J, \d u_\e^K)$, by continuity of $\Gamma_{JK}^L$ and uniform convergence of $g_\e$ and $u_\e$. Hence, by Lemma \ref{lemma:compensatedcompactnessWave},
\begin{align}
Q_\e^L \w Q^L = -\Gamma_{JK}^L(u) g^{-1}(\tp{d}u^J,\tp{d}u^K)\quad \tp{ in } L^2\,.
\label{eq:convergencenonlinearities}
\end{align}

The first part of the lemma is now a direct consequence of the trilinear compensated compactness of Lemma~\ref{lemma:trilinear}. Indeed, for any $A\in\Psi^0_{1,c}$, we have
\begin{align*}
&\lim_{\e\to 0}\langle A(\p_\gamma u_\e^I-\p_\gamma u^I), Q^L_\e-Q^L\rangle \\
&\quad= \lim_{\e\to 0}\big\langle \Gamma_{JK}^L(u)\, A(\p_\gamma u_\e^I-\p_\gamma u^I),
g^{-1}(\tp{d}u^J,\tp{d}u^K)-g^{-1}(\tp{d}u_\e^J,\tp{d}u_\e^K)\big\rangle \\
&\quad= \lim_{\e\to 0} - \big\langle \Gamma_{JK}^L(u)\, A(\p_\gamma u_\e^I-\p_\gamma u^I), 
g^{-1}\lp(\tp{d}(u^J-u^J_\e),\tp{d}(u^K-u^K_\e)\rp)\big\rangle \\
&\quad\qquad+ \lim_{\e\to 0}\big\langle \Gamma_{JK}^L(u)\, A(\p_\gamma u_\e^I-\p_\gamma u^I), g^{-1}\lp(\tp{d}(u^J-u^J_\e),\tp{d}u^K\rp)+g^{-1}\lp(\tp{d}u^J,\tp{d}(u^K -u^K_\e)\rp)\big\rangle
\end{align*}
again by continuity of $\Gamma^L_{JK}$ and uniform convergence of $u_\e$ and $g_\e$. By Lemma~\ref{lemma:trilinear}, the first limit on the right-hand side vanishes, hence we arrive at
\begin{align*}
\lim_{\e\to 0}\langle A(\p_\gamma u_\e^I-\p_\gamma u^I), Q^L_\e-Q^L\rangle &= 
\big\langle \Gamma_{JK}^L(u)\,\tilde \nu^{IJ}_{\gamma\alpha}, \,  g^{\alpha\beta}\p_{\beta}u^K \sigma^0(A) \big\rangle + 
\big\langle \Gamma_{JK}^L(u) \, \tilde \nu^{IK}_{\gamma\alpha}, \, g^{\alpha\beta}\p_{\beta}u^J\sigma^0(A) \big\rangle\,.
\end{align*}

For the second part, we begin by recalling the formula for the Christoffel symbols:
\begin{equation}
\mf g_{IL} \Gamma^L_{JK}  = \frac 1 2 \lp ( \mf g_{KI,J} + \mf g_{JI,K} - \mf g_{JK,I}\rp)
\,,
\label{eq:Christoffelsymbols}
\end{equation}
where $\mf g_{JK,I}\equiv \p_{u^I} \mf g_{JK}$, and likewise for the other terms. Then
\begin{align*}
2\mf g_{IL} \Gamma^L_{JK}\p_\beta u^K  \tilde \nu^{IJ} 
& = \mf g_{IJ,K}  \p_\beta u^K \tilde\nu^{IJ}+ \lp( \mf g_{IK,J} -\mf g_{JK,I}\rp ) \p_\beta u^K \tilde \nu^{IJ}\\
& = \mf g_{IJ,K} \p_\beta u^K  \tilde \nu^{IJ} + \mf g_{IK,J} \p_\beta u^K \tilde \nu^{IJ} -\mf g_{IK,J} \p_\beta u^{K} \tilde \nu^{JI}\\
& = \p_{x^\beta} \mf g_{IJ}\,  \tilde \nu^{IJ} + 2i \,\mf g_{IK,J} \p_{x^\beta} u^K \Im \tilde\nu^{IJ}
\end{align*}
where in the last line we used the fact that $\tilde \nu^{JI} = (\tilde \nu^{IJ})^*$
% and that $2 g_{IL} \Gamma^L_{JI} = g_{II,J}$, which follows by setting $K=I$ in \eqref{eq:Christoffelsymbols}. 
To conclude, it now suffices to use the first part of the lemma, recalling that $\Gamma^I_{JK}=\Gamma^I_{KJ}$.
\end{proof}

\begin{proof}[Proof of Theorem \ref{thm:wave-map}]
We first note that $(g,(u^I)_{I=1}^N, (f^I)_{I=1}^N)$ is a distributional solution of \eqref{eq:wave-map-system}: this follows at once from Proposition~\ref{prop:limit-equation-LinWaveOsc} and \eqref{eq:convergencenonlinearities}.
%, it suffices to show that $$\Gamma^I_{JK}(u_\e) g_\e^{-1}(\tp d u_\e^J,\d u_\e^K)\w \Gamma^I_{JK}(u) g^{-1}(\tp d u^J,\d u^K) \qquad \tp{ in } L^2_\tp{loc}\,.$$ 
%The terms with $\Gamma^{I}_{JK}$ pose no problem, since $u_\e$ converges uniformly and the functions $\Gamma^I_{JK}$ are continuous. We have
%\begin{align*}
%g_\e^{-1}(\tp d u_\e^J, \d u_\e^K)- g^{-1}(\tp d u^J, \d u^K) & =
%g_\e^{-1}(\tp d u_\e^J, \d u_\e^K)-g^{-1}(\tp d u_\e^J,\d u_\e^K)\\
%& \quad +g^{-1}(\tp d u_\e^J,\d u_\e^K) -g^{-1} (\tp d u^J,\d u^K).
%\end{align*}
%The first term is $(g_\e^{\a\beta}-g^{\a\beta})\p_\a u_\e^J\p_\beta u_\e^K$, which converges to zero in the sense of distributions, since $\p_\a u_\e^J\p_\beta u_\e^K$ is uniformly bounded in $L^2_\tp{loc}$. The second term converges to zero by the usual quadratic compensated compactness as in \cref{lemma:compensatedcompactnessWave}, since $\Box_{g} u_\e^I$ is compact in $H^{-1}_\tp{loc}$ by Hypothesis~\ref{hyp:main}(\ref{it:hypsols}).  Thus equation \eqref{eq:wave-map-system} is also valid in the limit.

Using the Localization Lemma, just as in the proof of Theorem \ref{thm:linearwave}, we find that
\begin{equation}
\label{eq:threemeasures}
\tilde \nu_{\a\beta}^{IJ} = \xi_\a\xi_\beta \nu^{IJ},
\qquad \tilde \lambda^{IJ}_{ 0}=\xi_0 \lambda^{IJ},
\qquad \tilde \sigma^{IJ}_{0}=\xi_0 \sigma^{IJ}\,,
\end{equation} 
for some Radon measures $\nu^{IJ}, \lambda^{IJ}, \sigma^{IJ}$ and for each $I$ and $J$. Likewise, the measures $\nu^{IJ}, \lambda^{IJ}, \sigma^{IJ}$ are supported on the zero mass shell of $g$, and hence
the measures $\nu\equiv \mf g_{IJ} \nu^{IJ}$,  $\lambda\equiv \mf g_{IJ} \lambda^{IJ}$ and $\sigma\equiv \mf g_{IJ} \sigma^{IJ}$ are also supported on the same set. 
Furthermore, $\nu^{IJ}$ and hence also $\nu$ are even, whereas $\sigma$ and $\lambda$ are odd.  By the uniform convergence of both $u_\e$ and $g_\e$, and using the polarization identity,
\begin{align*}
&\lim_{\e\to 0}\int \mathbb{L}_{\alpha\beta}[u_\e]Y^\alpha Y^\beta \d \tp{Vol}_{g_\e}- \int \mathbb{L}_{\alpha\beta}[u] Y^\alpha Y^\beta \d \tp{Vol}_{g}\\
\quad&=\lim_{\e\to 0}\int \mathfrak{g}_{IJ}(u) \lp[\p_\alpha u_\e^I\p_\beta u_\e^J-\p_\alpha u^I \p_\beta u^J\rp] Y^\alpha Y^\beta \d \tp{Vol}_{g}\\
&= \lim_{\e\to0} \int \mathfrak{g}_{IJ}(u) \p_\alpha (u_\e^I-u^I)\p_\beta (u_\e^J-u^J) Y^\alpha Y^\beta \d \tp{Vol}_{g} = \langle \nu^{IJ},\sqrt{|g|}\mathfrak{g}_{IJ}\xi_\alpha\xi_\beta Y^\alpha Y^\beta  \rangle
= \langle \nu,\sqrt{|g|}\xi_\alpha\xi_\beta Y^\alpha Y^\beta \rangle\,,
\end{align*}
for any test vector field $Y$. This proves part (\ref{it:thm-wave-map-limit-equation}).

It remains to prove the propagation property of $\nu$. We first note that \begin{equation}
\label{eq:propagationpropsec5}
\int_{S^*\mc M} \lp[g^{\alpha\beta}\xi_\alpha\p_{x^\beta}\tilde{a}-\frac12\p_{x^\gamma}g^{\alpha\beta}\xi_\alpha\xi_\beta(\p_{\xi_\gamma}\tilde{a})\rp] \d\nu^{IJ} =
- \int_{S^*\mc M} \tilde{a}\tp{\,d}(\Re \sigma^{IJ})\,.
\end{equation}
This is proved by repeating verbatim the arguments in the proof of Theorem \ref{thm:linearwave}\eqref{it:propagation}: the only difference is that we multiply the equation for $\Box_g(A u_\e^1)$ with $X(\overline{Au_\e^2})$ and the one for $\Box_g(A u_\e^2)$ with $X(Au_\e^1)$, c.f.\ \eqref{eq:energyidentitiesA}. It follows that $\nu^{IJ}$ satisfies the equation
\begin{equation}
\label{eq:auxpropagationprop}
\int_{S^*\mc M} \lp[g^{\alpha\beta}\xi_\alpha\p_{x^\beta}(\tilde{a}\,\mf g_{IJ} )-\frac12\p_{x^\gamma}g^{\alpha\beta}\xi_\alpha\xi_\beta(\p_{\xi_\gamma}\tilde{a}) \mf g_{IJ}\rp] \d\nu^{IJ} =
- \int_{S^*\mc M} \tilde{a}\, \mf g_{IJ} \tp{\,d}(\Re \sigma^{IJ})\,,
\end{equation}
which is obtained from \eqref{eq:propagationpropsec5} by replacing $a$ with $a\, \mf g_{IJ}$.
Setting $\sigma\equiv \mf g_{IJ} \sigma^{IJ}$, we have
\begin{equation*}
\begin{split}
\int_{S^*\mc M} \lp[g^{\alpha\beta}\xi_\alpha\p_{x^\beta}\tilde{a}
- \frac12\p_{x^\gamma}g^{\alpha\beta}\xi_\alpha\xi_\beta\p_{\xi_\gamma}\tilde{a}\rp] \d\nu 
= -\langle g^{\a\beta}\xi_\a  \p_{x^\beta} \mf g_{IJ}\, \nu^{IJ},\tilde{a}\rangle
- \langle \Re \sigma,\tilde{a}\rangle\,.
\end{split}
\end{equation*}
Here, repeating the arguments in Section~\ref{sec:oscillating-wave-elementary-reductions}, by the parity of the measures involved, it is clear that we need only consider $\tilde{a}(x,\xi)$ to be odd in $\xi$, or equivalently, to let $A$ in Lemmas~\ref{lemma:rhssection4} and \ref{lemma:contributionsRHSwavemap} satisfy \eqref{eq:hypsymbol}. Then, Lemma~\ref{lemma:rhssection4} shows that no contribution to $\Re\sigma$ is made by the metric oscillations, $H_\e$. The contributions from $Q_\e$ are non-trivial, as shown in the last part of  Lemma \ref{lemma:contributionsRHSwavemap}. Since the second term in the right-hand side of \eqref{eq:contributionHepsilon} is imaginary, it follows from \eqref{eq:threemeasures}  that
$$\langle \Re \sigma,a\rangle = -\langle g^{\a\beta}\xi_\a  \p_{x^\beta} \mf g_{IJ}\, \nu^{IJ},a\rangle + \langle\Re \lambda,a\rangle\,.$$
Combining the previous two computations yields the result.
\end{proof}

\subsection{Proof of Theorem~\ref{thm:luk-huneau}}
\label{sec:proof-Burnett}

Theorem \ref{thm:luk-huneau} follows as a simple application of Theorem \ref{thm:wave-map}:

\begin{proof}[Proof of Theorem~\ref{thm:luk-huneau}] Setting $N=2$,  and labeling $u^1=\psi, u^2=\omega$, \eqref{eq:wave-maps-from-Einstein} is a wave map system from $\mc M\times \R$ into the Poincaré plane, equipped with metric $\mathfrak{g}$ and with Christoffel symbols $\Gamma_{JK}^I$ as follows:
\begin{align*} 
\mathfrak{g}=2\d\psi\otimes \d\psi +\frac12 e^{-4\psi}\d\omega\otimes \d\omega\,, \qquad
\Gamma_{JK}^1=\begin{dcases}
\frac{1}{2}e^{-4\psi}, &J=K=2,\\
0, & \text{otherwise},
\end{dcases}
 \qquad \Gamma_{JK}^2 = \begin{dcases}
-2, & J\neq K,\\
0, & \text{otherwise}.
\end{dcases}
\end{align*}
Following the notation of Theorem~\ref{thm:wave-map}, we set $\mathbb{L}_{\alpha\beta}[\psi,\omega]\equiv 2\p_\alpha\psi \p_\beta\psi+\frac14 e^{-4\psi}\p_\alpha \omega\p_\beta\omega$.

Let $\tilde{\nu}^{II}$ and $\tilde{\lambda}^{II}$ be as in Theorem \ref{thm:wave-map}. Further introduce the Radon measure
$$\tilde{\mu}\equiv 
2\begin{bmatrix}
\tilde \nu^{11} &  \tilde{\lambda}^{11} \\ 
(\tilde\lambda^{11})^* &  \star
\end{bmatrix}
+\frac12e^{-4\psi}
\begin{bmatrix} \tilde \nu^{22} &  \tilde \lambda^{22} 
\\ (\tilde \lambda^{22})^* &  \star
\end{bmatrix}
\equiv 
\begin{bmatrix}
 \tilde \nu &  \tilde\lambda \\ \tilde\lambda^* &  \star
\end{bmatrix}\,. $$
By Theorem~\ref{thm:wave-map}, we have $\tilde \nu_{\alpha\beta}=\xi_\alpha\xi_\beta\nu$ and $\tilde{\lambda}_\gamma = \xi_\gamma \lambda$, where
\begin{align*}
\nu\equiv 2\nu^1+\frac12e^{-4\psi}\nu^2\,,\qquad 
\lambda\equiv 2\lambda^1+\frac12e^{-4\psi}\lambda^2
\end{align*}
satisfy the localization and propagation properties stated as Theorem~\ref{thm:luk-huneau}(\ref{it:thm-cor-luk-huneau-Vlasov}), respectively.

Let us now turn to  the proof of Theorem~\ref{thm:luk-huneau}(\ref{it:thm-cor-luk-huneau-T}). In Section~\ref{sec:quasilinear}, we have already justified that $\mathbf{Ric}(\bm{g})_{\a3}=\mathbf{Ric}(\bm{g})_{33}=0$. For the $(\alpha,\beta)$ direction, the result follows from \eqref{eq:Ricci-coefficient-alphabeta-simpler} and Theorem~\ref{thm:wave-map}\eqref{it:thm-wave-map-limit-equation}.
\end{proof}

\appendix
\section{On elliptic gauge conditions}
\label{sec:appendix}

In this appendix we collect some standard facts concerning elliptic gauge. In fact, the results that we require hold in the more general setting of three-dimensional spacetimes allowing a constant mean curvature spacelike folliation, see Definition \ref{def:CMC} below.

Let $\mc{M}\subseteq \mathbb{R}^{1+n}$ be a smooth manifold, covered by global coordinates $(t\equiv x^0,x^1,\dots,x^n)$. As before, we take greek indices to range in  $\{0,1,\dots, n\}$ and roman indices to range in $\{1,\dots,n\}$, and assume the Einstein summation convention. Let $\mc{M}$ be equipped with a Lorentzian metric $g=g_{\alpha\beta}\d x^\alpha \d x^\beta$, with inverse $g^{-1}=g^{\alpha\beta}\p_{x^\alpha} \p_{x^\beta}$. It will be convenient to consider the Cauchy frame defined in \eqref{eq:Cauchyframe}:
\begin{align*}
\beta^i\equiv -\frac{g^{0i}}{g^{00}}\,, \qquad e_0\equiv \p_0-\beta^i\p_i\,, \qquad \tilde{g}^{ij}\equiv g^{ij}-\frac{g^{0i}g^{0j}}{g^{00}}\,.
\end{align*}
Define $N$ via $g^{00}=-N^{-2}$ and $\tilde{g}_{ij}$ as the inverse of the Riemannian metric $\tilde{g}^{ij}$. Then, $g$ may be written as 
\begin{align} \label{eq:elliptic-gauge-metric}
g=-N^2(\d x^0)^2+\tilde{g}_{ij}(\d x^i+\beta^i\d x^0)(\d x^j+\beta^j \d x^0)\,.
\end{align}
Note that $g_{ij}=\tilde{g}_{ij}$. The second fundamental form associated to constant $x^0$ hypersurfaces is
\begin{align*}
K_{ij}\equiv -\frac{1}{2N} \mc{L}_{e_0} \tilde{g}_{ij}=-\frac{1}{2N}\lp[e_0\tilde{g}_{ij} -\p_j \beta^k\tilde{g}_{ki}-\p_i \beta^k\tilde{g}_{kj}\rp]=H_{ij}+\frac{1}{n}\tilde{g}_{ij} \tau\,,
\end{align*}
where $\tau$ and $H_{ij}$ are, respectively, the trace and the traceless part of $K_{ij}$.  We easily compute: 
\begin{align}
\tau &\equiv\tilde{g}^{ij}K_{ij}= -\frac{1}{2N}\lp[\tilde{g}^{ij}e_0\tilde{g}_{ij} -2\p_k \beta^k\rp]\,,\label{eq:elliptic-gauge-tau}\\
H_{ij} &\equiv K_{ij}-\frac{1}{n}\tilde{g}_{ij}\tau =\frac{1}{2N}\lp[\p_j \beta^k\tilde{g}_{ki}+\p_i \beta^k\tilde{g}_{kj}-\frac{2}{n}\p_k\beta^k \tilde{g}_{ij}\rp]\label{eq:elliptic-gauge-H}\,.
\end{align}

Let $R_{\alpha\beta}$ denote the Ricci tensor components of $g$. We denote by $\tilde{D}^i$ and $\tilde{R}_{ij}$ the covariant derivative and Ricci tensor components, respectively, of the Riemannian metric $\tilde{g}_{ij}$. We have, see \cite[Chapter VI.3]{Choquet-Bruhat2008},
\begin{numcases}{}
\tilde{D}^j \p_j \beta_i= -2R_{0i}+2\tilde{D}^k N H_{ik}-\tilde{D}_k\p_i\beta^k+\frac{2}{n}\tilde{D}_i\p_k \beta^k +2N(1-1/n)\p_i \tau\,, \label{eq:elliptic-gauge-laplace-beta}\\
\triangle_{\tilde{g}}N = \frac{R_{00}}{N}+N|H|^2-e_0\tau +N\frac{\tau^2}{n}\,, \label{eq:elliptic-gauge-laplace-N}\\
\tilde{g}^{ij}\tilde{R}_{ij}= 2N^{-2}\lp(R_{00}-\frac12 g_{00} g^{\alpha\beta}R_{\alpha\beta}\rp) +|H|^2+\tau^2(1-1/n)\,, \label{eq:elliptic-gauge-laplace-gamma}
\end{numcases}
with indices raised and lowered through $\tilde{g}^{ij}$ and $\tilde{g}_{ij}$, respectively. Furthermore,
%\begin{empheq}[left={\empheqlbrace}]{align}
%& \tilde{D}^kH_{ik} = -\frac{R_{0i}}{N}+(1-1/n)\p_j \tau\,, 
%\label{eq:elliptic-gauge-DH}\\
%       &R_{ij}=\tilde{R}_{ij} -\frac{1}{N}\lp(\tilde{D}_i\p_jN+ e_0H_{ij}-\p_j\beta^kH_{ik}-\p_i\beta^kH_{jk}\rp)-2H_{il}H_j^l +\lp(K_{ij}-\frac{2}{n}H_{ij}\rp)\tau-\frac{\tilde{g}_{ij}e_0\tau}{nN}\,.
%\label{eq:elliptic-gauge-Rij}
%\end{empheq}
\begin{equation}
R_{ij}=\tilde{R}_{ij} -\frac{1}{N}\lp(\tilde{D}_i\p_jN+ e_0H_{ij}-\p_j\beta^kH_{ik}-\p_i\beta^kH_{jk}\rp)-2H_{il}H_j^l +\lp(K_{ij}-\frac{2}{n}H_{ij}\rp)\tau-\frac{\tilde{g}_{ij}e_0\tau}{nN}\,.
\label{eq:elliptic-gauge-Rij}
\end{equation}

\begin{definition}\label{def:CMC}
 We say $g$ of the form \eqref{eq:elliptic-gauge-metric}  
\begin{enumerate}
\item is \textit{spatially conformally flat} if the metric induced on constant $x^0$ hypersurfaces, $\tilde{g}_{ij}$,  satisfies  $\tilde{g}_{ij}=e^{2\gamma}\delta_{ij}$ for some conformal factor $\gamma$ defined on $\mc{M}$.
\item has a \textit{constant mean curvature} spacelike folliation if $\tau$ is constant on each $\{x^0=\tp{constant}\}$ slice, and we let the constant be either $\tau=x^0$ or $\tau=0$. In the latter case, we say that the folliation is \textit{maximal}.
\end{enumerate}
\end{definition}

\begin{remark}[Elliptic gauge] For $n=2$, Ricci-flat manifolds are locally conformally flat. Hence, the condition that $g$ (globally) takes the form \eqref{eq:elliptic-gauge-metric}  and that $\tilde{g}^{ij}$  is (globally) conformally flat is sometimes referred to as an \textit{elliptic gauge} condition for Ricci-flat $g$ in $1+2$ dimensions, see e.g.\ \cite{Huneau2019}.  
That $g$ have a constant mean curvature spacelike folliation is also sometimes referred to as the \textit{CMC gauge condition}, though, strictly speaking, it is a \textit{geometric condition} which is propagated by the vacuum Einstein equations in evolution \cite{Andersson2003}, see also \cite[Footnote 3]{Huneau2018a}.
\end{remark}

The following lemma justifies Hypotheses~\ref{hyp:main}\eqref{it:hypmetrics-Einstein} as well as Lemma \ref{lemma:ricconv}:
\begin{lemma} \label{lemma:appendix} Let $(g_\e)_\e$ be a sequence of metrics on $\mc{M}$ of the form \eqref{eq:elliptic-gauge-metric} which are spatially conformally flat with a {\normalfont fixed} constant mean curvature time folliation. Suppose, further, that they satisfy
\begin{itemize}
\item $g_\e\to g$ in $C^0_\tp{loc}$, for some Lorentzian metric $g$ which is also spatially conformally flat with the same constant mean curvature time folliation, and that $g_\e$ is uniformly bounded in $W^{1,2}_{\tp{loc}}$ entrywise;
\item the Ricci tensor of $g_\e$, $\mr{Ric}(g_\e)$, is uniformly bounded in $L^{2}_{\tp{loc}}$ entrywise.
\end{itemize}
Then, if $n= 2$, we also have
\begin{itemize}
\item $\delta^{ij} \p_{ij}^2 g^{\alpha\beta}_\e$ are bounded in $L^2_{\tp{loc}}$, and $e_0 \tilde{g}^{ij}_\e \to  e_0 \tilde{g}^{ij}$ strongly in $L^4_{\tp{loc}}$;
\item $\mr{Ric}(g_\e) \wstar \mr{Ric}(g)$ in the sense of distributions.
\end{itemize}
\end{lemma}
\begin{proof}  First notice that equations \eqref{eq:elliptic-gauge-laplace-gamma} and \eqref{eq:elliptic-gauge-Rij} take on a more familiar form in the spatially conformally flat case: letting $\tilde{g}_{ij}=e^{2\gamma}\delta_{ij}$ for some function $\gamma$, we have
\begin{align}
\begin{split}
-\tilde{R}_{ij}&= \delta_{ij}\triangle \gamma +(n-2)\p_{ij}\gamma +(n-2)[\delta_{ij}\delta^{kl}\p_k\gamma\p_l\gamma-\p_i\gamma\p_j\gamma]\,,\\
-e^{2\gamma}\tilde{g}^{ij}\tilde{R}_{ij}&= 2(n-1)\triangle \gamma +(n-2)(n-1)\delta^{ij}\p_i\gamma\p_j\gamma\,.
\end{split} \label{eq:elliptic-gauge-conformally-flat}
\end{align}

From the assumptions and from \eqref{eq:elliptic-gauge-H}, it is clear that the right hand side of (\ref{eq:elliptic-gauge-laplace-beta})--(\ref{eq:elliptic-gauge-laplace-gamma}) is bounded. For the metric component $N$, it is then immediate to see that its spatial laplacian is bounded in $L^2_\tp{loc}$. To draw the same conclusion for $\tilde{g}_{ij}$, we require the conformal flatness hypothesis, as per \eqref{eq:elliptic-gauge-conformally-flat}. Finally, for $\beta$, the same conclusion holds if $n=2$, since then the right hand side of \eqref{eq:elliptic-gauge-laplace-beta} has no second order terms:
\begin{align*}
-\tilde{D}_k\p_i\beta^k+\tilde{D}_i\p_k \beta^k  =\p^{k}\gamma\p_k\beta_i-\p_i\gamma\p^k\beta_k\,. %\label{eq:elliptic-gauge-2-dimensions}
\end{align*}

As the spatial laplacians of $N,\beta,\gamma$ are bounded in $L^2_\tp{loc}$, we conclude that all second spatial derivatives of the metric components are bounded in $L^2_\tp{loc}$: for any sufficiently regular function $f$ and any constant $c$,
\begin{align}\label{eq:elliptic-gauge-intermediate}
\begin{split}
\sup_{i,j} \lVert \p_{ij}^2f \rVert_{L^2_\tp{loc}(\{x^0=c\})}\lesssim  
\lVert \delta^{ij}\p_{ij}^2f \rVert_{L^2_\tp{loc}(\{x^0=c\})}
+\lVert f \rVert_{L^2_\tp{loc}(\{x^0=c\})}
 \,\\
 \implies  
\sup_{i,j} \lVert \p_{ij}^2f \rVert_{L^2_\tp{loc}}\lesssim  \lVert \delta^{ij}\p_{ij}^2f \rVert_{L^2_\tp{loc}}+\lVert f \rVert_{L^2_\tp{loc}}\,,
\end{split}
\end{align}
after integration in $x^0$. Similarly, a Gagliardo--Nirenberg interpolation inequality for fixed $x^0$ followed by integration in $x^0$ shows that $\p_j g^{\alpha\beta}_\e\to \p_j g^{\alpha\beta}$ strongly in $L^4_{\tp{loc}}$. Then, from \eqref{eq:elliptic-gauge-tau}, if $\tau\in W^{1,\infty}_\tp{loc}$ does not oscillate, then $e_0\tilde{g}^{ij}_\e$ converges strongly in $L^{4}_\tp{loc}$.

Finally, to show that $\mr{Ric}(g_\e)$  converges to $\mr{Ric}(g)$ in the sense of distributions, we must check that each nonlinear term in equations \eqref{eq:elliptic-gauge-laplace-beta}, \eqref{eq:elliptic-gauge-laplace-N} and \eqref{eq:elliptic-gauge-Rij} is the product of a strongly converging term in $L^2_\tp{loc}$ by a weakly converging term in $L^2_\tp{loc}$. Considering the previous remarks, this is certainly the case if we can show that $H_{ij}$ is in fact bounded in $W^{1,2}_{\tp{loc}}$. For the  $L^4_{\tp{loc}}$ boundedness of spatial derivatives, it suffices to differentiate \eqref{eq:elliptic-gauge-H} in a spatial direction and use \eqref{eq:elliptic-gauge-intermediate} with $f_\e=\beta^k_\e$. For $e_0H^{ij}_\e$, we appeal to \eqref{eq:elliptic-gauge-Rij}, noting that all other terms in this equation are bounded in $L^2_\tp{loc}$.
\end{proof}

%\total{citnum}
\bibliographystyle{abbrv-andre}
{\small \bibliography{../../../References/Mendeley/library}}

\end{document}